\documentclass[12pt,a4paper]{amsbook}

\usepackage{amsmath,amssymb,amsthm}
\usepackage{amscd}
\usepackage{diagrams}
\newcommand{\ity}{\infty}
\newcommand{\C}{\mathbb{C}}
\newcommand{\R}{\mathcal{R}}
\newcommand{\F}{F(\mathcal{R})}
\newcommand{\J}{J(\mathcal{R})}
\newcommand{\D}{\mathbb{D}}
\newcommand{\ti}{\widetilde}

\newcommand{\la}{\lambda}
\newcommand{\B}{\Big}
\newcommand{\A}{\mathcal{A}}
\usepackage{hyperref}

\numberwithin{equation}{section}
\numberwithin{section}{chapter}
\newtheorem{theorem}{Theorem}[section]
\newtheorem{lemma}[theorem]{Lemma}
\newtheorem{corollary}[theorem]{Corollary}
\newtheorem{proposition}[theorem]{Proposition}
\newtheorem{result}[theorem]{Result}
\theoremstyle{remark}
\newtheorem{remark}[theorem]{Remark}
\newtheorem{example}[theorem]{Example}
\newtheorem{definition}[theorem]{Definition}

\begin{document}
\frontmatter
\thispagestyle{empty}
\begin{center} \large
\textbf{FATOU'S NO WANDERING DOMAIN CONJECTURE : SULLIVAN'S PROOF}
\end{center}

\vfill

\begin{center}
Dissertation submitted to the University of Delhi in partial fulfilment of the requirements for the degree of
\end{center}

\vfill

\begin{center}\large
\textbf{MASTER OF PHILOSOPHY\\
 IN  \\
 MATHEMATICS }
\end{center}

\vfill

\vfill

\begin{center}
By\\ \large \textbf{DINESH KUMAR}
\end{center}
\vfill

\begin{center}
Under the Supervision of\\ \begin{large} \textbf{DR. SANJAY KUMAR} \end{large} \\
Department of Mathematics\\Deen Dayal Upadhyaya College\\University of Delhi, Delhi
\end{center}

\vfill

\begin{center}\bf
DEPARTMENT OF MATHEMATICS \\ UNIVERSITY OF DELHI\\ DELHI-110007\\
JANUARY, 2011
\end{center}


\newpage

\bigskip

\bigskip
\centerline{\textsc{Declaration}}
\bigskip

\bigskip
This dissertation entitled ``\textbf{Fatou's No Wandering  Domain Conjecture : Sullivan's Proof}'' contains a comprehensive study of the Dynamical Partition of the Riemann Sphere and a complementary analysis of the work done by Dennis Sullivan. This study has been carried out by me under the supervision of \textbf{Dr.~Sanjay~Kumar}, Associate Professor, Department of Mathematics, Deen Dayal Upadhyaya College, University of Delhi, Delhi, for the award of the degree of Master of Philosophy in Mathematics.

I hereby also declare that, to the best of my knowledge, the work included in this dissertation has not been submitted earlier, either in part or in full, to this or any other University/institution for the award of any degree or diploma.
 \bigskip
  \bigskip

\bigskip

\bigskip
\begin{flushright}
\textbf{(Dinesh Kumar)}
\end{flushright}

\bigskip
 \bigskip

\bigskip \bigskip

\bigskip \bigskip

\bigskip
\bigskip
\bigskip
\bigskip

\begin{tabular}{l}
  (Supervisor)\\
 \textbf{Dr. Sanjay Kumar}\\
  Department of Mathematics\\
  Deen Dayal Upadhyaya College\\
  (University of Delhi)\\
  Karam Pura, New Delhi-110015\\
\end{tabular}
\hfill
\begin{tabular}{l}
(Head of the Department)\\
\textbf{Prof. B. K. Dass}\\
Department of Mathematics\\
University of Delhi\\
Delhi-110007\\
\end{tabular}


\newpage

\centerline{\textbf{ Acknowledgment }}
\bigskip
\bigskip

I take this opportunity to express my deep sense of gratitude to my supervisor, \emph{Dr. Sanjay Kumar}, for his constant encouragement, cooperation and invaluable guidance in the successful accomplishment of this dissertation. I also express my gratitude to \emph{Prof. B. K. Dass}, Head,  Department of Mathematics, University of Delhi for providing necessary facilities and constant encouragement during the course of this study.

I also wish to extend my thanks to all the faculty members of the Department of Mathematics, University of Delhi for their help, guidance and motivation for the work presented in the dissertation. They have always been there for me whenever I needed support from them, providing me critical research insights and answering my questions with their valuable time. Their academic excellence has also been a great value to my dissertation.

I am also thankful to my friends and fellow research scholars (specially  Sumit Nagpal, Kuldeep Prakash, Sarika Goyal and  Rani Kumari) for their help and discussion during the course of my study. I am also thankful to M.M.Mishra, Assistant Professor, Hansraj college, for his valuable guidance in Latex.

I also wish to express my gratitude to the C.S.I.R for granting me the fellowship which was a great financial assistance in the completion of my M. Phil programme.
I am sincerely thankful to my parents for motivating me to do higher studies. I would also like to extend my gratitude to my brothers  and sisters for helping me in every possible way and encouraging me to achieve my long cherished goal.

Above all, I thank, The Almighty, for all his blessings bestowed upon me in completing this work successfully.

\tableofcontents

\mainmatter
\chapter{Introduction}\label{ch1}
\section{A Short History}
The subject known today as complex dynamics\,- is the study of iterations of analytic functions. It originated with the Newton-Raphson iteration method of approximation of roots. The foundation of the modern theory of complex dynamics, iteration of analytic functions, was laid by Fatou and Julia early in the twentieth century.\,They exploited the  work of Montel , which he developed around that time  on normal families, to investigate the iteration of rational maps of the Riemann sphere (extended complex plane) $\mathbb{\ti{C}}$ and found that these dynamical systems had an extremely rich orbit structure. If $\R$ is a non constant rational function, the Fatou set of $\R$, denoted by $\F$ is the maximal open subset of the extended complex plane $\mathbb{\ti{C}}$ on which $\{\R^n\}$ is equicontinous, and the Julia set of $\R$ denoted by $\J$ is the complement of $\F$ in the extended complex plane $\mathbb{\ti{C}}$.\\
Both $\F$ and $\J$ are completely invariant subsets of $\mathbb{\ti{C}}$ under the action of $\R$.\ The dynamics of $\R$ (iterative behavior)  on $\F$ is {\it{stable}} (predictable) and on $\J$ it is {\it{chaotic}} (highly unpredictable) . Fatou gave special attention to {\it{stable}} subset $\F$ , also known as Fatou set. The components of $\F$ are mapped to other components, and Fatou observed that these seemed to eventually fall into a periodic cycle of components (components do not wander forever). But he could not prove this fact, although he verified it for many examples. He however conjectured that rational maps have no wandering domains.

It soon became apparent that even the local dynamics of an analytic map was not well understood. It was not always possible to linearize the dynamics of a map near a fixed or periodic point, and remarkable examples to that effect were produced by H.Cremer. In the succeeding decades researchers in the subject turned to this {\it{linearization}} problem, and the more global aspects of the dynamics of rational maps were all but forgotten for about half a century. With the arrival of fast computers and the first pictures of the Mandelbrot set, interest in the subject began to be revived. But the real revolution in the subject came with the work of D. Sullivan in the early 1980's. He was the first to realize that the Fatou-Julia theory was strongly linked to the theory of Kleinian groups, and he established a dictionary between the two theories. Borrowing a fundamental technique first used by Ahlfors in the theory of Kleinian groups, Sullivan proved Fatou's long standing conjecture on no wandering domains, which states that each component of the Fatou set is eventually mapped into a periodic component.
\section{Overview of thesis }

 In Chapter two, we reviewed the existing theory of complex analytic dynamics on the Riemann sphere $\ti\C.$ The standard references for the material presented in this chapter are the textbooks by Beardon~\cite{beardon}, Carleson and Gamelin~\cite{carleson and gamelin} and Milnor~\cite{milnor}, as well as expository articles by Blanchard~\cite{blanchard} and Blanchard and Chiu~\cite{complex dynamics}. Section \ref{ch1,sec2} emphasised on rational maps with specific reference to extended complex plane, spherical and chordal metrics, normal and equicontinuous families and the important notion of conjugacy.
In Section \ref{ch1,sec3}, we used the rational map $\R$ to partition the sphere into two disjoint completely invariant sets, on one of which $\R$ is well behaved (the Fatou set, $\F$), on the other of which $\R$ has chaotic behaviour (the Julia set, $\J$). In Section \ref{ch1,sec4} we discussed the dynamics in a neighbourhood of a periodic point, and the global consequences of Montel's Theorem is highlighted in Section \ref{ch1,sec5}.
Section \ref{ch1,sec6} focussed on the behaviour of a rational map $\R$ on the Fatou set $\F$ and in Section \ref{ch1,sec7}, we classified the periodic components of $\R.$

In Chapter three entitled ``Making Sullivan Readable,'' we did a complementary analysis of \textbf{Sullivan's} paper~\cite{sullivan}. In Sections \ref{ch2,sec1} and \ref{ch2,sec2}, we reviewed the direct and inverse limit of spaces and the theory of Riemann surfaces. In rest of the sections, we explained the remarkable argument of Sullivan, using the theory of quasiconformal homeomorphisms. Sullivan's idea is explained as follows : If $U$ is a wandering stable region (Fatou component), then the regions $U,\R(U),\ldots,$ are pairwise disjoint. The possible complications near the boundary of $U$ is overcomed using Caratheodory's theory of ``prime ends,'' Section \ref{ch2,sec8}. By deforming the conformal structure of each domain $\R^n(U)$ (Section \ref{ch2,sec5}), we produced a deformation of the rational map $\R$ (by the \textbf{Measurable Riemann Mapping Theorem of Ahlfors-Bers}~\cite{riemann's mapping theorem for variable metrics}, Section \ref{ch2,sec9}). Finally in Section \ref{ch2,sec12}, we concluded that, this association sending conformal structures of $U\cup\R(U)\cup\ldots,$ to certain deformations of $\R$ is injective. However, this leads to a contradiction, for the space of conformal structures of  $U\cup\R(U)\cup\ldots,$ is infinite dimensional whereas the space of rational mappings of given $degree\;d,$ viz. $CP_0^{2d+1}$ is finite dimensional.\\
 In addition, the theories have been well illustrated with examples.\\
Appendix gives a reflection of Topological viewpoint on Complex analysis.
\chapter[Complex Analytic Dynamics]{Dynamical Partition of the Riemann Sphere}\label{ch1}

\section{Introduction}\label{ch1,sec1}
We are interested in studying the dynamics of a discrete dynamical system of the Riemann Sphere $\mathbb{\ti{C}}$ generated by a holomorphic transformation $\R :\ti{\C}\rightarrow\ti{\C}$. In other words, our phase space will be the unique, simply connected, closed Riemann surface (see Section \ref{ch2,sec2}) $\ti{\C} = \C\cup \{\infty\}$ which is homeomorphic to the two dimensional sphere 
\begin{equation}
S^2 = \{(x_1,x_2,x_3)\in\mathbb{R}^3\; |\;{ x_1}{^2} + {x_2}{^2} + {x_3}{^2} = 1\}
\end{equation}
via ``stereographic projection''. We will use the variables $z$ and $w = \frac{1}{z}$ to represent the two standard coordinate charts on $\ti{\C}$ determined by the stereographic projection, where $w$ is a point at $\ity$. Any holomorphic map $\R :\ti{\C}\rightarrow\ti{\C}$ on the Riemann sphere can be expressed as a rational function i.e as the quotient $\R(z) = P(z)/Q(z)$ of two polynomials with complex coefficients and having no common roots. Hence there is a one-to-one correspondence between rational functions and holomorphic maps. The poles of the rational function are the points of $\ti{\C}$ which are mapped to infinity. Some excellent sources for most of the material dealing with  complex  dynamics  are the textbooks by Beardon~\cite{beardon}, Carleson and Gamelin~\cite{carleson and gamelin} and Milnor~\cite{milnor}, as well as expository articles by Blanchard~\cite{blanchard} and Blanchard and Chiu~\cite{complex dynamics}.


\section{Rational Maps}\label{ch1,sec2}

\subsection{The Extended Complex Plane}

The extended complex plane denoted by $\ti{\C}$ is obtained by adjoining a point $\ity$ to ${\C}.$ Thus  $\ti{\C}=\C\cup \{\infty\}$. It is homeomorphic to the two dimensional sphere 
\[
S^2 = \{(x_1,x_2,x_3)\in\mathbb{R}^3\; |\;{ x_1}{^2} + {x_2}{^2} + {x_3}{^2} = 1\}.
\]
 Let $\pi\; : \C\,\rightarrow\,S^{2} , \; z \rightarrow z^{\ast}, z^{\ast}=\frac{2x}{(|z|^2+1)}, \frac{2y}{(|z|^2+1)}, \frac{(|z|^2-1)}{(|z|^2+1)}, z=x+iy$ be the stereographic projection of $\C$ into the sphere $ S^2$.\;Define $\pi(\ity) = \xi\, (= (0,0,1)),$ then $\pi$ is a bijective map from $\ti{\C}$ to $S^2.\;$This bijection gives a metric $\sigma$ on $\ti{\C}$ defined as :
\begin{equation}
\begin{split}
\notag
\sigma(z,w) 
&= |\pi(z) - \pi(w)|\\
&= |z^{\ast} - w^{\ast}|\\
&= \dfrac{2|z -w|}{(1 + |z|^2)^{1/2}\,(1 + |w|^2)^{1/2}}\;\; \forall\; z,w \in \C.
\end{split}
\end{equation}
\begin{equation}
\Rightarrow\sigma(z,w) = \dfrac{2|z -w|}{(1 + |z|^2)^{1/2}\,(1 + |w|^2)^{1/2}}\;\; \forall\; z,w \in \C.
\end{equation}
If $w=\ity$ then
\begin{equation}
\begin{split}
\notag
 \sigma(z,\ity)
&= \lim_{w\rightarrow \ity} \sigma(z,w)\\
&= \dfrac{2}{(1 +|z|^2)^{1/2}}\cdot
\end{split}
\end{equation}
$\sigma$ is called the chordal metric on $\ti{\C}$  (since $\sigma(z,w)$ is the Euclidean length of the chord joining $z^{\ast}$ to  $w^{\ast}$). There is an alternative metric on $\ti{\C}$, viz. the spherical metric $\sigma_{0}$, which is equivalent to the chordal metric $\sigma$. The spherical distance $\sigma_{0}$ between $z$ and $w$ in $\ti{\C}$ is the Euclidean length of the shortest path on $S^2$ (an arc of great circle) between $z^{\ast}$ and $w^{\ast}$.\ If the chord joining $z^{\ast}$ and $w^{\ast}$ subtends an angle $\theta$ at the origin, then
\[
\sigma_{0}(z,w) = \theta,\;\sigma(z,w) = 2\,{\sin{\frac{\theta}{2}}},\; \Rightarrow\sigma(z,w) = 2\,{\sin}({\sigma_{0}(z,w)}/2).
\]
The following inequality connects the spherical and chordal metric :
\begin{equation}
\dfrac{2}{\pi}\,\sigma_{0}(z,w)\,\leqslant\,\sigma(z,w)\,\leqslant\,\sigma_{0}(z,w).
\end{equation}

\begin{definition}
The degree, $deg\, \R$ of any continuous map $\R : \ti\C \rightarrow \ti\C$ is a homotopy invariant (see Appendix \ref{A}), i.e if $\; \R\thickapprox \R'$ where $ \R' : \ti\C \rightarrow \ti\C$ is another continuous map, then $deg\,\R = deg\,\R'$. $deg\,\R$ measures how many times $\R$ wraps the sphere $S^2$ around itself.\,The degree $d$ of $\R = P/Q,$ where $P$ and $Q$ are coprime polynomials, is then equal to the maximum of degrees of $P$ and $Q$. Equivalently, the degree of $\R$ is the number (counted with multiplicity) of inverse images of any point of $\ti{\C}$.
\end{definition}

The Fatou - Julia theory applies to rational maps $\R$ whose degree is atleast two. A {\emph{dynamical system}} is formed by the repeated application (iteration) of the map $\R : \ti{\C}\rightarrow\ti{\C}$.\\

\textbf{Notation}: The symbol $\,\R^{n}$ denotes the $n$  fold composition $\R^{n} = \R\circ \R\circ \R\circ\cdots\circ \R$  of $\R$ with itself.


\begin{definition}
 Let $U$ be an open subset of $\ti{\C}$ and $\mathcal{F} =\{f_{i}\,|\,i \in I\}$ a family of meromorphic functions defined on $U$. The family $\mathcal{F}$ is a normal family if every sequence $(f_n)$ contains a subsequence$ (f_{n_{j}})$ which converges locally uniformly on $U,$ i.e uniformly on each compact subset of $U$ (in the spherical metric on $\ti{\C}$). 
\end{definition}
We illustrate with some examples.
\begin{example}
 Consider $f_{n}(z) = z^n , \; n = 1,2,\ldots, \;z\in \D$. It can be seen easily that  $(f_n)$ is a normal family.
\end{example}
\begin{example}
 Let $f(z)=az,\; |a|<1.$ Define $f_n(z)=f^n(z),$ the nth iterate of $f$. Then $(f_n)$ forms a normal family of functions on any domain $U$ in ${\C},$ since $(f_n)$ converges  locally uniformly on $U$ to the constant function $O$.
\end{example}
The family $(f_{i}\,|\,i \in\mathcal I)$ is not normal at $z_0,$ if it fails to be normal in every neighbourhood of $z_0$.
For a  complete discussion on normal families one can refer to Ahlfor's book~\cite{complex analysis}.
\begin{definition}
 Let $(X,d)$  be a metric space. A family of functions $\{f_{i}\,: X\rightarrow X\}$ is  equicontinuous at a point $x_0$ if for any $\epsilon > 0,\; \exists\; \delta > 0$ such that 
\[
 d(x,x_0) < \delta\; \Rightarrow d(f_{i}(x),f_{i}(x_0)) < \epsilon\quad\forall\;i.
\]
 The family  $\{f_{i}\,: X\rightarrow X\}$ is  equicontinuous on $X$ if it is equicontinuous at every point of $X$.
\end{definition}
We illustrate with some examples.
\begin{example}
 Let $\{f_n\}$ be a finite set of functions defined on a metric space $E$ and let $x_0\in E$. The continuity of $\{f_n\}$ at $x_0$ (respectively at $E$) implies the equicontinuity of $\{f_n\}$ at $x_0$ (respectively at $E$).
\end{example}
\begin{example}
 Consider $f_{n}(z) = z^{2n},\;  n = 1,2,\ldots ,\;z\in \D$. It can be seen easily that $(f_n)$ is an equicontinuous family.
\end{example}
We state a classical theorem due to \emph{Arzel$\grave{\text{a}}$-Ascoli} which relates normal families to equicontinuous ones.
\begin{theorem}
\textbf{(Arzel$\grave{\textbf{a}}$-Ascoli)}
 Let $U$ be an open subset of $\ti{\C}$ and $\mathcal{F} =\{f_{i}\,|\,i \in I\}$ a family of meromorphic functions defined on $U$. Then $\mathcal{F}$ is a normal family if and only if  it is an equicontinuous family on every compact subset of $U.$
\end{theorem}
We shall also need the following sufficient condition for normality.
\begin{theorem}\label{1.2.9}
 Let $U$ be an open subset of $\C$ and let $\mathcal{F} =\{f_i : U\to\C\}$ be a family of holomorphic functions. If the family $\mathcal{F} $ is locally uniformly bounded on $U$ i.e bounded on each compact subset of $U$, then it is a normal family.
\end{theorem}
The rational maps of degree one are the \textbf{M$\ddot{\text{o}}$bius maps}  $z\rightarrow \dfrac{az+b}{cz+d}\,,$ where the coefficients are complex numbers satisfying $\; ad-bc \neq 0$. These constitute the group of analytic homeomorphisms of $\ti{\C}$ onto itself.
\begin{definition}\label{d1}
\textbf{(Conjugacy)}
Two rational maps $R$ and $S$ are conjugate if and only if there is some M$\ddot{\text{o}}$bius map $g$ with $S=gRg^{-1}$. Clearly conjugacy is an equivalence relation.\\
We observe that, if $R\sim S$, then $R^n\sim S^n$.\\
Since $R\sim S,$ therefore $\exists$ a M$\ddot{\text{o}}$bius map $g$ such that $S=gRg^{-1}.$ Then $S^2= S\circ S=(gRg^{-1})\circ(gRg^{-1})=gR(g^{-1}g)Rg^{-1}=gR^2g^{-1}\;$( $g^{-1}g=Identity).$\\
By induction we have $S^n=gR^ng^{-1}$ which implies that $S^n\sim R^n$.\\
 Hence $R^n$ and $S^n$ are holomorphically, the same dynamical system. It means that a problem concerning $R$ can be transferred to a (possibly simpler) problem concerning a conjugate $S$ of $R$ and then the problem can be solved easily in terms of $S$. $R\sim S$ is exhibited by the commuting diagram
\[
\begin{CD}
\ti{\C} @ >R>>  \ti{\C}\\
@VVMV @VVMV\\
\ti{\C} @>S>>\ti{\C}
\end{CD}
\] 
where M is a M$\ddot{\text{o}}$bius transformation.
\end{definition}
We illustrate this with an example.
\begin{example}
Consider a quadratic polynomial $p(z)=az^2+2bz+d, \;a\neq 0.\; p(z) $ can be conjugated to some quadratic poynomial $p_c(z)=z^2+c,$ for some c to be determined later. This can be easily seen.
 Let $M(z)=az+b, a\neq 0$ be a M$\ddot{\text{o}}$bius map. Then
\begin{equation}
\begin{split}
\notag
M\circ p(z)
&=M(az^2+2bz+d)\\
&=a(az^2+2bz+d)+b\\
&=a^2z^2+2abz+ad+b
\end{split}
\end{equation}
and
\begin{equation}
\begin{split}
\notag
p_c\circ M(z)
&=p_c(az+b)\\
&=(az+b)^2+c\\
&=a^2z^2+2abz+b^2+c
\end{split}
\end{equation}
so that
\begin{equation}
\begin{split}
\notag
M\circ p(z)
&=p_c\circ M(z)\\
\iff a^2z^2+2abz+ad+b
&=a^2z^2+2abz+b^2+c\\
\iff ad+b
&=b^2+c\\
\iff c
&=ad+b-b^2.
\end{split}
\end{equation}
Thus it is sufficient to study the class of quadratic polynomials of the form $z\rightarrow z^2+c$ in order to understand the dynamics of all quadratic polynomials.
\end{example}

\section{The Fatou and Julia Sets}\label{ch1,sec3}

\emph{Fatou} and \emph{Julia} studied the iteration of rational maps $\R : \ti{\C}\,\rightarrow \ti{\C}$ under the assumption that $deg\,\R \geqslant 2$.
\begin{definition} A point $z_0 \in \ti{\C}$ is an element of the Fatou set $\F$ of $\R$ if $\exists$ a neighbourhood $U$ of $z_0$ in $\ti{\C}$  such that the family of iterates $ \{{\R^n}\vert_{U}\}$ is a normal family. The Fatou set of $\R$ is the maximal open subset of $\ti\C$ on which the family of iterates $\{\R^n\}$ is a normal family. The Julia set $\J$ is the complement of the Fatou set.
\end{definition}
\begin{definition}
A point $z_0$ is called a fixed point of a rational map ${\R}\;\text{if}\; {\R}(z_0)=z_0.$ The fixed points of a rational map ${\R}$ are the solutions of the equation $ P(z)-zQ(z)=0$ where ${\R}=P/Q,\; P$ and $Q$ being coprime polynomials.
\end{definition}
We illustrate with examples.
\begin{example}
 Consider $R(z)=z+z^2.$ Then
\begin{equation}
\begin{split}
\notag
R(z)
&=z\\
\Rightarrow z+z^2
&=z\\
\Rightarrow z^2
&=0\\
\Rightarrow z
&=0.
\end{split}
\end{equation}
Also $R({\ity})={\ity}.$ Therefore, the fixed points are $0$ and ${\ity}.$
\end{example}
\begin{example}
 Consider $R(z)=\dfrac{2z^2+z}{2+z^2}\cdot$ Then
\begin{equation}
\begin{split}
\notag
R(z)
&=z\\
\Rightarrow\dfrac{2z^2+z}{2+z^2}
&=z.
\end{split}
\end{equation}
 On solving we get $z=0,1.$ Therefore, the fixed points are $0$ and $1.$
\end{example}
\begin{example}
Consider $R(z)=\dfrac{1+z^2}{2z}\cdot\;$ Then $R(z)=z,\;\Rightarrow \dfrac{1+z^2}{2z}=z.$  On solving we get $z=1,-1.$ Also $R({\ity})={\ity}.$ Therefore, the fixed points are $1,-1,{\ity}.$
\end{example}
A rational map and its conjugate have the same number of fixed points, i.e if $z_0$ is a fixed point of a rational map $R$ and $g $ is a M$\ddot{\text{o}}$bius map, then $gRg^{-1}$ has the same number of fixed points at $g(z_0)$ as $R$ has at $z_0$.\\
Using the above invariance we can count the number of fixed points at ${\ity}. $\\  Let $R(\ity)=\ity$. Suppose $g$ is a M$\ddot{\text{o}}$bius map such that $g(\ity)=z_0$, where $z_0$ is some finite point. Then
\begin{equation}
\begin{split}
\notag
gRg^{-1}(z_0)
&=gR(\ity)\quad( g\; \text{is one-one})\\
&=g(\ity)\\
&=z_0.
\end{split}
\end{equation}
Thus $\ity$ is transferred to some finite fixed  point $z_0$ of $gRg^{-1}$ and the number of fixed points of the conjugate function at $g(\ity)=z_0$  gives the  number of fixed points of $R$ at $\ity.$\\
 We now find the number of fixed  points of a rational map of a given degree.
\begin{theorem} 
A rational map of degree $d$ has precisely $d+1$ fixed points in $\ti{\C}$.
\end{theorem}
\begin{proof}
 We know that any rational map $R$ is conjugate to a rational map $S$ which does not fix $\ity$, and the number of fixed points of $R$ and $S$ are the same. Thus we may assume that $R$ does not fix $\ity$, otherwise we use $S$.\\ Let $R=P/Q$, where $P$ and $Q$ are coprime polynomials. Let $z_0$ be any fixed point of $R$. As $Q(z_0)\neq 0$, the number of zeros of $R(z)-z$ at $z_0$ is the same as the number of zeros of $P(z)-zQ(z)$ at $z_0.$  Hence the number of fixed points of $R$ is the number of  solutions of the equation $P(z)-zQ(z)=0$ in ${\C}.\;$ Since $ R$ does not fix $\ity,$ therefore $deg\,P\leqslant deg\,Q=deg\,R.$ Thus the degree of $P(z)-zQ(z)$ is exactly $deg\,R+1=d+1$ and this completes the proof.
\end{proof}
\begin{definition} The {\it{forward orbit}} of $z_0$, denoted by $O^{+}(z_0)$, is the sequence $(z_n)$ given inductively by the relation $z_{n+1} = \R(z_n)$. If $z_n = z_0,$ for some $n$, i.e $\R^{n}(z_0) = z_0$ for some $n$, then $z_0$ is a periodic point and $O^{+}(z_0)$ is a periodic orbit. If $n$ is the first natural number such that $z_n = z_0$, then $n$ is called the period of the orbit. If $n = 1$, then $z_0$ is called a fixed point.
\end{definition}
\begin{definition}\label{d2}
\textbf{(Multiplier)}  To each fixed point $z_0$ of a rational map $\R$, we assign a complex number ${\la}$ called the multiplier $m(\R,z_0$) of $\R$ at $z_0$. It is equal to $\R'(z_{0})$ if $z_{0}\neq \ity$ and we shall see that it equals $\dfrac{1}{\R'(z_{0})}$ if $z_{0} = \ity$. If $z_0$ is a periodic point  of period n, then it is a fixed point of $\R^n$ and therefore $m(\R,z_0)=(\R^n)'(z_0).$  
\end{definition}
We illustrate with examples.
\begin{example}
 For the rational function $R(z)=\dfrac{2z^2+z}{2+z^2}\,,$ we had seen that $0$ and $1$ are the fixed points. Now $R'(z)=\dfrac{(2+z^2)(4z+1)-(2z^2+z)(2z)}{(2+z^2)^2}\cdot$\\ Therefore ${\la}$ at $0$  equals $\dfrac{1}{2}$ and ${\la}$ at $1$ equals $1.$
\end{example}
\begin{example}
 Consider $R(z)=z+z^2.$ Then
\begin{equation}
\begin{split}
\notag
 R(z)
&=z\\
\Rightarrow z+z^2
&=z\\
\Rightarrow z^2
&=0\\
\Rightarrow z
&=0.
\end{split}
\end{equation}
 Also $R(\ity)=\ity.$ So the fixed points of $R$ are $0$ and $\ity$. $R'(z)=1+2z.$ So ${\la}$ at $0$ equals $1$ and ${\la}$ at $\ity$ equals $0.$
\end{example}

Now we compute the multiplier at $\ity$ of a rational map $R$.\\
\textbf{\underline{Multiplier at $\ity$ }}\\
Suppose $R(z)=\dfrac{a_0+a_1z+...     +a_nz^n}{b_0+b_1z+....     +b_mz^m}$\,, where $a_nb_m\neq 0$ and $n>m.$ Clearly $R$ fixes $\ity$. $m(R,\ity)$ is the derivative of $S : z\rightarrow \dfrac{1}{R(1/z)}$ at $0$.\\
\begin{equation}
\begin{split}
\notag
S(z)
&=\dfrac{b_0+\dfrac{b_1}{z}+....    +\dfrac{b_m}{z^m}}{a_0+\dfrac{a_1}{z}+....    +\dfrac{a_n}{z^n}}\\
&=(b_0+\dfrac{b_1}{z}+....    +\dfrac{b_m}{z^m})(a_0+\dfrac{a_1}{z}+....    +\dfrac{a_n}{z^n})^{-1}.
\end{split}
\end{equation}
\begin{equation}
\begin{split}
\notag
 S'(z)
&=\dfrac{(\dfrac{-b_1}{z^2}+\cdots +\dfrac{-mb_m}{z^{m+1}})}{(a_0+\dfrac{a_1}{z}+\cdots   +\dfrac{a_n}{z^n})}+\dfrac{(b_0+\dfrac{b_1}{z}+\cdots   +\dfrac{b_m}{z^m})}{(a_0+\dfrac{a_1}{z}+\cdots +\dfrac{a_n}{z^n})^2}(-1)(\dfrac{-a_1}{z^2}+\cdots +\dfrac{-na_n}{z^{n+1}})
\end{split}
\end{equation}
\[=\dfrac{1}{z^{m+1}}\B(\dfrac{-b_1z^{m-1}+\cdots  -mb_m}{\dfrac{1}{z^n}(a_0z^n+\cdots +a_n)}\B)+\dfrac{1}{z^m}\B(\dfrac{(b_0z^m+\cdots +b_m)}{\dfrac{1}{z^{2n}}(a_0z^n+\cdots +a_n)^2}\B)\dfrac{1}{z^{n+1}}(a_1z^{n-1}+\cdots +   na_n)
\]
\[
=\dfrac{-mb_m}{a_n}+\dfrac{b_mna_n}{a_n^2},\quad\text{if}\; n=m+1,\;\;\text{and equals}\\
\;\,0,\quad\text{if}\; n>m+1.
\]
Thus 
 \begin{equation}
\notag
 S'(z)=
\begin{cases}
\dfrac{b_m}{a_n}    &\text{if $n=m+1$ \quad and}\\
0    &\text{if $n>m+1$}
\end{cases}
\end{equation}

Now
\begin{equation}
\begin{split}
\notag
R'(z)
&=\dfrac{a_1+\cdots  +na_nz^{n-1}}{b_0+b_1z+\cdots +b_mz^m}+\dfrac{(a_0+a_1z+\cdots   +a_nz^n)}{(b_0+b_1z+\cdots    +b_mz^m)^2}\,(-1)\,(b_1+\cdots   mb_mz^{m-1})\\
&=\dfrac{z^{n-1}}{z^m}\dfrac{({\dfrac{a_1}{z^{n-1}}+\cdots +  na_n})}{({\dfrac{b_0}{z^m}+\cdots +  b_m})}-z^n\dfrac{(\dfrac{a_0}{z^n}+\cdots +   a_n)}{{z^{2m}}{(\dfrac{b_0}{z^m}+\cdots  +b_m})^{2}}\,z^{m-1}\,(\dfrac{b_1}{z^{m-1}}+\cdots   +mb_m).
\end{split}
\end{equation}


Then by continuity
\[
      R'(\ity)=\lim_{z\rightarrow\ity}R'(z)
\]
$$=\lim_{z\rightarrow\ity}\B(\dfrac{z^{n-1}}{z^m}\dfrac{({\dfrac{a_1}{z^{n-1}}+\cdots +   na_n})}{({\dfrac{b_0}{z^m}+\cdots +   b_m})}-z^n\dfrac{(\dfrac{a_0}{z^n}+\cdots +    a_n)}{{z^{2m}}{(\dfrac{b_0}{z^m}+\cdots +b_m})^{2}}\,z^{m-1}\,(\dfrac{b_1}{z^{m-1}}+\cdots   +mb_m)\B)$$

\begin{equation}
\notag
=
\begin{cases}
\dfrac{na_n}{b_m}\,\dfrac{-a_n\,mb_m}{b_m^2}    &\text{if $ n=m+1$ \quad and}\\
\ity    &\text{if $n>m+1$}.
\end{cases}
\end{equation}
Since
\begin{equation}
\begin{split}
\notag
m(R,\ity)
&=S'(0)\\
&=\dfrac{1}{R'(\ity)}
\end{split}
\end{equation}
 we get that the multiplier $m(R,z_0)$ of a rational map $R$ at a fixed point $z_0$ is given by 
\begin{equation}
\notag
m(R,z_0)=
\begin{cases}
R'(z_0)\,    &\text{if $z_0\neq \ity$ \quad and}\\
\dfrac{1}{R'(\ity)}\,    &\text{if $z_0=\ity$}.
\end{cases}
\end{equation}
\begin{definition}\label{d3}
\textbf{(Critical Points)}
A point $z_0$ is a critical point of a rational map $\R$ if $\R$ is not locally injective at $z_0$. The set of  critical points of $\R$ consists of solutions of the equation $\R'(z)=0$ and of poles of $\R$ of order atleast 2. The image of a critical point is called a critical value of $\R.$\\
  If $deg\,\R=p$\, and $w$ is not a critical value, then clearly $\R^{-1}(w)$ consists of $p$ distinct points, say $z_1,z_2,...,z_p.$ As the $z_j\text{'s}, \;1\leqslant j\leqslant p$ are not critical points, we can find neighbourhoods $U$ of $w$ and $U_1,\ldots ,U_p$ of $z_1,z_2,\ldots,z_p$ such that 
\[
\R : U_j\to U
\]
is a bijection. Thus for each $j$, the restriction $R_j$ of the map $\R$ to $U_j$ has an inverse
\[
R_j^{-1} : U\to U_j
\]
and are called the \emph{branches} of $\R^{-1}$ at $w$.
\end{definition}
\begin{example}
Consider $\R(z)=z^2.$ Then $\R'(z)=2z$ and so
\begin{equation}
\begin{split}
\notag
\R'(z)
&=0\\
\Rightarrow 2z
&=0\\
\Rightarrow z
&=0.
\end{split}
\end{equation} 
Thus $z=0$ is a critical point of $\R$. It is the only critical point of  $\R$ besides $\ity$.
\end{example}

\subsection{Classification of fixed points of $\R$}\label{s1}
We now give the classification for the fixed points of a rational map $\R$. \\
If $z_{0}$ is a fixed point of $\R$ and $\la= \R'(z_{0}),$ the multiplier of $\R$ at $z_0$, then
\begin{enumerate}
\item\ $z_0$ is an attracting fixed point if $0<|\la|<1$;
\item\ superattracting fixed point if $\la = 0$;
\item\ repelling fixed point if $|\la|>1$;
\item\ indifferent fixed point if $|\la| = 1$;
\begin{enumerate}
 \item\ Rationally indifferent fixed point if $\la$ is a root of unity;
 \item\  Irrationally indifferent fixed point if $ |\la| = 1$, but $\la$ is not a root of unity.
\end{enumerate}
\end{enumerate}
A superattracting fixed point is a critical point of $\R$ whereas an attracting fixed point is not.\ Thus $\R$ has a local inverse near $z_0$, if $z_0$ is attracting , but not if $z_0$ is superattracting.\\
We illustrate with examples.
\begin{example}
 For the rational function $R(z)=\dfrac{2z^2+z}{2+z^2}\,,$ we had seen that $0$ and $1$ are the fixed points.  $R'(z)=\dfrac{(2+z^2)(4z+1)-(2z^2+z)(2z)}{(2+z^2)^2}\cdot$ Now ${\la}$ at $0$ equals $\frac{1}{2}$ and ${\la}$ at $1$ equals $1.$ Thus $0$ is an attracting fixed point while $1$ is rationally indifferent fixed point.
\end{example}
\begin{example}
 Consider $R(z)=z+z^2.$ Then
\begin{equation}
\begin{split}
\notag
R(z)
&=z\\
\Rightarrow z+z^2
&=z\\
\Rightarrow z^2
&=0\\
\Rightarrow z
&=0.
\end{split} 
\end{equation}
Also $R(\ity)=\ity.$ So the fixed points of $R$ are $0$ and $\ity$. $R'(z)=1+2z.$ So ${\la}$ at $0$ equals $1$ and ${\la}$ at $\ity=0.$ Thus $0$ is a rationally indifferent fixed point whereas $\ity$ is superattracting fixed point.
\end{example}
\begin{example}
Consider $R(z)=z^2-\dfrac{z}{2}+\dfrac{1}{2}\,\cdot$ Then
\begin{equation}
\begin{split}
\notag
 R(z)
&=z\\
\Rightarrow z^2-\dfrac{z}{2}+\dfrac{1}{2}
&=z\\
\Rightarrow z^2-\dfrac{3z}{2}+\dfrac{1}{2}
&=0\\
\Rightarrow 2z^2-3z+1
&=0\\
\Rightarrow (z-1)(2z-1)
&=0\\
\Rightarrow z
&=1,\frac{1}{2}\cdot
\end{split}
\end{equation}
 Also $R(\ity)=\ity.\;$ Therefore $1,\dfrac{1}{2}$ and $\ity$ are the fixed points of $R$. $R'(z)=2z-\dfrac{1}{2}\cdot$
\begin{equation}
\begin{split}
\notag
m(R,1)
&=R'(1)\\
&=2-\dfrac{1}{2}\\
&=\dfrac{3}{2}>1.
\end{split}
\end{equation}

\begin{equation}
\begin{split}
\notag
m(R,\dfrac{1}{2})
&=R'(\dfrac{1}{2})\\
&=2\dfrac{1}{2}-\dfrac{1}{2}\\
&=\dfrac{1}{2}<1.
\end{split}
\end{equation}

\begin{equation}
\begin{split}
\notag
m(R,\ity)
&=\dfrac{1}{R'(\ity)}\\
&=0.
\end{split}
\end{equation}
Thus 1 is repelling, $\dfrac{1}{2}$ is attracting and $\ity$ is superattracting fixed point of $R$.
\end{example}
\begin{example}
Consider $R(z)=z^2-2.$ Then
\begin{equation}
\begin{split}
\notag
R(z)
&=z\\
\Rightarrow z^2-2
&=z\\
\Rightarrow  z^2-z-2
&=0\\
\Rightarrow z
&=-1,2.
\end{split}
\end{equation} 
Also $R(\ity)=\ity. $ Therefore  $-1,2 $ and $\ity$ are the  fixed points. Now  $R'(z)=2z.$
\begin{equation}
\begin{split}
\notag
m(R,-1)
&=R'(-1)\\
&=-2\\
\Rightarrow |\la|
&=2>1.
\end{split}
\end{equation}
\begin{equation}
\begin{split}
\notag
m(R,2)
&=R'(2)\\
&=4>1.
\end{split}
\end{equation}
\begin{equation}
\begin{split}
\notag
m(R,\ity)
&=\dfrac{1}{R'(\ity)}\\
&=0.
\end{split}
\end{equation}
Thus -1 and 2 are repelling fixed points whereas $\ity$ is a superattracting fixed point.
\end{example}
\begin{remark}\label{1.3.17}
 If $f$ is a polynomial with $deg\,f=d\geqslant 2,$ then $\ity$ is always a superattracting fixed point with multiplier $\la=0.$
\end{remark}
\begin{proof}
It follows from the definition of multiplier at $\ity.$
\end{proof}
\begin{remark}\label{1.3.18}
 If $R(\ity)=\ity,$ we can conjugate $R$ by a M$\ddot{\text{o}}$bius map so that $\ity$ is transferred to a finite fixed point say $z_0$.
\end{remark}
\begin{proof} Suppose $g$ is a M$\ddot{\text{o}}$bius map which sends $\ity$ to a finite  point say $z_0$. Then $g^{-1}(z_0)=\ity\;(g\;\text{is one-one}).$ Now
\begin{equation}
\begin{split}
\notag
 gRg^{-1}(z_0)
&=gR(\ity)\\
&=g(\ity)\\
&=z_0 \\
\Rightarrow gRg^{-1}(z_0)
&=z_0.
\end{split}
\end{equation}
Thus $z_0$ is a fixed point of the conjugate $gRg^{-1}$ of $R$.
\end{proof}
We state the following result which we will use later on.
\begin{theorem}\label{1.3.19}
Let $\R$ be a non-constant rational map, and $p\in\mathbb Z^+.$ Then $F(\R^p)=\F$ and $J(\R^p)=\J.$
\end{theorem}
\begin{definition}\label{d4}
\textbf{Completely Invariant (Fully Invariant) Sets :}\; Let $g : X \rightarrow X$ be a map. A subset $E$ of $X$ is
\begin{enumerate}
\item\ forward invariant if$\;  g(E) = E$;
\item\ backward invariant if $\; g^{-1}(E) = E$;
\item\ completely invariant if $g(E) = E = g^{-1}(E)$.
\end{enumerate}
Clearly if $g$ is surjective then backward invariance and complete invariance coincide, for suppose $g$ is  surjective and $E$ is backward invariant, then $g^{-1}(E)=E.$ Applying $g$ on both sides, $g( g^{-1}(E))=g(E)$ and as $g$ is surjective therefore $E=g(E).$ Thus $E$ is forward invariant and hence  is completely invariant.\\
We will soon see that the Fatou set and Julia set of any rational map are completely invariant sets. 
\end{definition}
\begin{corollary}\label{1.3.21}
If $g : X\rightarrow X$ is surjective and $E\subset X$ is completely invariant, then so is $E^c=X\smallsetminus E.$
\end{corollary} 
\begin{proof}
It is sufficient to show that E is backward invariant. We know from topology that $g^{-1}$ preserves differences of sets, therefore
\begin{equation}
\begin{split}
\notag
g^{-1}(X\smallsetminus E)
&=g^{-1}(X)\smallsetminus g^{-1}(E)\\
&=X\smallsetminus E\quad(g^{-1}(E)=E).
\end{split}
\end{equation}
 Thus $E^c=X\smallsetminus E$ is backward invariant and hence completely invariant.
\end{proof}
\begin{theorem}\label{1.3.22}
 Let $g : X\rightarrow X$ be surjective, continuous and open where X is a topological space and suppose that E is completely invariant. Then so are $ E^{\circ}, \partial E$ and $\overline{E}$ where $ E^{\circ}$ and $\overline{E}$ denotes interior E and closure E respectively.
\end{theorem}
\begin{proof}
\begin{enumerate}
\item\ $E^{\circ}\subset E$ is open and  $g$ is continuous, therefore $g^{-1}( E^{\circ})$ is an open subset of $g^{-1}(E)=E$ which implies that
\begin{equation}
 g^{-1}( E^{\circ})\subset E^\circ\quad (E^\circ\;\text {is the maximal open subset contained in E}).
\end{equation}
 Since $g$ is open, therefore $g(E^\circ)$ is an open subset of $g(E)=E$ and so $g(E^\circ)\subset E^\circ$ which implies that
\begin{equation}
 E^\circ\subset g^{-1}(E^\circ).
\end{equation}
 From the above two equations we get that $g^{-1}(E^\circ)=E^\circ.$ Hence $E^\circ$ is completely invariant.
\item\ Let $x\in \overline{E},\;\Rightarrow \exists$ a sequence $(x_n)\subset E$ such that
\begin{equation}
\begin{split}
\notag
x
&=\lim_{n\to\ity}x_n\\
\Rightarrow g(x)
&=g(\lim_{n\rightarrow{\ity}}x_n)\\
&=\lim_{n\rightarrow{\ity}}(g(x_n))\quad(g \text{ is continuous}).
\end{split}
\end{equation}
 Now
\begin{equation}
\begin{split}
\notag
 g(x_n)
&\in g (E)\\
&=E\;\;\forall\,n\\
\Rightarrow \lim_{n\rightarrow{\ity}}g(x_n)
&\in\overline{E}\\
\Rightarrow g(x)
&\in \overline{E}\\
\Rightarrow g(\overline{E})
&\subset\overline{E}\\
\Rightarrow \overline{E}
&\subset g^{-1}(\overline{E}).
\end{split}
\end{equation}
Conversely,
 we show that $g^{-1}(\overline{E})\subset \overline{E}$.\\ Let $x\notin \overline{E},\\\Rightarrow \forall$ sequence$(x_n)\subset E, $
\begin{equation}
\begin{split}
\notag
\lim_{n\rightarrow{\ity}}x_n
&\neq x\\
\Rightarrow g(\lim_{n\rightarrow{\ity}}x_n)
&\neq g(x)\\
\Rightarrow \lim_{n\rightarrow{\ity}}(g(x_n))
&\neq g(x) \quad(g\;\text {is continuous}).
\end{split}
\end{equation}
 Now 
\begin{equation}
\begin{split}
\notag
g(x_n)\in g(E)
&=E \;\;\forall\, n\\
\Rightarrow g(x)
&\notin \overline{E}\\
\Rightarrow x
&\notin g^{-1}(\overline{E})\\
\Rightarrow g^{-1}(\overline{E})
&\subset\overline{E}.
\end{split}
\end{equation}
 Hence we conclude that $ g^{-1}(\overline{E})=\overline{E}$ so that $\overline{E}$ is completely invariant.
\item\ Finally we show that $\partial E$ is completely invariant.\\ We know that $\partial E= \overline{E}\cap\overline{(X\smallsetminus E)}$.
\begin{equation}
\begin{split}
\notag
g^{-1}(\partial E)
&=g^{-1}(\overline{E}\cap\overline{(X\smallsetminus E)})\\
&=g^{-1}(\overline E)\cap g^{-1}(\overline{(X\smallsetminus E)})\quad( g^{-1}\; \text{preserves intersection})\\
&=\overline{E}\cap\overline{(X\smallsetminus E)}\quad\text{(by part (2))}\\
&=\partial E.
\end{split}
\end{equation}
\end{enumerate}
\end{proof}
\begin{theorem}\label{1.3.23}
 Let ${\R}$ be any rational map. Then ${\F}$ and ${\J}$ are completely invariant under ${\R}$.
\end{theorem}
\begin{proof}
 It suffices to prove that ${\F}$ is backwards invariant, as a rational map $\R :\ti\C\to\ti\C$ is always surjective.\\ Let $z_0\in {\R}^{-1}({\F}),\;\Rightarrow {\R}(z_0)\in {\F}$. Let $w_0={\R}(z_0).$ Since $\{\R^n\}$ is equicontinuous at $w_0,\;
\Rightarrow\text{for}\;\; \epsilon >0, \; \exists \;\delta >0 \;$ such that 
\begin{equation}\label{eq1}
\sigma_0(w,w_0)<\delta\;\Rightarrow \sigma_0({\R}^n(w),{\R}^n(w_0))<\epsilon\quad\forall\;n.
\end{equation}
 Since ${\R}$ is continuous at $z_0$, so for
 $\delta>0,\;\exists \;\delta_1 >0$ such that
\begin{equation}
\begin{split}
\notag
 \sigma_0(z,z_0)
&<\delta_1\\
\Rightarrow \sigma_{0}({\R}(z),{\R}(z_0)=w_0)
&<\delta\\
\Rightarrow \sigma_0({\R}^n({\R}(z)),{\R}^n({\R}(z_0)))
&<\epsilon\;\;\forall\,n\quad\text{(by (\ref{eq1}))}\\
\Rightarrow \sigma_0({\R}^{n+1}(z),{\R}^{n+1}(z_0))
&<\epsilon \;\;\forall\,n.
\end{split}
\end{equation}
$\Rightarrow \{\R^{n+1}:n\geqslant 1\}$ is equicontinuous at $z_0$. Adding one more function ${\R}$, we get $\{\R^{n}:n\geqslant 1\}$ is equicontinuous at $z_0. \;$ Since $z_0\in {\R}^{-1}(\F)$ was arbitrary, $\;$ therefore  $\{{\R}^{n}:n\geqslant 1\}$ is equicontinuous on ${\R}^{-1}(\F). \;$ Since ${\F}$ is open, ${\R}^{-1}({\F})$ is open, and ${\F}$ is the maximal open subset on which $ \{{\R}^{n}:n\geqslant 1\}$ is equicontinuous, hence
${\R}^{-1}({\F})\subset {\F}$.\\
Conversely,
we show that ${\F}\subset{\R}^{-1}({\F})$. \\Let $z_0\in {\F}$. Let ${\
R}(z_0)=w_0.\;$ Since $z_0\in {\F}$, so for
$\epsilon>0,\;\exists \;\delta>0$ such that
\begin{equation}\label{eq2}
 \sigma_0(z,z_0)<\delta\;\;\Rightarrow \sigma_0({\R}^{n+1}(z),{\R}^{n+1}(z_0))<\epsilon\quad\forall\; n.
\end{equation}
 Consider $$N=\{z\;|\;\sigma_0(z,z_0)<\delta\}.$$ N is an open neighbourhood of $z_0.\;$ Since ${\R}$ is an open map,  therefore ${\R}(N)$ is an open neighbourhood of ${\R}(z_0)=w_0$. If $w\in {\R}(N)$ then $w={\R}(z)$ for some $z\in N.$
\begin{equation}
\begin{split}
\notag
\sigma_0({\R}^{n}(w),{\R}^{n}(w_0))
&=\sigma_0({\R}^{n+1}(z),{\R}^{n+1}(z_0))\\
&<\epsilon\;\;\forall\,n\quad\text{(by (\ref{eq2}))}.
\end{split}
\end{equation}
$\Rightarrow \{\R^n\}$ is equicontinuous at $w_0,\; \Rightarrow w_0\in{\F},$ i.e ${\R}(z_0)\in {\F},$ i.e $ z_0\in {\R}^{-1}({\F}).\;$ As $z_0\in {\F}$ was arbitrary, 
therefore ${\F}\subset{\R}^{-1}({\F}).$\\
Hence we conclude that ${\R}^{-1}({\F})={\F}.$\\
By Corollary \ref{1.3.21}, complement of a completely invariant set is completely invariant, so that
  ${\J}$ is completely invariant.
\end{proof}
\begin{theorem}\label{1.3.24}
 Let $\R$ be a rational map with $deg\,\R\geqslant 2$ and suppose that a finite set $E$ is completely invariant under $\R$. Then $E$ has atmost two elements.
\end{theorem}
\begin{proof} 
Suppose $E$ has $k$ elements. Since $E$ is completely invariant, therefore $\R(E)=E$ and $\R^{-1}(E)=E.$ Since $E$ is a finite set, therefore $\R$ is a permutation of $E$. So $\;\exists\;$ some $n\in\mathbb{N}$ such that $\R^n=I$, the identity map on E. Let $deg\, \R^n=d$ on $\ti\C$. For each $w\in E$, the equation $\R^n(z)=w$ has $d$ solutions say $z_1,z_2,\ldots,z_d$ in $\ti\C$ out of which atmost $d-1$ can be critical points of $\R^n\;( \R^n\vert_ E=I,$ one of the solution say $z_j$ lies in $E$ and as $I$ is injective, $z_j$ is not a critical point). Since $E$ has $k$ elements, therefore  the number of critical points  can be atmost $k(d-1)$ and so $k(d-1)\leqslant 2d-2\,$ (since, the number of critical points of $\R^n$ is atmost $2d-2$ by Theorem \ref{1.5.21}, to be shown later).
$\;\Rightarrow d(k-2)\leqslant k-2$. If $k-2>0$ i.e $k>2,$ then $d\leqslant 1$, a contradiction to the fact that $d\geqslant 2.\;$ Therefore $ k-2\leqslant 0\;$ i.e $k\leqslant 2.$ Hence $E$ has atmost two elements.
\end{proof}



\section{Periodic points}\label{ch1,sec4}

\subsection{Classification of periodic points of $\R$}\label{s2}
If $z_{0}$ is a periodic point of $\R$ of period $n$ and $\la= (\R^n)'(z_{0})$ be the multiplier of $\R$ at $z_0$, then
\begin{enumerate}
\item\ $z_0$ is an attracting periodic point if $0<|\la|<1$;
\item\ superattracting periodic point if $\la = 0$;
\item\ repelling periodic point if $|\la|>1$;
\item\ neutral (indifferent) periodic point if $|\la| = 1;$
\begin{enumerate}
\item\ Rationally indifferent periodic point if $\la$ is a root of unity;
\item\ Irrationally indifferent periodic point if $|\la|=1$, but $\la$ is not a root of unity.
\end{enumerate}
\end{enumerate}
We illustrate with examples.
\begin{example}
Consider the map $ R(z)=z^2+c\;\;$\emph{(Douady's Rabbit)}\\
where $c$ satifies $c^3+2c^2+c+1=0$ and \emph{Im}$\,c>0.$\\
We shall see that 0 is a superattracting periodic point of period 3.\\
We have $R(0)=c$. Then $R^2(0)
=R(c)
=c^2+c.$
\begin{equation}
\begin{split}
\notag
R^3(0)
&=R(c^2+c)\\
&=(c^2+c)^2+c\\
&=c^4+c^2+2c^3+c\\
&=c(c^3+2c^2+c+1)\\
&=0\quad(c\neq 0).
\end{split}
\end{equation}
Hence the orbit of 0 is
\begin{equation}
0\to c\to c^2+c.
\end{equation}
$R(z)=z^2+c;\;R^2(z)=(z^2+c)^2+c;\;R^3(z)=((z^2+c)^2+c)^2+c.$\\
$(R^3)'(z)=2\,((z^2+c)^2+c)\,2\,(z^2+c)\,2z$ and so\\
$(R^3)'(0)=0.$\\
Hence $0$ is a superattracting periodic point of period $3.$
\end{example}
\begin{example}
 Consider $R(z)=z^2-1.$ The point $z_0= 0$ is a periodic point of period 2, for 
\begin{equation}
0\to -1\to 0\to -1\to\cdots
\end{equation}
 $\Rightarrow\; 0\to -1.$\\
Now 
$R^2(z)
=R(z^2-1)
=(z^2-1)^2-1$.\\
$(R^2)'(z)=4z(z^2-1).\;(R^2)'(0)=0,$ which implies that the multiplier $\la$ at $0$ equals $0.$ Thus $0$ is a superattracting periodic point of period 2.\\
Consider $z_0=-1.$
The orbit of $z_0=-1$ is
\begin{equation}
 -1\to 0\to -1\to 0\to\cdots
\end{equation}
$\Rightarrow \;-1\to 0.$ Thus -1 is a periodic point of period 2.\\ Now $(R^2)'(-1)
=4\,(-1)\,(0)
=0.$
Hence $-1$ is a superattracting period $2$ point.
\end{example}
\begin{example}
 Consider $R(z)=z^2-2$ and consider $z_0=0.$ The orbit of 0 is
\begin{equation}
0\to -2\to 2\to 2\to\cdots 
\end{equation}
 In this case $0$ is not a periodic point. 
\end{example}
\begin{example}
Consider $R(z)=\dfrac{1+z^2}{2z}\cdot$ Then 
\begin{equation}
\begin{split}
\notag
R(z)
&=z\\
\Rightarrow \dfrac{1+z^2}{2z}
&=z\\
\Rightarrow 1+z^2
&=2z^2\\
\Rightarrow z^2
&=1\\
\Rightarrow z
&=1,{-}1\text{ are the two fixed points.}
\end{split}
\end{equation}
The fixed points of $R^2$ are given by
\begin{equation}
\begin{split}
\notag
R^2(z)
&=z\\
\Rightarrow R(\dfrac{1+z^2}{2z})
&=z\\
\Rightarrow \dfrac{1+(\dfrac{1+z^2}{2z})^2}{2(\dfrac{1+z^2}{2z})}
&=z\\
\Rightarrow \dfrac{4z^2+(1+z^2)^2}{4z(1+z^2)}
&=z\\
\Rightarrow {4z^2+(1+z^2)^2}
&={4z^2(1+z^2)}\\
\Rightarrow 4z^2(1+z^2-1)
&=(1+z^2)^2\\
\Rightarrow 4z^4
&=1+z^4+2z^2\\
\Rightarrow 3z^4-2z^2-1
&=0\\
\Rightarrow 3z^4-3z^2+z^2-1
&=0\\
\Rightarrow (z^2-1)(3z^2+1)
&=0\\
\Rightarrow z
&=1,-1, \dfrac{i}{\sqrt 3},\dfrac{-i}{\sqrt 3}\cdot
\end{split}
\end{equation}
Thus $\dfrac{i}{\sqrt 3},\dfrac{-i}{\sqrt 3}$ are the two periodic points of period 2.
\end{example}
\begin{definition}
 A point $z_0$ is called eventually periodic if for some $n \in \mathbb{Z}^{+}, R^n(z_0)$ is a periodic point.\\
A point $z_0$ is called pre-periodic if it is eventually periodic but not periodic.
\end{definition}
We illustrate with examples.
\begin{example}
Consider $$f : [0,1]\to [0,1]$$ defined as
\begin{equation}
\notag
f(x)=
\begin{cases}
2x &\text{if $0\leqslant x \leqslant \dfrac{1}{2}$ \quad and}\\
-2x+2 &\text{if $\dfrac{1}{2}\leqslant x \leqslant 1$}
\end{cases}
\end{equation}
We claim that $\;\dfrac{3}{8}$ is pre-periodic.\\
Since $\;\dfrac{3}{8}\to\dfrac{3}{4}\to\dfrac{1}{2}\to 1\to 0\to 0\to\cdots$\\
$\Rightarrow\; R^4(\dfrac{3}{8})$ is periodic of period $1,$ but $\dfrac{3}{8}$ is not periodic which proves the claim.\\
We show that $\; \dfrac{4}{17}$ is periodic.\\
Since $\;\dfrac{4}{17}\to \dfrac{8}{17}\to \dfrac{16}{17}\to \dfrac{2}{17}\to \dfrac{4}{17}\cdot$\\
$\Rightarrow \dfrac{4}{17}$ is periodic of period $4.$\\
We now show that $\;\dfrac{2}{3}$ is a fixed point.\\
Since
\begin{equation}
\begin{split}
\notag
 f( \dfrac{2}{3})
&=-2 \dfrac{2}{3}+2\\
&= \dfrac{-4}{3}+2\\
&= \dfrac{2}{3},
\end{split}
\end{equation}
so $\dfrac{2}{3}$ is a fixed point.
\end{example}
\begin{example}
Consider $R(z)= z^2-2.$ Then
$0$ is a preperiodic point, since $\; 0\to -2\to 2\to 2\to\cdots$
$\;\Rightarrow R^2(0)$ is periodic of period $1$, but $0$ is not a periodic point.
\end{example}
\begin{example}
Consider $R(z)=z^2-1.$ Then
$1$ is an eventually periodic point,
since
$\;1\to 0\to -1\to 0\to -1\to\cdots  $
$\;\Rightarrow R(1)$ is periodic of period $2$, but $1$ not periodic.
\end{example}
\begin{example}
Consider $R(z)=z^2+z-1.$ Then
$0$ is eventually periodic,
since $0\to -1\to -1\to\cdots \;\Rightarrow R(0)$ is periodic of period $1$, but $0$ is not a periodic point.\\
Also $-2$ is eventually periodic since,
\begin{equation}
 -2\to 1\to 1\to\cdots
\end{equation}
$\Rightarrow R(-2)$ is periodic of period $1$, but $-2$ is not a periodic point.
\end{example}

\begin{definition}\label{d5}
\textbf{(Basin of attraction)}$\;$ Let $z_0$ be an attracting fixed point of a rational map $\R$.
The {\it{basin of attraction}} of $z_0$,  denoted by $\mathcal{A}$ (or $\A(z_0)$),  is defined as the set of all points $\{z:\  \R^n(z) \rightarrow z_0 , \ n \rightarrow \ity\}$.\\
The component of   this set $\mathcal{A}$ containing $z_0$ is called the local basin of attraction of $z_0$ or the immediate basin of attraction of $z_0$ denoted by $\mathcal{A}^{\ast}(z_0)$.
\end{definition}
\begin{proposition}
 $\mathcal{A}$ is an open set.
\end{proposition}
\begin{proof}
We show that  $ B=\mathcal{A}^{c}$ is closed where $B=\{z:\; \R^n(z) \nrightarrow z_0 , \ n \rightarrow \ity\}$.\\
Let $z_1\in\overline B,$ then $\exists$\, a sequence $(z_n)\subset B\;$ such that $\;z_n\to z_1, \;\Rightarrow \R^n(z_n)\to  \R^n(z_1)$ (\;$\R^n$ are continuous).
As $z_n\in B, \;\Rightarrow \R^n(z_n) \nrightarrow z_0. $ 
Since $\R^n(z_n)\to  \R^n(z_1),\;\Rightarrow \R^n(z_1)\nrightarrow z_0$ and so 
$z_1\in B.\;$ Hence $B$ is closed so that
 basin of attraction $\mathcal{A}$ is an open set.
\end{proof}
\begin{proposition}\label{1.4.12}
Let $z_0$ be an attracting fixed point of a rational map $\R$. Then the basin of attraction $\A$ is completely invariant.
\end{proposition}
\begin{proof} It is sufficient to show that $\R^{-1}(\A)=\A\quad(\R$ is surjective). By definition of $\A,$ it is forward invariant, so $\R(\A)=\A,\;\Rightarrow\A\subset\R^{-1}(\A).$\\
 Now let 
\begin{equation}
\begin{split}
\notag
z
&\in\R^{-1}(\A)\\
\Rightarrow\R(z)
&\in\A\\
\Rightarrow\lim_{n\to\ity}\R^n(\R(z))
&=z_0\quad(\text{by definition})\\
\Rightarrow\lim_{n\to\ity}\R^{n+1}(z)
&=z_0\\
\Rightarrow z
&\in\A\\
\Rightarrow\R^{-1}(\A)
&\subset\A\quad(z\in\R^{-1}(\A)\; \text{was arbitrary}).
\end{split}
\end{equation}
Thus $\R^{-1}(\A)=\A$ and hence $\A$ is completely invariant.
\end{proof}
Analogously we can define the basin of attraction of a (super) attracting periodic point.\\
Let $z_0$ be a (super) attracting periodic point of period n. Let $$O^+(z_0)=\{z_0, \R(z_0),\ldots,\R^{n-1}(z_0)\}$$ be the orbit of $z_0$. The basin of attraction of the orbit $O^+(z_0)$ is the collection of all points $z$ for which $\{\R^{jn}\}$ converges to some point of the orbit.\\
The local basin of a (super) attracting periodic orbit $O^+(z_0)=\{z_0, \R(z_0),\ldots, \R^{n-1}(z_0)\},$ is the union of the components of ${\F}$ which contains points of the orbit.\\
\begin{example}\label{eg1}
 Consider $R(z)=z^2.$ Then  $0$ and $\ity$ are the two fixed points of $R$. The basin of attraction of 0 is the unit disk ${\D}$ (which is connected and hence has only one component). So the basin of attraction of 0 coincides with its local basin of attraction.\\
The basin of attraction of $\ity$ equals $\ti{\C}\smallsetminus \overline{\D}$ (which is also connected and hence has only one component). Again the basin of attraction of $\ity$ coincides with its local basin of attraction.
\end{example}
We state \emph{Vitali's} theorem.
\begin{theorem}
 Let $\{f_1,f_2,\ldots  \}$ be a family of analytic maps normal in a domain $D$ and which converges pointwise to some function $f$ on some non-empty open subset $W$ of $D$. Then $f$ extends to a function $\Phi$ which is analytic on $D$, and $f_n\to \Phi$ locally uniformly on $D$.
\end{theorem}
We give a nice application of Vitali's theorem.
\begin{result}\label{1.4.15}
If $z_0$ is a (super) attracting fixed point of a rational map $\R$, then $z_0\in {\F},$ the Fatou set of $\R$.
\end{result}
\begin{proof}
As  $z_0$ is a (super) attracting fixed point of $\R,\;\Rightarrow |\R'(z_0)|<1$. Choose a number $\alpha$ such that $|\R'(z_0)|<\alpha<1.$ Since $\R$ is analytic at $z_0,\;\Rightarrow\;$ for some open neighbourhood $D$ of  $z_0 \;\R$ has a Taylor expansion $\R(z)=\R(z_0)+\R'(z_0)(z-z_0)\;$(On neglecting the higher order terms).\\
$\;\Rightarrow |\R(z)-\R(z_0)|=|\R'(z_0)||(z-z_0)|$.\\
 Now
\begin{equation}
\begin{split}
\notag
|\R(z)-z_0|
&=|\R(z)-\R(z_0)|\quad(\R(z_0)=z_0)\\
&=|\R'(z_0)||z-z_0|\\
&<\alpha |z-z_0|\\
&<|z-z_0|\quad( \alpha<1).
\end{split}
\end{equation} 
$\Rightarrow \R(z)\in D\;\;\forall z\in D$ and so $\R(D)\subset D$.\\
Thus each $\R^n$ maps $D$ into itself. Thus $\{\R^n\;\vert \;n\in \mathbb{N}\}$ is locally uniformly bounded on $D.$ By Theorem \ref{1.2.9}, $\{\R^n\;\vert \;n\in \mathbb{N}\}$ is normal on $D$. This implies that $z_0\in {\F}$ and $\R^n\to z_0$ uniformly on $D$.\\ Let $E$ be the component of ${\F}$ containing $z_0$ on which $\{\R^n\}$ is normal. Then by Vitali's theorem, $\R^n\to z_0$ locally uniformly on $E$.
\end{proof}
The above statement for (super) attracting fixed points generalize immediately to (super) attracting periodic points.
\begin{result}\label{1.4.16}
If $z_0$ is a (super) attracting periodic point of $\R$ of period $m$, then $z_0\in \F.$
\end{result}
\begin{proof}
We have $\R^m(z_0)=z_0,\;\Rightarrow z_0$ is a (super)attracting fixed point of $\R^m.$ From above result we have $z_0\in F(\R^m)=\F$ ( by Theorem \ref{1.3.19} ). Hence $z_0\in \F$. 
\end{proof}
\begin{proposition}\label{1.4.17}
 If $O^{+}(z_0)$ is a (super) attracting periodic orbit (of period n) then it is contained in $\F$.
\end{proposition}
\begin{proof}
 Let  $O^{+}(z_0)=\{z_0,z_1,\ldots,z_{n-1}\}$ be the forward orbit of $z_0$. Then each $z_j$, lies in a component say $F_j$ of $\F\;$ (by above) and as $k\to\ity,\R^{kn}\to z_j,$ locally uniformly on $F_j$.
\end{proof}
\begin{proposition}\label{1.4.18}
 If $z_0$ is a repelling fixed point of a rational map $\R$, then $z_0\in \J$.
\end{proposition}
\begin{proof}
 We have $\R(z_0)=z_0$ and $|\la|=|\R'(z_0)|>1.$ We shall see that no sequence of iterates of $\R$ can converge uniformly near $z_0$.
We know by \emph{Chain Rule},
\begin{equation}
\begin{split}
\notag
(\R^n)'(z_0)
&=\prod_{k=0}^{n-1}\R'(\R^k(z_0))\\
&=\R'(z_0)\,\R'(\R(z_0))\,\R'(\R^2(z_0))\ldots\R'(\R^{n-1}(z_0))\\
&=\R'(z_0)\,\R'(z_0)\,\R'(z_0)\ldots\R'(z_0)\quad( z_0\; \text{is a fixed point of}\;\R)\\
&=\la^n.
\end{split}
\end{equation}
Since $\R^n$ is analytic at $z_0,\;\Rightarrow$ for some open neighbourhood $D$ of $z_0 \;\R$ has a Taylor expansion,
\[
\R^n(z)=\R^n(z_0)+(\R^n)'(z_0)(z-z_0)\;(\text{On neglecting the higher order terms})
\]
\begin{equation}
\begin{split}
\notag
\Rightarrow|\R^n(z)-\R^n(z_0)|
&=|(\R^n){'}(z_0)||z-z_0|\\
&=|\la|^n|z-z_0|\quad(\text{by above})\\
&\to\ity\quad\text{as}\quad n\to\ity\quad(|\la|>1).
\end{split}
\end{equation}
$\Rightarrow z_0\notin \F.$ Hence $z_0\in \J$.
\end{proof}
The above proposition for repelling fixed points generalizes immediately to repelling periodic points.

Now we discuss a  \emph{linearization} problem due to \emph{Koenigs}. \\
First we state a result on infinite products which we will be using in the proof.
\begin{result}\label{1.4.19}
Suppose  the functions $f_1,f_2,\ldots$ are holomorphic in a domain $D$ and the series $\sum_{m=0}^{\ity}|f_m(z)|$ is locally uniformly convergent on $D$. Then $\lim_{n\to\ity}\prod_{m=1}^{n}(1+f_m(z))$ exists locally uniformly in $D$.
\end{result}
\begin{theorem}
\ Let $z_0$ be an attracting fixed point of a rational map $R$ with $R'(z_0)=a$ and let $B$ be the $local\;basin$ of $z_0$. Then there exists a unique analytic homeomorphism $g :B\to D_r$ for some $r$, where $D_r$ is some disk with centre $0$ and radius $r$  with $g(z_0)=0$ and $g'(z_0)=1$ and such that the diagram
\[
\begin{CD}
{B} @ >R>>  {B}\\
@VVgV @VVgV\\
D_r @>z\to az>>D_r
\end{CD}
\]
is commutative i.e $$gR(z)=ag(z).$$
\end{theorem}
\begin{proof}
 We may assume that $z_0=0$, for let $h$ be a M$\ddot{\text{o}}$bius map such that $h(z_0)=0.$ Then 
\begin{equation}
\begin{split}
\notag
hRh^{-1}(0)
&=hR(z_0)\quad(h\;\text{is one-one})\\
&=h(z_0)\quad(R(z_0)=z_0)\\
&=0.
\end{split}
\end{equation}
Thus 0 is a fixed point of the conjugate $hRh^{-1}$ of $R$. So, we may assume that $z_0=0$.\\
If $g$ satisfies $gR(z)=ag(z)$, near the origin, then for each n, $gR^n(z)=a^ng(z),$ so that
\begin{equation}\label{1.4.6}
\dfrac{R^n(z)}{a^n}=g(z)\B(\dfrac{g(R^n(z))}{R^n(z)}\B)^{-1}.
\end{equation}
Now $R^n(z)\to 0$ as $n\to\ity,\, z\in B.$
Also $g(0)=0, \;\Rightarrow$ it is $(\dfrac{0}{0})$ form. By \emph{L\,-\,Hospital Rule} we have
\begin{equation}
\begin{split}
\notag
\dfrac{R^n(z)}{a^n}
&=g(z)\B(\lim_{n\to\ity}\dfrac{g'(R^n(z))}{(R^n)^{'}(z)}(R^n)'(z)\B)^{-1}\\
&=g(z)(g'(0))^{-1}\\
&=g(z)\quad( g'(0)=1)
\end{split}
\end{equation}
\begin{equation}\label{1.4.7}
\Rightarrow \dfrac{R^n(z)}{a^n}\to g(z) \;\;\;\text{as}\;n\to\ity.
\end{equation}
Using this representation, we can establish the existence of $g$, viz. by proving that $\dfrac{R^n(z)}{a^n}$ converges to some limit as $n\to \ity$.\\
We observe that if this limit exists on some neighbourhood $W$ of $0$ then it exists throughout $B$ since, given $z\in B$ say $R^m(z)$ is in $W$, then
\begin{equation}
\begin{split}
\notag
\dfrac{R^{n+m}(z)}{a^{n+m}}
&=a^{-m}\B(\dfrac{R^n(R^m(z))}{a^n}\B)\\
&\to a^{-m}g(z)\quad\text{as}\;n\to\ity\quad(\text{ by (\ref{1.4.7})}).
\end{split}
\end{equation}
We also observe that if such a function $g$ exists, then it is unique, since limit is unique (in a Hausdorff space).\\
We now give the formal proof.\\
Let $N$ be a disk with centre $0$ and let $\alpha$ satisfies $|a|<\alpha<1$ and $|R(z)|\leqslant\alpha|z|$ (using Taylor expansion of $R$ around 0). Define the function $\phi$ on N by 
\begin{equation}\label{1.4.8}
R(z)=az(1+\phi(z)).
\end{equation}
$\Rightarrow \dfrac{R(z)}{az}
=1+\phi(z),\;
\Rightarrow\dfrac{R(z)}{az}-1
=\phi(z).$ 
Now $\lim_{z\to 0}\dfrac{R(z)}{az}$ is an indeterminate form, 
\begin{equation}
\begin{split}
\notag
\Rightarrow \lim_{z\to 0}\dfrac{R(z)}{az}
&=\lim_{z\to 0}\dfrac{R'(z)}{a}\quad (\text{using L\,-\,Hospital Rule})\\
&=\dfrac{a}{a}\quad( R'(0)=a)\\
&=1.
\end{split}
\end{equation}
$\Rightarrow \lim_{z\to 0}\dfrac{R(z)}{az}-1=0=\phi(0),\;
\Rightarrow \phi(0)=0.$
Choosing $N$ to be small enough, we may assume that $\phi$ is analytic on $\overline{N}.$ There exists\; $M>0$ such that $|\phi(z)|\leqslant M|z|\;$( a continuous function on a compact set is bounded). On $N$ we have, 
\begin{equation}
\begin{split}
\notag
|\phi(R^m(z))|
&\leqslant M|R^m(z)|\quad( R(N)\subset N,\;\Rightarrow\; R^m(z)\;\in N\; \text{for}\;z\in N)\\
&\leqslant\alpha^{m}|z|.
\end{split}
\end{equation}
Hence
\begin{equation}\label{1.4.9}
|\phi(R^m(z))|\leqslant\alpha^{m}|z|.
\end{equation}
For $z\in N,$
\begin{equation}
\begin{split}
\notag
 R^{n+1}(z)
&=R(R^n(z))\\
&=aR^n(z)\B(1+\phi(R^n(z))\B)\quad( R(N)\subset N)\\
&=aR(R^{n-1}(z))\B(1+\phi(R^n(z))\B)\\
&=a(aR^{n-1}(z))\B(1+\phi(R^{n-1}(z))\B)\B(1+\phi(R^n(z))\B)\\
&=a^2R(R^{n-2}(z))\B(1+\phi(R^{n-1}(z))\B)\B(1+\phi(R^n(z))\B)\\
&=a^3R^{n-2}(z)\B(1+\phi(R^{n-2}(z))\B)\B(1+\phi(R^{n-1}(z))\B)\B(1+\phi(R^n(z))\B)\\
&=\cdots\\
&=\cdots\\
&=\cdots\\
&=a^{n+1}z\B(1+\phi(z)\B)\B(1+\phi(R(z))\B)\B(1+\phi(R^2(z))\B)\cdots\B(1+\phi(R^n(z))\B)
\end{split}
\end{equation}
\begin{equation}\label{1.4.10}
\Rightarrow \dfrac{R^{n+1}(z)}{a^{n+1}}=z\prod_{m=0}
^{n}\B(1+\phi(R^m(z)\B).
\end{equation}
The expression (\ref{1.4.10}) together with expression (\ref{1.4.9}) and standard Result \ref{1.4.19} on infinite products ensures that the infinite product 
\[
g(z)=z\prod_{m=0}^{\ity}\B(1+\phi(R^m(z)\B)
\]
exists and is analytic on $N$. By (\ref{1.4.10}), $g$ satisfies (\ref{1.4.7}) on $N$. Also $g(0)=0.$
\begin{equation}
\begin{split}
\notag
g'(z)
&=\prod_{m=0}^{\ity}\B(1+\phi(R^m(z)\B)+z\prod_{m=0}^{\ity}\B(1+\phi(R^m(z)\B)'\\
\Rightarrow g'(0)
&=\prod_{m=0}^{\ity}\B(1+\phi(0)\B)+0\quad( R^m(0)=0)\\
&=1\quad( \phi(0)=0).\\
\end{split}
\end{equation}
This shows that $g$ is conformal in $B$.
Again from (\ref{1.4.10}) we have 
\begin{equation}
\begin{split}
\notag
g(R(z))
&=\lim_{n\to \ity}\dfrac{R^{n+1}(R(z))}{a^{n+1}}\\
&=a\lim_{n\to \ity}\dfrac{R^{n+2}(z)}{a^{n+2}}\\
&=ag(z)\quad z\in N
\end{split}
\end{equation}
and hence holds throughout $B$ (by \emph{Identity Theorem}).
\end{proof}
\begin{definition}
We say that an analytic map $f$ is linearizable at an attracting fixed point $z_0$, if there is some function $g$ that is analytic near $z_0$, with $g(z_0)=0, g'(z_0)=1$ and for all $z$ sufficiently close to origin, $gfg^{-1}(z)=f'(z_0)\,z$.
\end{definition}
The above theorem can be interpreted in the context of repelling fixed points, for if $z_0$ is a repelling fixed point of a rational map $R$, then there is a branch of $R^{-1}$ for which $z_0$ is an attracting fixed point. Applying the above theorem to this branch, we get that $R$ is locally conjugate to $z\to R'(z_0)\,z$ in some neighbourhood of $z_0$.

Next we deal with neutral (indifferent) fixed points. \\
Suppose $R(z_0)=z_0$ and $|\la|=|R'(z_0)|=1.$ Let $U$ be an open neighbourhood of $z_0$. If we conjugate $\{R\vert_ U\}$ to its derivative $z\to \la \,z,$ where $\la=R'(z_0)$, then the commutative diagram 
\[
\begin{CD}
{U} @ >R>>  {U}\\
@VV\phi V @VV\phi V\\
D_r @>z\to\la z>>D_r
\end{CD}
\] 
 gives a functional equation
\[
\phi\circ R(z)=\la\,\phi(z)\tag{SFE}
\]
where $D_r$ is a disk with origin as centre and radius $r$ and $\phi : U\to D_r$ is an analytic homeomorphism. This equation is called \emph{Schroder Functional Equation}\;(SFE).\\
Geometrically this equation means that $z_0$ is the centre of a disk on which the map $z\to \la\,z$ is a rotation, i.e $\;\exists\;$ an analytic homeomorphism $\phi : U\to D_r$ such that the diagram
\[
\begin{CD}
{U} @ >R>>  {U}\\
@VV\phi V @VV\phi V\\
D_r @>z\to e^{i\theta} z>>D_r
\end{CD}
\] 
commutes, with $\phi(z_0)=0$ and $\theta\in\mathbb{R}$.

We use \emph{Schwarz Lemma} (see Appendix \ref{A}) to investigate the dynamics of $R$ in a neighbourhood of a neutral fixed point $z_0$.
\begin{proposition}\label{1.4.22}
 Let $\R$ be a rational function with a neutral fixed point at $z_0$. Then $z_0\in\F$ if and only if $\R$ is locally conjugate to its derivative $z\to\la z$ on some neighbourhood of $z_0$, where $\la=\R'(z_0),\;|\la|=1.$
\end{proposition}
The same statement also holds for a neutral periodic point of period n if $\R$ is replaced by $\R^n$ in the definition of (SFE).
\begin{proof}
 Recall the definition of local conjugacy. Let $U\subset\ti{\C}$ be open and $\R: U\to U$ is a rational function with a fixed point at $z_0$. Suppose there is an analytic homeomorphim $\phi: U\to V$, where V is a neighbourhood of 0 with $\phi(z_0)=0$ and $f: V\to V$ is another rational map. We say that $\R$ is locally analytic conjugate to $f,$ if the diagram 
\[
\begin{CD}
{U} @ >\R>>  {U}\\
@VV\phi V @VV\phi V\\
V @>f >>V
\end{CD}
\] 
is commutative, i.e $\phi^{-1}f\phi=\R$.\\
Suppose there is a local conjugacy $\phi$ between $\R$ and $g:z\to\la z$ so that
$\; \phi^{-1}g\phi=\R.$
\[
\begin{CD}
{U} @ >\R>>  {U}\\
@VV\phi V @VV\phi V\\
V @>g >>V
\end{CD}
\] 
Then ${\R^n\vert_ {U}}=(\phi^{-1}g^n\phi)\vert_ {U}$ is locally uniformly bounded on $U$ and therefore is a normal family by Theorem \ref{1.2.9}. Hence $z_0\in \F.$\\
Conversely,
suppose that $z_0\in \F$, and let $U$ be the component of $\F$ that contains $z_0$. Suppose that $U$ is simply connected. Now $U\subset\ti{\C}$ and $z_0\in U$, therefore by Riemann mapping theorem,\;$\exists$ a unique conformal isomorphism $\phi$ from $U$ onto $\D$ with $\phi(z_0)=0$ and $\phi'(z_0)\neq 0.$ Let $g=\phi\R\phi^{-1}.$ Then
\begin{equation}
\begin{split}
\notag
g(0)
&=\phi\R\phi^{-1}(0)\\
&=\phi\R(z_0)\quad( \phi\;\text{is one-one})\\
&=\phi(z_0)\quad( \R(z_0)=z_0)\\
&=0.
\end{split}
\end{equation}
Also
\begin{equation}
\begin{split}
\notag
g\,\phi(z)
&=\phi\,\R(z)\\
\Rightarrow g'(\phi(z))\,\phi'(z)
&=\phi'(\R(z))\R'(z)\quad(\text{on differentiating both sides})\\
\Rightarrow g'(\phi(z_0))\,\phi'(z_0)
&=\phi'(\R(z_0))\,\R'(z_0)\quad(\text{on putting}\;z=z_0)\\
\Rightarrow g'(0)\,\phi'(z_0)
&=\phi'(z_0)\,\la\quad( \R'(z_0)=\la, \;|\la|=1)\\
\Rightarrow g'(0)
&=\la\quad( \phi'(z_0)\neq 0)\\
\Rightarrow  |g'(0)|
&=|\la|\\
&=1.
\end{split}
\end{equation}
Then
\begin{equation}
g=\phi\R\phi^{-1}: \D\to\D
\end{equation}
satisfies the hypothesis of Schwarz Lemma, and so $g(z)=\la z$ for some $\la$ with $|\la|=1,$ i.e $\phi\R\phi^{-1}(z)=\la z.$ Thus $\R$ is locally conjugate to its derivative on some neighbourhood of $z_0.$\\
If $U$ is not simply connected, then we consider $\ti U$, the universal cover of $U$ and apply the same reasoning to $\ti U$ as $\ti U$ is simply connected. ( $U\cap \J=\emptyset,$ so $U$ misses atleast 3 points and the \emph{Uniformization Theorem} (see Appendix \ref{A}) gives that the universal covering $\ti U$ of $U$ is conformally equivalent to $\D$).
\end{proof}
\begin{corollary}\label{1.4.23}
If $z_0$ is a neutral fixed point and $\la$ is a root of unity, then $$Schroder\;Functional\;Equation\; (SFE) $$ does not have a solution.
\end{corollary}
\begin{proof}
 Suppose $\la^k=1$ for some $k\in\mathbb Z^+$ and suppose (SFE) has a solution $\phi$ in a neighbourhood of $z_0.$ Then we have
\[
\begin{CD}
{U} @ >\R>>  {U}\\
@VV\phi V @VV\phi V\\
D_r @>g >>D_r
\end{CD}
\] 

\begin{equation}
\begin{split}
\notag 
\phi\R(z)
&=\la\phi(z)\quad(\text{where}\;g(z)=\la z)\\
\Rightarrow\phi\R\phi^{-1}(z)
&=\la z\\
\Rightarrow \phi\R^k\phi^{-1}(z)
&=\la^k z\quad(\R^k(U)\subset U)\\
&=z\quad( \la^k=1)\\
&=Iz, I\;\text{is the identity map}.
\end{split}
\end{equation}
$\Rightarrow \phi\R^k\phi^{-1}=I.\;$ Since $\R^k\sim \phi\R^k\phi^{-1},\;\Rightarrow \R^k\sim I$ on $U$ and so $\R^k\sim I$ on $\ti{\C}$
 by Identity theorem. Since $deg\,\R$ is invariant under conjugacy, therefore $deg\,I=deg\,\R^k.$
But $deg\,I=1$ and $deg\,\R^k>1,$ we get a contradiction. Hence (SFE) does not have a solution.
\end{proof}
\begin{corollary}\label{1.4.24}
 Suppose $\la=e^{2\pi i\theta},\theta\in \mathbb{Q},$ then $z_0\in \J.$
\end{corollary}
\begin{proof}
 Suppose $z_0\notin\J.$ Let $\theta=\dfrac{p}{q}\;\text{where}\;(p,q)=1.$ Then\\
$\la
=e^{2\pi i\frac{p}{q}},\;
\Rightarrow \la^q
=e^{2\pi ip}
=cos 2\pi p+i\,sin 2\pi p
=1.$\\
Since $z_0\in\F,$ therefore by Theorem \ref{1.4.22}, (SFE) has a solution $\phi$ in a neighbourhood $U$ of $z_0$ with $\phi(z_0)=0.$
\[
\begin{CD}
{U} @ >\R>>  {U}\\
@VV\phi V @VV\phi V\\
D_r @>g >>D_r
\end{CD}
\] 
Now
\begin{equation}
\begin{split}
\notag
\phi\R(z)
&=\la\phi(z)\quad(\text{where}\; g(z)=\la z)\\
\Rightarrow\phi\R\phi^{-1}(z)
&=\la z\\
\Rightarrow \phi\R^q\phi^{-1}(z)
&=\la^q z\quad( \R^q(U)\subset U)\\
&=z\quad(\la^q=1)\\
&=Iz, I\;\text{is the identity map}.
\end{split}
\end{equation}
$\Rightarrow \phi\R^q\phi^{-1}=I.$ Since $\R^q\sim \phi\R^q\phi^{-1},\;\Rightarrow \R^q\sim I$ on $U$ and so $\R^q\sim I$ on $\ti{\C}$
by Identity theorem.
Since $deg\,\R$ is invariant under conjugacy, therefore $deg\,I=deg\,\R^q.$ But $deg\,I=1$ and $deg\,\R^q>1,$ we get a contradiction. Hence $z_0\in \J.$
\end{proof}
\begin{definition}\label{d6}
\textbf{(Parabolic periodic point)}
 A periodic point $z_0=\R^n(z_0)$ is called parabolic if $\la=(\R^n)'(z_0)$ is some root of unity but $\R^k\neq I\;\; \forall\; k\in \mathbb N.$ 
\end{definition}
\begin{example}
 Consider $\R(z)=z^2-z.$ Then $0$ and $2$ are the fixed points of $\R(z)$ with multiplier $-1$ and $3$ respectively. So $0$ is a parabolic fixed point, since the multiplier $\la=-1$ of 0 is a root of unity and $\R^k\neq I\;\; \forall\; k\in \mathbb N.$ 
\end{example}
\begin{corollary}\label{1.4.27}
 Every parabolic periodic point of $\R$ belongs to  $\J.$
\end{corollary}
\begin{proof}
Let $z_0$ be a  parabolic periodic point of $\R$ and 
 suppose $z_0\notin\J,\;\Rightarrow z_0\in \F.$ By Theorem \ref{1.4.22}, (SFE) has a solution $\phi$ in a neighbourhood $U$ of $z_0$ with $\phi(z_0)=0.$ 
\[
\begin{CD}
{U} @ >\R^n>>  {U}\\
@VV\phi V @VV\phi V\\
D_r @>g >>D_r
\end{CD}
\] 
$\phi\R^n(z)=\la\phi(z)\quad(\text{where}\; g(z)=\la z).\\
\Rightarrow\phi\R^n\phi^{-1}(z)=\la z.$ Since $z_0$ is a parabolic periodic point, so  $\la^k=1$ for some $k\in\mathbb{N}.$ Then
\begin{equation}
\begin{split}
\notag
\phi\R^{nk}\phi^{-1}(z)
&=\la^k z\quad(\R^{nk}(U)\subset U)\\
&=z\\
&=Iz.
\end{split}
\end{equation}
$\Rightarrow \phi\R^{nk}\phi^{-1}=I.$ Since $\R^{nk}\sim \phi\R^{nk}\phi^{-1},\;\Rightarrow \R^{nk}\sim I$ on $U$ and so $\R^{nk}\sim I$ on $\ti{\C}$
by Identity theorem, which is a contradiction to $z_0$ being a parabolic periodic point. Therefore $z_0\in\J.$\\
Hence every parabolic periodic point belongs to $\J.$
\end{proof}
We now consider the case of  a rationally indifferent cycle .
\begin{theorem}
 If $deg\,\R\geqslant 2,$ then every rationally indifferent cycle of $\R$ lies in $\J.$
\end{theorem}
\begin{proof}
 We recall from Subsection \ref{s2} that $z_0$ lies in a rationally indifferent cycle of $\R$ of length $m,$ if $\R^m(z_0)=z_0$ and $\la=(\R^m)'(z_0)$ is some root of unity.\\
 First suppose that $z_0$ is a rationally indifferent fixed point of $\R$, and $z_0\notin \J, \;\Rightarrow z_0\in\F$. Suppose $\la^l=1$ for some $l\in\mathbb Z^+.$ Since $ z_0\in \F$, by Theorem \ref{1.4.22}, (SFE) has a solution $\phi$ in a neighbourhood $U$ of $z_0$ with $\phi(z_0)=0.$
\[
\begin{CD}
{U} @ >\R>>  {U}\\
@VV\phi V @VV\phi V\\
D_r @>g >>D_r
\end{CD}
\] 
Now
\begin{equation}
\begin{split}
\notag
\phi\R(z)
&=\la\phi(z)\quad(\text{where}\; g(z)=\la z)\\
\Rightarrow\phi\R\phi^{-1}(z)
&=\la z\\
\Rightarrow \phi\R^l\phi^{-1}(z)
&=\la^l z\quad( \R^l(U)\subset U)\\
&=z\quad( \la^l=1)\\
&=Iz, I\;\text{is the identity map}.
\end{split}
\end{equation}
$\Rightarrow \phi\R^l\phi^{-1}=I.$ Since $\R^l\sim \phi\R^l\phi^{-1},\;\Rightarrow \R^l\sim I$ on $U$ and so $\R^l\sim I$ on $\ti{\C}$, which is a contradiction to $deg\,\R\geqslant 2.$ Hence $z_0\in \J$.\\
Now if $z_0$ is any point of a rationally indifferent cycle of $\R$ of length $m$, then by above argument $z_0\in J(\R^m)=\J$ (by Theorem \ref{1.3.19}).\\
Hence every rationally indifferent cycle of $\R$ lies in $\J.$
\end{proof}

\section{Applications of Montel's theorem}\label{ch1,sec5}

It was \textbf{P. Montel} who first formulated the notion of a normal family of meromorphic functions and proved the criterion for normality that bears his name. This theorem has great applications in complex dynamics. Fatou and Julia used Montel's theorem as a key ingredient for studying the dynamical behaviour of the iterates of a rational function $\R.$ This section deals with the global consequences of Montel's theorem.\\
We state Montel's theorem.
\begin{theorem}\label{1.5.1}
[\textbf{Montel}]
Let $D$ be a domain in $\ti{\C}.$ Let $a,b$ and $c$ be three distinct points in $\ti\C$ and  let $U=\ti\C\smallsetminus\;\{a,b,c\}$  be another domain. Then the family $\mathfrak{F}=\{f_\la : D\to U\;\vert\;\la\in \Lambda\} $ of all meromorphic functions  is a normal family.
\end{theorem}

Now we discuss some of the properties of the Julia set $\J$ of a rational map $\R.$
\begin{result}\label{1.5.2}
 Let $\R$ be a rational map and let $z_0\in\J$. If U is a neighbourhood of $z_0,$ then the set 
\begin{equation}E_{U}=\ti{\C}\setminus\bigcup_{n>0}\R^n(U)
\end{equation}
 contains at most two points.
\end{result}
\begin{proof}
Suppose $E_{U}$ contains more than 2 points, say it contains 3 points. In that case the family $\{\R^n(U)\}=\{\R(U),\R^2(U),\ldots\R^n(U),\ldots\}$ omits 3 values in $\ti{\C}$ and hence $\{\R^n\vert_ U\}$ is a normal family by Montel's theorem. This implies that  $z_0\in\F,$ which is a contradiction to the fact that  $z_0\in\J.$  Hence the set $E_{U}$ contains atmost two points.
\end{proof}
Such points are called \emph{exceptional points}.\\
We show that the set  $E_{U}$ is backward invariant under $\R$ and hence completely invariant as $\R$ is surjective.\\
\begin{equation}
\begin{split}
\notag
E_{U}
&=\ti{\C}\setminus\bigcup_{n>0}\R^n(U)\\
\Rightarrow\R^{-1}(E_{U})
&=\R^{-1}(\ti{\C}\setminus\bigcup_{n>0}\R^n(U))\\
&=\R^{-1}(\ti{\C})\setminus\R^{-1}(\bigcup_{n>0}\R^n(U))\quad(\R^{-1}\; \text{preserves difference of sets})\\
&=\ti{\C}\setminus\bigcup_{n>1}\R^{n-1}(U)\\
&=\ti{\C}\setminus\bigcup_{n>0}\R^{n}(U)\quad(\text{on replacing }n\to n-1)\\
&=E_{U}.
\end{split}
\end{equation}
Thus $E_{U}$ is backward invariant and hence completely invariant under $\R$.

Consider a polynomial $p(z)$. Then by Remark \ref{1.3.17}, $\ity$ is always a superattractive fixed point of $p(z)$ with $p^{-1}(\ity)=\{\ity\}.\quad$(i.e $\ity$ is a fixed point whose only inverse image is itself). 
\begin{proposition}\label{1.5.3}
 Let $R$ be a rational map which fixes $z_0$ and $R^{-1}(z_0)=\{z_0\}.$ Then $R$ can be conjugated by a M$\ddot{o}$bius transformation $\phi$ such that $\phi(\ity)=z_0,$ and the result is a polynomial.
\end{proposition}
\begin{proof}
 Let $\phi(z)=\dfrac{az+b}{cz+d},\;\; ad-bc\neq 0$ be a M$\ddot{\text{o}}$bius transformation. Then 
\begin{equation}
\begin{split}
\notag
\phi(\ity)
&=z_0\\
\Rightarrow \dfrac{a}{c}
&=z_0\\
\Rightarrow a
&=cz_0.
\end{split}
\end{equation}
So the M$\ddot{\text{o}}$bius transformation takes the form $\phi(z)=\dfrac{cz_0z+b}{cz+d}\cdot$ Then\\
\begin{equation}
\begin{split}
\notag
\phi^{-1}R\phi(\ity)
&=\phi^{-1}R(z_0)\\
&=\phi^{-1}(z_0)\\
&=\ity\quad( \phi\;\text{is one-one}).
\end{split}
\end{equation}
Thus $\ity$ is a fixed point of the conjugate $Q=\phi^{-1}R\phi$, whose only inverse image is itself, for
\begin{equation}
\begin{split}
\notag
Q^{-1}(\ity)
&=(\phi^{-1}R\phi)^{-1}(\ity)\\
&=\phi^{-1}R^{-1}\phi(\ity)\\
&=\phi^{-1}R^{-1}(z_0)\\
&=\phi^{-1}(z_0)\quad( R^{-1}(z_0)=\{z_0\},\;\text{by assumption})\\
&=\{\ity\}\quad( \phi\;\text{is one-one}).
\end{split}
\end{equation}
Thus there are no poles of $Q$ in ${\C}$. Hence $R$ is conjugate to a polynomial $Q$.
\end{proof}
Now we give an example of an exceptional point.
\begin{example} Let $P$ be any polynomial. We have seen that $\ity$ is a superattracting fixed point of $P$ and therefore $\ity\in F(P)$ (by Result \ref{1.4.15}), so that $J(P)\subset{\C}.$ Let $z_0\in J(P)$ and $U$ be a neighbourhood of $z_0.$ Then
\begin{equation}
\begin{split}
\notag
\bigcup_{n>0}P^n(\C)
&={\C}\\
\Rightarrow \ti{\C}\smallsetminus\bigcup_{n>0} P^n(\C)
&=\ity.
\end{split}
\end{equation}
So we have 
\begin{equation}
\begin{split}
\notag
U
&\subset\C\\
\Rightarrow P^n(U)
&\subset P^n(\C)\\
\Rightarrow\bigcup_{n>0} P^n(U)
&\subset\bigcup_{n>0} P^n(\C)\\
\Rightarrow(\bigcup_{n>0} P^n(\C))^c
&\subset(\bigcup_{n>0} P^n(U))^c\quad(\text{ taking complement both sides})\\
\Rightarrow\ti\C\smallsetminus\bigcup_{n>0} P^n(\C)
&\subset\ti\C\smallsetminus\bigcup_{n>0} P^n(U)\\
\Rightarrow\ity
&\in\ti\C\smallsetminus\bigcup_{n>0} P^n(U).
\end{split}
\end{equation}
Hence $\ity$ is an exceptional point of the polynomial $P$.\\
The set of exceptional points of a rational map $\R$ is denoted by $E(\R)$.
\end{example}
\begin{theorem}\label{1.5.5}
 A rational map $\R$ with $deg\,\R\geqslant 2$, has atmost two exceptional points. 
\begin{enumerate}
\item\ If $E(\R)=\{z_0\}$ then $\R$ is conjugate to a polynomial with $z_0$ corresponding to $\ity.$
\item\ If $E(\R)=\{z_1,z_2\}$ where $z_1\neq z_2$, then 
 $\R$ is conjugate to some map $z\to z^{\pm{d}},$ where $z_1, z_2$ correspond to $0$ and $\ity$ respectively.
\end{enumerate}
\end{theorem}
\begin{proof} We know that $E(\R)$ is completely invariant, so by Theorem \ref{1.3.24}, $E(\R)$ has atmost two elements. Clearly,  $E(\R)$ consists of either one fixed point, two fixed points or an orbit of period two.
\begin{enumerate}
\item\ If  $E(\R)=\{z_0\},$ then by Proposition \ref{1.5.3}, $\R$ can be conjugated to a polynomial $Q=\phi^{-1}\R\phi,$ where $\phi$ is a M$\ddot{\text{o}}$bius transformation with $\phi(\ity)=z_0\;$(since $E=\{z_0\}$ is completely invariant, therefore $\R^{-1}(z_0)=\{z_0\}$). 
Hence $\R$ is conjugate to a polynomial $Q(z)=\phi^{-1}\R\phi(z)$ with the fixed point $z_0$ of $\R$ corresponding to the fixed point $\ity$ of $Q$.
\item\ $E(\R)=\{z_1,z_2\}$ where $z_1\neq z_2$.
Firstly, let $z_1, z_2$ be fixed points of $\R$ and let $\phi$ be a M$\ddot{\text{o}}$bius  transformation such that $\phi(\ity)=z_1$ and $\phi(0)=z_2.$ Then
\begin{equation}
\begin{split}
\notag
\phi^{-1}\R\phi(\ity)
&=\phi^{-1}\R(z_1)\\
&=\phi^{-1}(z_1)\\
&=\ity\quad(\phi\;\text{is one-one})
\end{split}
\end{equation}
and
\begin{equation}
\begin{split}
\notag
\phi^{-1}\R\phi(0)
&=\phi^{-1}\R(z_2)\\
&=\phi^{-1}(z_2)\\
&=0\quad(\phi\;\text{is one-one}).
\end{split}
\end{equation}
Thus $\R$ is conjugate to some polynomial $P(z),$ where $P(z)=\phi^{-1}\R\phi(z)$, and $0$,\,$\ity$ being the fixed points of $P(z)$. Then $0$ is a fixed point of multiplicity $d$ for $P(z),$ as no other points can map to $0$ (as $E(\R)$ is completely invariant). Also $\ity$ is a fixed point of multiplicity $d$ for $P(z),$ as no other points can map to $\ity \;$(as $E(\R)$ is completely invariant). Hence $P(z)=kz^d,$ where the constant $k$ can be eliminated by conjugation.\\
Similarly it can be seen easily that if $0$ and $\ity$ form an orbit of period two, then $P(z)=kz^{-d},$ where the constant $k$ can be eliminated by conjugation.
\end{enumerate}
\end{proof}
As a consequence of this theorem we have,
\begin{remark}\label{1.5.6}
If $\R$ is a rational map with $deg\,\R\geqslant 2$, then $E(\R)\subset\F.$
\end{remark}
\begin{proof}
 Following the proof of previous theorem we get that, in the first case $E(\R)$ consists of a superattractive fixed point which belongs to $\F$ by Result \ref{1.4.15}. Hence  $E(\R)\subset\F$.\\
In the second case we have two subcases. In subcase one $E(\R)$ consists of two superattractive fixed points which belongs to $\F$ by  Result \ref{1.4.15}. Hence  $E(\R)\subset\F$.\\
In subcase two $E(\R)$ consists of an orbit $\{z_1,z_2\}$ of period 2, and $\R$ is conjugate to some map $ p : z\to z^{{-}{d}}$ where $z_1, z_2$ correspond to $0$ and $\ity$ respectively. Then
\begin{equation}
\begin{split}
\notag
p(z)
&=z^{-d}\\
\Rightarrow p^2(z)
&={z^{-d}}^{-d}\\
&=z^{d^2}.
\end{split}
\end{equation}
So $p^2(z)=z^{d^2}.$ Thus $E(\R)$ consists of a superattractive orbit of period 2, and hence belongs to $\F$ by  Result \ref{1.4.16}. Hence  $E(\R)\subset\F$.
\end{proof}
\begin{corollary}
 Let $\R$ be a rational map with $deg\,\R\geqslant 2$ and let $int\,(\J)\neq\emptyset$ where $int\,\J$ denotes interior of $\J$, then $\J=\ti\C.$
\end{corollary}
\begin{proof}
 Suppose $z_0\in int\,(\J)$. Then $\;\exists\,$ a neighbourhood $U$ of $z_0$ such that $$z_0\in U\subset int\,(\J)\subset \J.$$
By Theorem \ref{1.3.23}, $\J$ is forward invariant so that
\begin{equation}
\begin{split}
\notag
\J
&\supset\bigcup_{n>0}\R^n(U)\\
&=\ti\C\smallsetminus E(\R)\\
\Rightarrow\J
&\supset\overline{\ti\C\smallsetminus E(\R)}\quad(\J\;\text{is closed})\\
&=\ti\C.
\end{split}
\end{equation}
$\Rightarrow\J\supset\ti\C\supset\J,\;\Rightarrow\J=\ti\C$.\\
Thus if $int\,(\J)\neq\emptyset,$ then $\J=\ti\C.$
\end{proof}

Recall (Definition \ref{d5}) of basin of attraction $\A$ of an attracting fixed point $z_0$ of $\R$.\\
We show $\J$ equals the boundary of $\mathcal{A}$.
\begin{corollary}\label{1.5.8}
 Let $\mathcal{A}$ be the basin of attraction of an  attracting fixed point $z_0$ of $\R$. Then $\partial{\mathcal{A}}=\J$.
\end{corollary}

\begin{proof}
 Firstly we show that $\J\subset\partial{\mathcal{A}}$.\\
 So let $z_1\in\J$ and $N$ be a neighbourhood of  $z_1$. By Result \ref{1.5.2} we know that the set
\begin{equation}
E_N=\ti\C\smallsetminus\bigcup_{n>0}\R^n(N)
\end{equation}
 consists of atmost two points which lie in $\F,$ by Remark \ref{1.5.6}. So $\bigcup_{n>0}\R^n(N)$ intersects $\mathcal{A}\;$(since $\bigcup_{n>0}\R^n(N)$ can miss at the most 2 points)
and hence some $\R^n(N)$ intersects $\mathcal{A}$.\\
 We claim that $N\cap\A\neq\emptyset.$\\
 Since $\R^n(N)\cap \A\neq\emptyset,$ let $w\in\R^n(N)\cap \A.$ So $w\in\R^n(N)$ and $w\in\A.$ As $w\in\R^n(N),$ so $w=\R^n(z)$ for some $z\in N,\;\Rightarrow z=\R^{-n}(w)\subset O^{-}(w)$, the backward orbit of $w.$ Since $w\in\A$ and $\A$ is completely invariant (by Proposition \ref{1.4.12}),\;$\Rightarrow z\in\A$ and hence $ N\cap\A\neq\emptyset$, which proves the claim.\\
  Now $z_1\in N$ and $N\cap\A\neq\emptyset,\;\Rightarrow z_1 \in\overline{\A}$ and so $\J\subset\overline{\A}\; (z_1\in\J$ was arbitrary). Since $\;\J\cap\A=\emptyset,\;\Rightarrow \J\subset\partial\A\;\;(\overline{\A}=\A\cup\partial\A).$\\
For the reverse inclusion,
let $z\in\partial\A$ and $U$ be a neighbourhood of $z.$ Then $U$ contains points of $\A$ and $\A^{c}.$ Any limit of iterates $\{\R^n\vert_U\}$ has a jump discontinuity between $\A$ and $\partial\A\;$(otherwise $\A$ will contain points which are outside $\A$, a contradiction). Therefore $z\notin\F$ and so $z\in\J$ which implies that $\partial\A\subset\J.$\\ Hence we conclude that $\partial\A=\J.$
\end{proof}
We illustrate with examples.
\begin{example}\label{eg2}
 Consider the polynomial $p(z)=z^2-2.$ Then by Remark \ref{1.3.17}, $\ity$ is a superattracting fixed point of $p.$ Define $\;h : \{\xi:|\xi|>1\}\to \ti\C\smallsetminus\,[-2,\,2]$ as\;
$h(\xi)=\xi+\dfrac{1}{\xi}.$ It can be easily seen that $h$ is a well defined conformal mapping of $ \{\xi:|\xi|>1\}$ onto $\ti\C\smallsetminus\,[-2,\,2].$ Then we have,
\begin{equation}
\begin{split}
\notag
ph(\xi)
&=p(\xi+\dfrac{1}{\xi})\\
&=(\xi+\dfrac{1}{\xi})^2 -2\\
&=\xi^2+\dfrac{1}{\xi^2}\\
&=h(\xi^2)\\
\Rightarrow h^{-1}ph(\xi)
&=\xi^2.
\end{split}
\end{equation}
Thus $p(z)$ is conjugate to the map $g:\xi\to\xi^2.$ As the dynamics of two conjugate maps are same, therefore the dynamics of $p(z)$ on  $\ti\C\smallsetminus\,[-2,\,2]$ is same as those of $g$ on $ \{\xi:|\xi|>1\}$. Since the iterates of any $\xi,\, |\xi|>1$ under $g$ tends to $\ity$, so will the iterates of any $z\in\ti\C\smallsetminus\,[-2,\,2]$ under $p$.\\
We claim that $[-2,\,2]$ is invariant under $p$.\\
The boundary of $[-2,\,2]$ is $\{-2,2\}.$ Now $p(-2)=2\in [-2,\,2],$ and $p(2)=2\in  [-2,\,2].$ Since $p$ is analytic, therefore $p(-2<z<2)\subset -2\leqslant z <2\;$(as $p(0)=-2$). Therefore $[-2,\,2]$ is invariant under $p$.\\
 Now, the basin of attraction of $\ity$ for $p,$ viz.  $\A(\ity)=\{z\, |\, p^n(z)\to\ity\}=\ti\C\smallsetminus\,[-2,\,2].$ Therefore by Corollary \ref{1.5.8}, $\partial\A(\ity)=J(p),$ the Julia set of $p$ so that $J(p)=[-2,\,2].$
\end{example}
\begin{example}\label{eg3}
 Consider $\R(z)=z^2.$ Then\;$\R^n(z)=z^{2^{n}}$ converges to 0 in $\{z\,:\,|z|<1\}$ and to $\ity$ in $\{z\,:\,|z|>1\}.$ This implies that $\A(0)=\{z\,:\,|z|<1\}$ and $\A(\ity)=\{z\,:\,|z|>1\}.$ So by Corollary \ref{1.5.8}, $\partial\A(0)=\partial\A(\ity)=\{z\,:\,|z|=1\}.$ Therefore  Julia set of $\R$  is the unit circle.
\end{example}
\begin{corollary}\label{1.5.11}
If $z_0\in\ti\C\smallsetminus\;E(\R)$, then 
\begin{equation}
\J\subset\{\emph{accumulation points of}\;\bigcup_{n\geqslant 0}\R^{-n}(z_0) \}.
\end{equation}
 Consequently, if $z_0\in\J$ then
\begin{equation}
\J=\emph{closure}\;\B(\bigcup_{n\geqslant 0}\R^{-n}(z_0)\B).
\end{equation}
\end{corollary}
\begin{proof}
Let $w\in\J$ and $U$ be any neighbourhood of $w.$ Since$\;z_0\in\ti\C\smallsetminus\;E(\R)=\bigcup_{n>0}\R^n(U),\;\Rightarrow z_0\in\bigcup_{n>0}\R^n(U)$ and so $z_0\in\R^n(U)$ for some $n>0$.$\;\Rightarrow \R^{-n}(z_0)\in U$.$\;\Rightarrow\;\text{for}\;\epsilon>0,  \;\exists \;n_0\in\mathbb N\,$ such that $|\R^{-n}(z_0)-w|<\epsilon, \;\forall\;n\geqslant n_0,$ which implies that $w $ is an accumulation point of $\,\bigcup_{n\geqslant 0}\R^{-n}(z_0)$. Hence $\J\subset\{$accumulation points of\;$\bigcup_{n\geqslant 0}\R^{-n}(z_0)\}$.\\
If $z_0\in\J$, then
\begin{equation}
\begin{split}
\notag
\bigcup_{n\geqslant 0}\R^{-n}(z_0)
&\subset\J\quad(\J\;\text{ is backward invariant, by Theorem \ref{1.3.23}})\\
\Rightarrow \emph{closure}\,\{\bigcup_{n\geqslant 0}\R^{-n}(z_0)\}
&\subset\J\quad(\J\;\text{ is closed}).
\end{split}
\end{equation}
Hence $\J=\emph{closure}\,\{\bigcup_{n\geqslant 0}\R^{-n}(z_0)\}.$
\end{proof}
Thus we have shown that iterated preimages of any point $z_0\in\J$ is dense in $\J$.
\begin{theorem}\label{1.5.12}
$\J$ contains no isolated points.
\end{theorem}
\begin{proof}
As $\J$ is non-empty, let $z_0\in\J$ and $U$ be a neighbourhood of $z_0.$
 First assume that $z_0$ is not periodic. Therefore $\R^n(z_0)\neq z_0\;\,\forall\,n\in\mathbb N,\;\Rightarrow z_0\notin\R^{-n}(z_0)\;\forall\;n\in\mathbb N.$ In particular $z_0\notin\R^{-1}(z_0).$ Let $z_1\in\R^{-1}(z_0)\subset\J\;(\J$ is backward invariant, by Theorem \ref{1.3.23}),$\;\Rightarrow\R(z_1)=z_0.$\\
We claim that $z_1\notin O^{+}(z_0)$ i.e 
\begin{equation}\label{eq3}
z_1\neq\R^n(z_0)\;\;\forall\;n\in\mathbb N.
\end{equation}
Let if possible $z_1=\R^n(z_0)$ for some n,$\;\Rightarrow z_0=\R(z_1)=\R^{n+1}(z_0)$, a contradiction to $z_0$ being non-periodic. Hence the claim.\\
As $z_1\in\J$, therefore by Corollary \ref{1.5.11}, backward iterates of $z_1$ are dense in $\J$. Now $U$ is any neighbourhood of $z_0\in\J,\;\Rightarrow U\cap\{\bigcup_{n\geqslant 0}\R^{-n}(z_1)\}\neq\emptyset.$ Let $w\in U\cap\{\bigcup_{n\geqslant 0}\R^{-n}(z_1)\},\;\Rightarrow w\in\R^{-n}(z_1)$ for some $n>0$ and so $w\in\J$ which implies that $w\in U\cap\J$ and $w\neq z_0 \;(\text{using}\;(\ref{eq3}))$. Thus $U\cap\J$ is a neighbourhood of $z_0\in\J$ which contains some other point $w\neq z_0.$ Hence $z_0$ is an accumulation point of $\J.$ Since $z_0\in\J$ was arbitrary, therefore every point of $\J$ is an accumulation point of $\J$. \\
Now assume that $z_0$ is periodic with period n.\\
If $z_0$ were the only solution of the equation
\begin{equation}
\R^n(z)=z_0,
\end{equation}
then by Proposition \ref{1.5.3}, we can conjugate $\R^n$ to some polynomial say, $Q$  and the fixed point $z_0$ of $\R^n$ corresponds to the fixed point $\ity$ of $Q$. This implies that $z_0$ is a superattracting fixed point of $\R^n,$ a contradiction to $z_0\in\J.$ Hence there is $z_1\neq z_0$ with $\R^n(z_1)=z_0.$ Again $z_1\notin O^{+}(z_0),$ for if  $\R^m(z_0)=z_1$ for some m, then by periodicity of $z_0,\;\;0\leqslant m<n,$ and 
\begin{equation}
\begin{split}
\notag
\R^m(z_0)
&=\R^{n+m}(z_0)\\
&=\R^n(z_1)\\
&=z_0
\end{split}
\end{equation}
a contradiction to minimality of n. As done earlier $\R^m(w)=z_1,$ for some $w\in U\cap\J,$ and $w\neq z_0$. Hence $z_0$ is an accumulation point of $\J$ and so $\J$ contains no isolated points. Thus $\J$ is a \emph{perfect set} (and so is uncountable).
\end{proof}
\begin{theorem}\label{1.5.13}
Let $\R$ be a rational map, $deg\,\R\geqslant 2,$ and let $E$ be a closed, completely invariant subset of $\ti\C.$ Then either :
\begin{enumerate}
\item\ E has atmost two elements and $E\subset E(\R)\subset\F;$ or
\item\ E is infinite and $E\supset\J.$
\end{enumerate}
\end{theorem}
\begin{proof}
Clearly, $E$ is either finite or infinite.\\
 If $E$ is finite, then by Theorem \ref{1.3.24} and Theorem \ref{1.5.5}, $E$ has atmost two elements which are exceptional points. Therefore $E\subset E(\R)\subset\F.$\\
Suppose now that $E$ is infinite. As $E$ is completely invariant,  by Corollary \ref{1.3.21},   $E^c=\Omega\;$(say) is also completely invariant and $\Omega$ is an open set. Hence $\R^n(\Omega)=\Omega\;\;\forall\;n.$ Let $a, b$ and $c$ be any three distinct points in $E$. The family $\{\R^n\vert_\Omega : n\in\mathbb{N}\}$  is meromorphic in $\Omega$ and misses $a,b,c.$ By Montel's theorem,\;$\{\R^n\vert_\Omega : n\in\mathbb{N}\}$ is normal in $\Omega$. Therefore $\Omega\subset\F$, so that ${\Omega}^c\supset{\F}^c$ i.e $E\supset\J$ and this completes the proof.
\end{proof}
\begin{remark}\label{1.5.14}
The above theorem says that $\J$ is the smallest closed, completely invariant set with atleast three points. This property is referred as \emph{minimality} of $\J.$
\end{remark}
\begin{theorem}\label{1.5.15}
Let $\R$ be a rational map, $deg\;\R\geqslant 2.$ Then 
\begin{equation}
\J\subset \emph{closure}\,\{\emph{periodic points of}\;\R\}.
\end{equation}
\end{theorem}
\begin{proof}
Let $N$ be any open set in $\ti\C.$ Then $N\cap\J$ is open in $\J.$ Assume that $N\cap\J\neq\emptyset\;$(otherwise take any other open set which intersects $\J$).\\
 We show that $N$ contains some periodic point of $\R.$\\
Let $w\in N\cap\J,$ and we assume that $w$ is not a critical value of $\R^2\;$(otherwise replace it with another nearby point of $\J$ which is not a critical value). As $deg\;\R^2\geqslant 4$ and $w$ is not a critical value of $\R^2$, then by the discussion following Definition \ref{d3}, we get that\;$\R^{-2}(w)$ has atleast $4$ distinct points. Let $w_1,w_2$ and $w_3$ be three of them distinct from $w.$ Let $N_0, N_1, N_2$ and $N_3$ be neighbourhoods of $w,w_1,w_2$ and $w_3\;$(with pairwise disjoint closures, this is possible as $\ti\C$ is a normal space) such that $N_0\subset N$ and for $1\leqslant j\leqslant 3$
\begin{equation}\label{eq4}
\R^2 : N_j\to N_0
\end{equation}
is a homeomorphism. Also let for $1\leqslant j\leqslant 3$,
\begin{equation}
S_j : N_0\to N_j
\end{equation}
be the inverse of  (\ref{1.5.8}).\\
We claim that $\;\exists\;$ some $z_0\in N_0,$ some $j\in\{1,2,3\}$ and some $n\geqslant 1,$ such that
\begin{equation}\label{eq5}
\R^n(z_0)=S_j(z_0).
\end{equation}
Suppose the claim is not true.\\
Then $\R^n(z)\neq S_j(z)\;\forall z\in N_0,\, j\in\{1,2,3\}$ and \,$\forall\;n\geqslant 1.$\;This implies that $\{\R^n\vert_{N_0}\}$ misses atleast 3 values, say $S_1(z), S_2(z)$ and $S_3(z).$ Hence by Montel's theorem\;$\{\R^n\vert_{N_0}\}$ forms a normal family so that  $N_0\subset\F.$ But $N_0\cap\J\neq\emptyset\;\{\text{as}\; w\in N_0\cap\J\},\;\Rightarrow w\in\F\cap\J,$ a contradiction. This proves the claim.\\
 Hence $\;\exists\;$ some $z_0\in N_0,$ some $j\in\{1,2,3\}$ and some $n\geqslant 1,$ such that 
(\ref{eq5}) holds. Then $\R^2(\R^n(z_0))=\R^2 S_j(z_0)=z_0\;(\R^2 S_j=I\,\,\text{on}\,\,N_0),\;\Rightarrow \R^{n+2}(z_0)=z_0$ and so $z_0$ is a periodic point contained in $N_0\subset N.$ This implies that $N\cap\{\emph{periodic points\,of}\;\R\}\neq\emptyset,$ so that $w\in \emph{closure}\,\{\emph{periodic points of}\;\R\}.$  Therefore  $\J\subset \emph{closure}\,\{\emph{periodic points of}\;\R\}.$
\end{proof}
\begin{remark}\label{1.5.16}
Since $\J$ is an infinite set, the above theorem says that $\R$ has infinitely many periodic points.
\end{remark}
\begin{definition}\label{d7}
 Let $\R :\ti{\C}\rightarrow\ti{\C}$ be a rational map with $deg\,\R=d \geqslant 2$ and let $z \in \ti{\C}$. The \emph{`deficiency'} $\delta_{z}$ of $z$ is defined as
\begin{equation}
\delta_{z}= d - (\emph{the cardinality of}\;\, \R^{-1}(z)).
\end{equation}
 For any set $A$, we define the total deficiency of $\R$ over $A$ as 
\begin{equation}
\delta_{A} = \Sigma_{z \in A}\;{\delta_z}.
\end{equation}
\end{definition}
The total deficiency of $\R$ is the sum of all the deficiencies for all values of $\R.$
\begin{proposition}\label{1.5.18}
 Let $\R :\ti{\C}\rightarrow\ti{\C}$ be a rational map  with $deg\,\R= d \geqslant 2$ and let $z\in\ti\C.$ Then $\delta_{z}\neq 0$ if and only if $z$ is a critical value of $\R.$
\end{proposition}
\begin{proof}
We know from (Definition \ref{d3}) that if $z$ is not a crtical value of $\R$, then $\R^{-1}(z)$ consists of $d$ distinct points. Therefore if $z$ is a critical value of $\R$, then cardinality\,$\{\R^{-1}(z)\}<d.$ Therefore $d-$cardinality\,$\{\R^{-1}(z)\}\neq 0$ if and only if $z$ is a critical value of $\R$, i.e $\delta_{z}\neq 0$ if and only if $z$ is a critical value of $\R.$
\end{proof}
\begin{theorem}\label{1.5.19}
 Let $\R$ be a rational map of $deg\, d \geqslant 2$. If $z_0$ is an attracting fixed point, then the local basin of attraction $\A^{\ast}(z_0)$ contains atleast one critical point.
\end{theorem}
\begin{proof}
Since $z_0$ is an attracting fixed point, therefore the multiplier $\la$ of $z_0$ satisfies $0<|\la|<1.$ Let $U_0=\Delta(z_0,\epsilon)$ be a small disk invariant under $\R$ on which the analytic branch $f$ of $\R^{-1}$ satisfying $f(z_0)=z_0$ is defined (as $0<|\la|<1,\;\exists$ a neighbourhood of $z_0$ in which $\R$ is one-one and by \emph{Inverse function theorem}$,\,f$ satisfying $f(z_0)=z_0$ is defined, since $f(z_0)=f(\R(z_0))=I(z_0)=z_0$ ). We have $U_0\subset\A^{\ast}(z_0)$ and $\R(U_0)\subset U_0.$ Also $\R^{-1}(U_0)\subset\A^{\ast}(z_0)$ i.e $\;f(U_0)\subset\A^{\ast}(z_0)$ and $f$ is one-one on $U_0.$ So $f : U_0\to f(U_0)=U_1\;$(say) is a homeomorphism. Since simple connectedness is preserved under a homeomorphism, therefore $U_1=f(U_0)$ is simply connected. Since $\R(U_0)\subset U_0,\;\Rightarrow U_0\subset\R^{-1}(U_0)=f(U_0)=U_1$ and so $U_0\subset U_1.$ Proceeding on similar lines, we construct $U_{n+1}=f(U_n)\supset U_n$ and extend $f$ analytically to $U_{n+1}$\;(by analytic continuation). If the procedure does not terminate, then we obtain a sequence $\{f^n: U_0\to U_n\}$ of analytic functions on $U_0$ such that 
\begin{equation}
\{f^n\vert_ {U_0}\}_{n>0}\;\subset\A^{\ast}(z_0)\;\subset\F.
\end{equation}
Therefore $\{f^n\vert_ {U_0}\}$ is normal on $U_0$ and so  $U_0\subset F(f)$, the Fatou set of $f$ which is a contradiction, since $z_0\in U_0$ is a repelling fixed point of $f$, therefore $z_0\in J(f),$ the Julia set of $f$ by Proposition \ref{1.4.18} (since $z_0$ is an attracting fixed point of $\R$, therefore it is a repelling fixed point of the analytic branch $f$ of $\R^{-1}$). Thus we get an $N$ such that we cannot extend $f$ to $U_n\;\forall\;n>N.$ This implies, there is a critical point $p\in\A^{\ast}(z_0)$ such that critical value $\R(p)\in U_n.$  Hence $\A^{\ast}(z_0)$ contains atleast one critical point.
\end{proof}
\begin{remark}\label{1.5.20}
The above theorem can be extended to the case when $z_0$ is an attracting periodic point of period $n>1.$\\
 The multiplier $\la=(\R^n)'(z_0)$ satisfies $|\la|<1.$ If $\{z_0,z_1,\ldots,z_{n-1}\}$ is the cycle of $z_0,$ then each $z_j$ is a fixed point of $\R^n.$ The above argument shows that each component of $\A^{\ast}(z_0)$  contains a critical point of $\R^n.$ By Chain Rule we have the identity
\begin{equation}
\begin{split}
\notag
(\R^n)'(z)
&=\prod_{k=0}^{n-1}\R'(\R^k(z))\\
&=\R'(z)\,\R'(\R(z))\cdots\R'(\R^{n-1}(z))\;\;\forall\;z\in\ti\C.
\end{split}
\end{equation}
If $z$ is a critical point of $\R^n$ then $(\R^n)'(z)=0,\;\Rightarrow \prod_{k=0}^{n-1}\R'(\R^k(z))=0$ and so $\R'(\R^k(z))=0$ for some $0\leqslant k\leqslant n-1.$ Therefore $\R^k(z)$ is a critical point of $\R$ and $\R^k(z)\in\A^{\ast}(z_0)$. Hence $\A^{\ast}(z_0)$ contains atleast one critical point.
\end{remark}
\begin{theorem}\label{1.5.21}
 Let $\R$ be a rational map with $deg\,\R= d\geqslant 2$. Then there are atmost $2d-2$ critical points counting multiplicity.
\end{theorem}
\begin{proof}
As degree of a rational map is invariant under conjugation, it suffices to prove the theorem for any conjugate $S$ of $\R$ for which $\ity$ is neither a critical point nor a critical value. Let $z_0\in\ti\C$ such that $\R(z_0)\neq z_0.$ 
We construct a M$\ddot{\text{o}}$bius transformation $g$ such that $g(\ity)=z_0$ and $g^{-1}(\R(z_0))=0.$\\
 Suppose $g(z)=\dfrac{az+b}{cz+d}\,,\;ad-bc\neq 0.$ Since $g(\ity)=z_0,\;\Rightarrow\dfrac{a}{c}=z_0$ and so $a=cz_0.$ Then $g(z)=w,\;\Rightarrow\dfrac{az+b}{cz+d}=w,\;\Rightarrow az+b=cwz+wd,\;\Rightarrow z(a-cw)=wd-b$ and so $ z=\dfrac{dw-b}{a-cw}\cdot$ Therefore
\begin{equation}\label{eq6}
g^{-1}(w)=\dfrac{dw-b}{a-cw}\quad(\text{since}\;\R(z_0)\neq z_0,\;\Rightarrow\; g^{-1}\;\text{is well defined}).
\end{equation}
Now $g^{-1}(\R(z_0))=0,\;\Rightarrow d\,\R(z_0)-b=0,\;\Rightarrow b=\R(z_0)\,d.$ So
\begin{equation}
g(z)=\dfrac{cz_0z+\R(z_0)\,d}{cz+d}
\end{equation}
is the required M$\ddot{\text{o}}$bius transformation such that  $g(\ity)=z_0$ and $g^{-1}(\R(z_0))=0.$\\
Then $g^{-1}\R g(\ity)=g^{-1}\R(z_0)=0.$ Let us call the conjugate  $g^{-1}\R g$ of $\R$ as $\R$ itself. Then $\R(\ity)=0.$ Therefore $\R(z)=\dfrac{P(z)}{Q(z)},$ where $P$ and $Q$ are coprime polynomials, has the form
\begin{equation}
\begin{split}
\notag
\R(z)
&=\dfrac{P(z)}{Q(z)}\\
&=\dfrac{\alpha z^{d-1}+\cdots}{\beta z^d+\cdots}\;\;\alpha,\beta \neq 0\\
\Rightarrow \R'(z)
&=\dfrac{Q(z)P'(z)-P(z)Q'(z)}{Q(z)^2}\\
&=\dfrac{(\beta z^d+\cdots)((d-1)\alpha z^{d-2}+\cdots)-(\alpha z^{d-1}+\cdots)(\beta dz^{d-1}+\cdots)}{Q(z)^2}\\
&=\dfrac{(\alpha\beta(d-1)z^{2d-2}+\cdots-\alpha\beta dz^{2d-2}+\cdots)}{Q(z)^2}\\
&=\dfrac{-\alpha\beta z^{2d-2}+\cdots}{Q(z)^2}
\end{split}
\end{equation}
Then $\R'(z)=0$ implies that
\begin{equation}\label{eq7}
-\alpha\beta z^{2d-2}+\cdots=0.
\end{equation}
Since $\alpha\beta\neq 0,$ equation (\ref{eq7}) has atmost $2d-2$ solutions counting multiplicity in $\ti\C.$ Hence there are atmost $2d-2$ critical points in $\ti\C$ counting multiplicity.
\end{proof}
\begin{corollary}
Let $\R$ be a rational map with $deg\,\R=d \geqslant 2$. Then the number of attracting fixed points of $\R$ is atmost $2d-2.$
\end{corollary}
\begin{proof}
If $z_0$ is an attracting fixed point of $\R$, then by Theorem \ref{1.5.19}, the local basin of attraction\;$\A^{\ast}(z_0)$ contains atleast one critical point of $\R.$ As there are atmost $2d-2$ critical points of  $\R$ counting multiplicity, therefore there can be atmost $2d-2$ attracting fixed points of  $\R$.
\end{proof}
\begin{corollary}
Let $\R$ be a rational map with $deg\,\R= d \geqslant 2$. Then the number of attracting periodic orbits of $\R$ is atmost $2d-2.$
\end{corollary}
\begin{proof}
If $z_0$ is an attracting periodic point of $\R$, then by  Remark \ref{1.5.20}, the local basin of attraction$\;\A^{\ast}(z_0)$ contains atleast one critical point of $\R.$ As there are atmost $2d-2$ critical points of  $\R$ counting multiplicity, therefore there can be atmost $2d-2$ attracting periodic orbits of $\R$.
\end{proof}
\section{ The Structure of Fatou set}\label{ch1,sec6}
\begin{definition}\label{d8}
\textbf{(Connectivity of a domain)}
The connectivity of a domain $D$ denoted by $c(D)$ is defined as the number of components of boundary $\partial{D}$ of $D$. If $D$ is simply connected, $c(D)=1.$ $D$ has finite connectivity n, if $D^c$, the complement of $D$ has exactly n components. $D$ has infinite connectivity if  $D^c$ has infinitely many components.
\end{definition}
\begin{proposition}\label{1.6.2}
A domain $D$ is simply connected if and only if its boundary $\partial D$ is connected.
\end{proposition}
\begin{proof}
$D$ is simply connected if and only if number of components of $\partial D$ is 1, i.e if and only if  $\partial D$ is connected.
\end{proof}
Let $\R$ be a rational map with $deg\,\R= d \geqslant 2$. We will study the structure of Fatou set $\F$ by studying the completely invariant components of $\F.$ We state a result (see~\cite{beardon}).
\begin{result}\label{1.6.3}
Let $D$ be an open subset of $\ti\C.$ Then  $\ti\C\smallsetminus D$ is connected if and only if each component of $D$ is simply connected.
\end{result} 
As a consequence of this result we have :
\begin{result}\label{1.6.4}
Let $\R$ be a rational map. Then $\J$ is connected if and only if each component of $\F$ is simply connected.
\end{result}
\begin{theorem}\label{1.6.5}
Let $\R$ be a rational map with $deg\,\R= d \geqslant 2$ and $F_0$ be a completely invariant component of $\F.$ Then
\begin{enumerate}
\item\ $\partial F_0=\J,$
\item\ $F_0$ is either simply connected or infinitely connected,
\item\ all other components of $\F$ are simply connected, and
\item\ $F_0$ is simply connected if and only if $\J$ is connected.
\end{enumerate}
\end{theorem}
\begin{proof}
\begin{enumerate}
\item\ Since $F_0$ is completely invariant, so is $\overline{F_0}$. By minimality of $\J$, we have $\J\subset\overline{F_0}\,$ i.e $\J\subset F_0\cup\partial F_0.$ Since $\J\nsubseteq F_0,$ so $\J\subset\partial F_0.$\\
For the reverse inclusion,
 let $z\in\partial F_0$ and $U$ be a neighbourhood of $z.$ Then $U$ contains points of $F_0$ and $F_0^{c}.$ Any limit of iterates $\{\R^n\vert_U\}$ has a jump discontinuity between $F_0$ and $\partial F_0$ (otherwise $F_0$ will contain points which are outside $F_0$, a contradiction). Therefore $z\notin\F$ and so $z\in\J.$ As $ z\in\partial F_0$ was arbitrary, therefore $\partial F_0\subset\J.\;$ Hence $\partial F_0=\J.$
\item\ If $F_0$ is infinitely connected, our assertion gets proved, so suppose $F_0$ is not infinitely connected. Suppose $c(F_0)=c,$ and let $E_1,\ldots,E_c$ be the components of $F_0^c.$ By Corollary \ref{1.3.21}, $F_0^c$ is completely invariant and so we can find an integer $m>0$ such that each $E_j$ is completely invariant under $\R^m.$ By Theorem \ref{1.5.12}, $\J$ is infinite and so one of $E_j$ say $E_1$ is infinite. By the minimality of $J(\R^m)$ we have $\J=J(\R^m)\subset E_1$ (see Remark \ref{1.5.14} and Theorem 1.3.19). Thus each $E_j$ intersects $\J\;$(since $\partial F_0=\J),$ so that $ E_1=E_2=\cdots=E_c.$ Thus $c=1$ and hence $F_0$ is simply connected.
\item\ We have $\overline{F_0}= F_0\cup\partial F_0=F_0\cup\J\;(\text{by part}\,(1)).$ Since $F_0$ is connected, so is $\overline{F_0}$. By Result \ref{1.6.3}, the components of $(\overline F_0)^c$ are simply connected. But these components are just the components of $\F$ other than $F_0.$ Hence all components of $\F$ other than $F_0$ are simply connected.
\item\ By Proposition \ref{1.6.2}, $F_0$ is simply connected if and only if $\partial F_0$  is connected, i.e if and only if $\J$ is connected (by part (1)).
\end{enumerate}
\end{proof}
We state the following  theorems which we will use (see~\cite{beardon}).
\begin{theorem}\label{1.6.6}
 Let $F_0$ and $ F_1$ be components of the Fatou set $\F$ of a rational map $\R$ and suppose that $\R$ maps $F_0$ into $F_1$.Then for some integer $m,\;\R$ is an $m$-fold map of $F_0$ onto $F_1$ and 
\begin{equation}\label{eq8}
 \chi(F_0) + \delta_{\R}(F_0) = m \chi (F_1) 
\end{equation}
where $\chi (F_i)\; i=1,2 $ is the Euler characteristic of $F_i$.\\ 
Expression (\ref{eq8}) is called \emph{Riemann - Hurwitz relation}.
\end{theorem}
\begin{theorem}\label{t0}
With the same assumptions as in above theorem, $c(F_0)\geqslant c(F_1).$
\end{theorem}
We now compute the number of components of the Fatou set\;$\F$.\\
Let $\R$ be a rational map with $deg\,\R=d\geqslant 2.$ Suppose first that $F_1,\ldots,F_k$ are completely invariant components of $\F.$ Applying Theorem \ref{1.6.5} \,part (3) to each of the components $F_j$ in turn, we get that every  $F_j$ is simply connected and therefore  $\chi (F_j)=1$ for each j. Applying Theorem \ref{1.6.6} to each $F_j,$ we get\\
$\chi(F_j) + \delta_{\R}(F_j)=m\chi (F_j),\;\Rightarrow  \delta_{\R}(F_j)
=(d-1) \chi(F_j)$ ($m=d,$ since $F_j$ is completely invariant),\;
$\Rightarrow \delta_{\R}(F_j)
=d-1$ ($\chi(F_j) =1$).
Thus we find that if $D$ is a simply connected completely invariant component of $\F$, then the total deficiency of the map $\R : D\to D$ is $d-1.$ Then
\begin{equation}
\begin{split}
\notag
\sum_{j=1}^{k}\delta_{\R}(F_j)
&=\delta_{\R}(F_1)+\delta_{\R}(F_2)+\cdots+\delta_{\R}(F_k)\\
&=d-1+d-1+\cdots+d-1\\
&=k(d-1).\\
\Rightarrow k(d-1)
&=\sum_{j=1}^{k}\delta_{\R}(F_j)\\
&\leqslant \delta_{\R}(\ti\C)\\
&=2d-2\quad(\text{total deficiency of}\;\R\;\text{is always}\;2d-2\;\text{by Theorem \ref{1.5.21}}).\\
\Rightarrow (k-2)\,d
&\leqslant k-2.\\
\Rightarrow k
&\leqslant 2\quad(k>2,\;\Rightarrow d\leqslant 1,\;\text{a contradiction}).
\end{split}
\end{equation}
Thus we conclude that $\F$ contains atmost two completely invariant components.
\begin{corollary}\label{1.6.7}
Let $\R$ be a rational map with $deg\,\R= d \geqslant 2$. Then the Fatou set  contains atmost two different completely invariant simply connected components.
\end{corollary}
\begin{proof}
From above theorem,  if $D$ is a simply connected completely invariant component of $\F$, then $\delta_{\R}(D)=d-1.$ So if there are more than two different simply connected completely invariant components, then total deficiency of $\R$ would exceed $2d-2$ which is not possible. Hence  $\F$ contains atmost two different simply connected completely invariant components.
\end{proof}
\begin{theorem}
Let $\R$ be a rational map with $deg\,\R=d\geqslant 2$. Then the number of components of the Fatou set can be $0, 1, 2$ or $\ity,$ and all cases occur.
\end{theorem}
\begin{proof}
Suppose $\F$ has finitely many components, say $F_1,\ldots,F_k.$ We can find an integer $m$ such that each $F_j$ is completely invariant under $\R^m$ (having same Fatou set as $\R$ by Theorem \ref{1.3.19}), and by Corollary \ref{1.6.7} (applied to $\R^m$) we get $k\leqslant 2.$
\end{proof}
We now illustrate through examples that all the cases in the above theorem occurs.
\begin{example}
 If $\R(z)=z^2-2,$ then by Example \ref{eg2} we have $\J= [-2,\,2]$ and $\F=\ti\C\smallsetminus [-2,\,2] $. So $\F$ has a single completely invariant component.
\end{example}
\begin{example}
 If $\R(z)=z^2,$ then by Example \ref{eg3}, $\J$ = unit circle and $\F=\{z : |z|<1\}\cup \{z : |z|>1\}.$ Thus $\F$ has two completely invariant components.
\end{example}
\begin{example} 
\emph{(Lattes map)}\;It is defined as $\R(z)=\dfrac{(z^2+1)^2}{4z(z^2-1)}\cdot$ This map has empty Fatou set, (see \textbf{Blanchard}~\cite{blanchard}) and hence $\F$ has no completely invariant component.
\end{example}
\begin{example}
 Consider $\R(z)=\la z+z^2.$ If $\la,\;|\la|=1$ is chosen as such to correspond to a \emph{Siegel disk} (to be defined in the next section), then we show that $\F$ has infinitely many components.\\
Let $F_0$ and $F_\ity$ be the components containing $0$ and $\ity$ respectively. Then $\R$ is conjugate to an irrational rotation of $F_0$, and there is an analytic homeomorphism $\phi : F_0\to \D$ such that the following diagram
\[
\begin{CD}
{F_0} @ >\R>>  {F_0}\\
@VV\phi V @VV\phi V\\
\D @>z\to \la z>>\D
\end{CD}
\] 
commutes, with $\phi(0)=0.$  Also there are no critical points of $\R$ in $F_0\;$(otherwise it will contradict the fact that $\phi$ is one-one). Therefore $\R$ is one-one on $F_0.$ Thus $F_0$ has a distinct preimage $F_1$ which is also distinct from $F_{\ity}.$ Hence $\F$ has infinitely many components.
\end{example}
\section{Classification of Forward Invariant Components}\label{ch1,sec7}
Here we give a complete analysis of the forward invariant components of the Fatou set of a rational map $\R$ with $deg\,\R=d\geqslant 2.$ The image of any component of $\F$ under $\R$ is again a component of $\F$ and the inverse image of any component of $\F$ is the disjoint union of atmost $d$ components.
\begin{definition}
A component $U$ of $\F$ is
\begin{enumerate}
\item\ periodic if for some positive integer n, $\R^{n}(U) = U$;
\item\ eventually periodic if for some positive integer $n$, $ \R^n(U)$ is periodic;
\item\ wandering if the sets $\R^n(U), n \geqslant 0 $ are pairwise disjoint.
\end{enumerate}
\end{definition}
\begin{definition}\label{d9}
 A forward invariant component $F_0$ of $\F$ is
\begin{enumerate}
\item\ an attracting component if it contains an  attracting fixed point $z_0$ of $\R$;
\item\ a superattracting component if it contains a superattracting fixed point $z_0$ of $\R$;
\item\ a parabolic component if there is a rationally indifferent fixed point $z_0$ of $\R$ on the boundary of $F_0$ and if $ \R^n \rightarrow z_0$ on $F_0$;
\item\ a \emph{Siegel disk} if $\R : F_0 \rightarrow F_0$ is analytically conjugate to an irrational rotation of the unit disk $ \D$ onto  itself;
\item\ a \emph{Herman ring} if $\R : F_0 \rightarrow F_0 $ is analytically conjugate to an irrational rotation of some annulus onto itself.
\end{enumerate}
\end{definition}
One of the problems which had defied solution for over 60 years is the following theorem, finally solved by \emph{Sullivan}. In the next chapter we shall give its complete proof.
\begin{theorem}
[\textbf{Sullivan}]
 Every component of the Fatou set of a rational map $\R$ is eventually periodic. Thus a rational function has no wandering domains.
\end{theorem}
The corresponding statement for entire maps is false and the first constructive proof of an entire function having wandering domains was given by \textbf{I. N. Baker} (see ~\cite{baker}).\\
The following example is taken from \textbf{S. Zakeri} (see~\cite{zakeri}).
 \begin{example}
The map $z \rightarrow z + \sin(2\pi z)$ has wandering domains.
\end{example}
Having ruled out the possibility of wandering domains, we now know that every component of the Fatou set is either periodic or preperiodic.\\
 We now state the Classification theorem for the forward invariant components of $\F$ (see~\cite{beardon},~\cite{carleson and gamelin}).
\begin{theorem}
 Let $\R$ be a rational map with $deg\,\R=d\geqslant 2.$ A forward invariant component of $\F$ is one of the types (1) to (5) in Definition \ref{d9}.
\end{theorem}


\chapter[Sullivan's proof]{Making Sullivan Readable}\label{ch2}

\section{Direct and Inverse Limits}\label{ch2,sec1}
A directed set $\mathcal A$ is a partially ordered set ($\mathcal A,\leqslant$) such that  for $\alpha,\beta \in\mathcal A, \; \exists \; \gamma \in \mathcal A$ with $\alpha \leqslant \gamma$ and $\beta \leqslant \gamma.$
\begin{example} \label{eg4}
 The collection of all subsets of a set S, partially ordered by inclusion. (i.e A $\leqslant$ B if A $\subset$ B)
\end{example}
\begin{example}\label{eg5}
The set of positive integers with the natural ordering $\leqslant$, is a directed set.
\end{example}
For each pair $\alpha,\beta$ with $\alpha \leqslant  \beta,$ assume that there is  a continuous map $\phi_{{\alpha}{\beta}} : Y_{\alpha} \rightarrow Y_{\beta}$  such that  $\alpha \leqslant  \beta \leqslant \gamma \Rightarrow$ 
$\phi_{{\alpha}{\gamma}} = \phi_{{\beta}{\gamma}} \circ \phi_{{\alpha}{\beta}}$. Then the family $\{Y_{\alpha}; \phi_{{\alpha}{\beta}}\}  $ of spaces $(Y_{\alpha})$ and connecting maps $\phi_{{\alpha}{\beta}}$ is called a direct  spectrum over  $\mathcal A$. The image of a $y_{\alpha} \in Y_{\alpha}$ under any connecting map is termed a successor of $y_{\alpha}.$\\
The set $X\cup Y$ is the {\it{free union}}  $X+Y$ of disjoint spaces $X,Y$ and we topologise it as :  $U\subset X+Y$ is open if and only if $U\cap X$ is open in $X$ and $U\cap Y$ is open in $Y.$\\
Each direct spectrum $\{Y_{\alpha}; \phi_{{\alpha}{\beta}}\}  $ gives a limit space in the following manner.\\
Let $ D = \Sigma (Y_{\alpha} | \alpha \in \mathcal A)$ be the free union of spaces. We say that $y_{\alpha} \in Y_{\alpha},$ $y_{\beta} \in Y_{\beta}$ in $D$ are equivalent  if  they have a common successor in the spectrum. It can be easily seen to be an equivalence relation.\\
 The quotient space $\Sigma {Y_{\alpha}}/{\sim}$ is called direct limit space of the spectrum and is denoted by $Y^{\ity}.$ \\
We illustrate it with a trivial example.
\begin{example}
Let $\mathcal A$ be a directed set and $Y$ be any space.  The direct spectrum $\{Y_{\alpha};\phi_{{\alpha}{\beta}}\}$ in which  $Y_{\alpha} = Y \;\forall\; {\alpha} \in \mathcal A$ and $\phi_{{\alpha}{\beta}} = \text{Identity} \;\forall\;{\alpha},{\beta} \in\mathcal A$ is called the trivial direct spectrum over $\mathcal A$ with space $Y.$ In this case $Y^{\ity}  \simeq Y.$ 
\end{example}
Let $\A$ be a directed set and $\mathfrak{a}=\{A_{\alpha} : \alpha\in\A \}$ be a collection of sets. Suppose that for $\alpha \leqslant  \beta$ in $\A$ there is a map $f_{{\beta}{\alpha}} : A_{\beta}\to \A_{\alpha}$ such that
\begin{enumerate}
\item\  $f_{{\beta}{\alpha}}\circ f_{{\gamma}{\beta}}=f_{{\gamma}{\alpha}},\quad \alpha \leqslant  \beta \leqslant \gamma $
\item\  $f_{{\alpha}{\alpha}}=I,\, \text{the identity}\;\forall\;\alpha\in\A.$
\end{enumerate}
Let $\mathcal{F}$ be the collection $\{f_{{\alpha}{\beta}}\}$ of all such transformations. The pair $\{\mathfrak{a},\mathcal{F}\}$ is called an \emph{inverse limit system} over the directed set $\A.$ We give two particular instances of such systems. \\
In the first instance, $\mathfrak{a}$ consists of topological spaces and each $f_{{\alpha}{\beta}}$ in  $\mathcal{F}$ is a continuous mapping. Consider the product space $\prod_{\alpha\in\A}A_{\alpha}$ and projection maps $\pi_{\beta} : \prod_{\alpha\in\A}A_{\alpha}\to A_{\beta}$ defined as
\begin{equation}
\pi_{\beta}(\{x_{\alpha}\})=x_{\beta}.
\end{equation}
The \emph{inverse limit space} $A_\ity$ of the system $\{\mathfrak{a},\mathcal{F}\}$ is defined as
\begin{equation}\label{eq9}
A_\ity=\B\{\{x_\alpha\}\in\prod_{\alpha\in\A}A_{\alpha} : f_{{\beta}{\alpha}}\pi_\beta(\{x_\alpha\})=x_\alpha\;\text{whenever}\;\alpha\leqslant\beta\B\}.
\end{equation}
$ A_\ity$ is a subspace of the product space $\prod_{\alpha\in\A}A_{\alpha}.$\\
In the second instance, $\mathfrak{a}$ consists of topological groups (see Appendix \ref{A}), and each $f_{{\alpha}{\beta}}$ in  $\mathcal{F}$ is a continuous homomorphism. The inverse limit\;group  $A_\ity$ of the system $\{\mathfrak{a},\mathcal{F}\}$ is defined as
\begin{equation}\label{eq10}
A_\ity=\B\{\{x_\alpha\}\in\prod_{\alpha\in\A}A_{\alpha} : f_{{\beta}{\alpha}}\pi_\beta(\{x_\alpha\})=x_\alpha\;\text{whenever}\;\alpha\leqslant\beta\B\}.
\end{equation}
$ A_\ity$ is a subgroup of the direct product group $\prod_{\alpha\in\A}A_{\alpha}.$ The binary operation in $A_\ity$  is defined by
\begin{equation}
\{x_\alpha\}\{y_\alpha\}=\{x_\alpha y_\alpha\},
\end{equation}
where the \emph{dot} on the right denotes the binary operation in $A_\alpha.$ Note that the element $\{e_\alpha :\alpha\in\A\}\in A_\ity,$ where $e_\alpha$ is the identity of $A_\alpha\;\forall\;\alpha\in\A.$ Thus the existence of $A_\ity$ is always ensured. It can be easily seen that $A_\ity$ forms a group under this binary relation.
\begin{definition}\label{d10}
[\textbf{Cantor group}]
It is an inverse limit of finite groups. Any topology on this group is discrete, therefore it is a totally disconnected group i.e each component is a single point.\\
For further details on direct and inverse limits, one can refer ~\cite{dugundji},~\cite{hocking and young}. 
\end{definition}
\section{ Riemann Surfaces}\label{ch2,sec2}
An $n-$dimensional manifold is a Hausdorff topological space $X$ such that every point  $a\in$ $X$ has an open neighbourhood which is homeomorphic to an open subset of ${\mathbb R}^{n}$.\\
 Let $X$ be a $1-$dimensional complex manifold. A complex chart on $X$ is a homeomorphism $\phi : U\rightarrow V$ of an open subset $U\subset $ X onto an open subset $V\subset {\C}$. Two complex charts $\phi_{i} : U_{i}\rightarrow V_{i} ,\, i = 1, 2$ are said to be holomorphically compatible if the transition map $\phi_{2}\circ \phi_{1}^{-1} : \phi_{1}(U_{1}\cap U_{2}) \rightarrow \phi_{2}(U_{1}\cap U_{2})$ is biholomorphic.\\
A complex atlas on $X$ is a system $\mathfrak{U} = \{\phi_{i} : U_{i} \rightarrow V_{i} ,\, i\in \mathcal I\}$ of charts which are holomorphically compatible and which cover $X$, i.e $\bigcup _{i\in \mathcal I}U_{i} = X.$\\
    Two complex atlases $\mathfrak{U}$ and $\mathfrak{U'}$ on $X$ are called analytically equivalent if every chart of $\mathfrak{U}$ is holomorphically compatible with every chart of $\mathfrak{U'}$.\\
 By a complex structure on a $1-$dimensional complex manifold $X$, we mean an equivalence class of analytically equivalent atlases on $X.$
\begin{definition}\label{d11}
A Riemann surface is a pair ($X,\Sigma$) where $X$ is a connected $1-$dimensio-nal complex manifold and $\Sigma$ is a complex structure on $X.$
One usually writes $X$ instead of ($X,\Sigma$) whenever it is clear which complex structure is meant. Hence a Riemann surface is a $1-$dimension complex analytic manifold.\\
We illustrate with examples.
 \end{definition}

\begin{example}
 [\textbf{The  complex plane ${\C}$}] The complex structure is defined by the atlas whose only chart is the identity map $I : {\C}\rightarrow {\C}.$
\end{example}
\begin{example}
 [\textbf{Domains}] Suppose $X$ is a Riemann surface and $Y\subset X$ is a domain i.e a non-empty connected open set. Then $Y$ has a natural complex structure which makes it a Riemann surface.The atlas consists of complex charts $\phi : U\rightarrow V$ on $X,$ with $U\subset Y.$ In particular, every domain $Y\subset {\C}$ is a Riemann surface.
\end{example}
\begin{example}
 [\textbf{The Riemann Sphere $\ti{\C}$}] We know that the topology on $\ti{\C}$ consists of the usual open sets $U\subset {\C}$ together with sets of the form $V\cup \{\ity\}$ where $V\subset {\C}$ is the complement of a compact set $K\subset {\C}$.\\
Let $U_1 = \ti{\C}\smallsetminus {\ity} = {\C}$ and\\
$U_2 = \ti{\C}\smallsetminus\{0\} = {\C^{\ast}}\cup\{\ity\}.$\\
We define maps $\phi_i : U_i\rightarrow {\C},\,i = 1, 2$ as follows :\\
$\phi_1$ is the identity map and
\begin{equation}
\notag
\phi_{2}(z) =
\begin{cases}
1/z  &\text{if z $\in {\C}^{\ast},$  and}\\
0   &\text{if z =$\ity$}
\end{cases}
\end{equation}
It can be easily checked that $\phi_1$ and $\phi_2$ are homeomorphisms. Since  $U_1$ and $U_2$ are connected and have non- empty intersection, $\ti{\C}$ is also connected.\\
The complex structure on $\ti{\C}$ is defined by the atlas consisting of the charts $\phi_i : U_i\rightarrow {\C},\, i =1, 2.$ For $\ti{\C}$ to be a Riemann surface we have to show that the two charts are holomorphically compatible and the transition map biholomorphic. But this is clear since $\phi_1(U_1\cap U_2) = \phi_2(U_1\cap U_2) = {\C^{\ast}}$\\
and $\phi_2 \circ\phi_1^{-1} :{\C^{\ast}}\rightarrow{\C^{\ast}},\, z\rightarrow 1/z$ is biholomorphic.\\
We now define holomorphic mappings between Riemann surfaces.
\end{example}
\begin{definition}\label{d12}
 Let $X$ and $Y$ be Riemann surfaces. A continuous mapping $f : X\rightarrow Y $ is  holomorphic if for every pair of charts $\psi_1 : U_1\rightarrow V_1 \;$ on $X$ and $\psi_2 : U_2\rightarrow V_2$ on $Y$ with $f(U_1)\subset U_2,$ the mapping $\psi_2\circ f\circ\psi_1^{-1} : V_1\rightarrow V_2$ is holomorphic in the usual sense.
\end{definition}
\section{Direct limits of Riemann Surfaces}\label{ch2,sec3}
Let
\[
\begin{CD}
\ R_{1} @>f_{1}>> \ R_{2} @>f_{2}>> \ R_{3} @>f_{3}>>\cdots
\end{CD}
\]
be a sequence of surjective analytic maps of Riemann surfaces $R_i.$ The Riemann surface $R^{\ity}$ represents the direct limit of the sequence $\{f_n\}$ if there are compatible analytic surjective maps  \[
\begin{CD}
\ R_{n} @>\pi_{n}>> \ R^ {\ity}
\end{CD}
\]
namely $\pi_{n+1} \circ f_n = \pi_n$,  so that some $\pi_n$ identifies two points if and only if they are identified in the sequence $\pi_{n}(x) = \pi_{n}(y)$ if and only if $g_{k}(x) = g_{k}(y)$ for some k where $g_{k} = \prod_{i=0}^{k}{f_{n+i}}$.
We illustrate this with an example.
\begin{example}
 Consider the direct system
\[
\begin{CD}
\ \C @>z\to z>> \C @>z\to e^z>> \C\smallsetminus \{0\} @>z\to z>> \C\smallsetminus \{0\} @>z\to z>>\cdots
\end{CD}
\]
Here $f_n(z)=z,\;n\neq 2$ and $f_2(z)=e^z.$\\
Define $\pi_n :\C\smallsetminus \{0\}\to\C\smallsetminus \{0\}$ as $\pi_n(z)=z,\;n\neq1, 2$;\\
 $\pi_1 :\C\to\C\smallsetminus \{0\}$ as $\pi_1(z)=e^z,$ and $\pi_2 :\C\to\C\smallsetminus\{0\}$ as $\pi_2(z)=e^z.$ Then $\;\forall\;n,\;\pi_n$ are analytic surjective maps.\\
For $n\geqslant 3,$ we have\\
$\pi_{n+1}\circ f_n(z)=z=\pi_n(z).$ For $n=1,\;\pi_2\circ f_1(z)=\pi_2(z)=e^z=\pi_1(z),\;\Rightarrow \pi_2\circ f_1(z)=\pi_1(z).$\\
For $n=2,\; \pi_3\circ f_2(z)=\pi_3(e^z)=e^z=\pi_2(z).$ Hence $\pi_{n+1}\circ f_n=\pi_n\;\forall\;n.$\\
Now for $n=2,\; \pi_n$ identifies $0$ and $2\pi i$ since $\pi_2(0)=e^0=1$ and $\pi_2(2\pi i)=e^{2\pi i}=1,\;\Rightarrow \pi_2(0)=\pi_2(2\pi i).$\\
We show that $g_k(0)=g_k(2\pi i)$ for some $k.$\\
Consider $k=1$ and $n=2.$ Then $g_1(0)=  \prod_{i=0}^{1}{f_{n+i}}=f_3\circ f_2(0)=f_3(1)=1$ and $g_1(2\pi i)=f_3\circ f_2(2\pi i)=f_3(1)=1$. Thus $g_1(0)=g_1(2\pi i).$ Hence the direct limit of the direct system given above is $\C\smallsetminus\{0\}.$
\end{example}
\begin{definition}\label{d13}
\textbf{(Hyperbolic Riemann Surface)} 
 A Riemann surface is hyperbolic if its universal covering space is conformally equivalent to the unit disk ${\D}$.\\
We illustrate with an example.
\end{definition}
\begin{example}
Every component of the Fatou set ${\F}$ of a rational map ${\R}$ is hyperbolic.\\
Let $U$ be a component of ${\F}$. $U$ has a natural complex structure which makes it a Riemann surface, viz. the atlas consists of  complex charts $\phi : V\rightarrow {\C}$ on $\ti{\C}$ with $V\subset U$.\\
{\underline{Reason for $U$ being hyperbolic :}}  Since ${\J}$ has atleast 3 points, $U$ misses atleast 3 points and by Uniformization theorem (see Appendix \ref{A}), the universal cover of $U$ is conformally equivalent to the unit disk ${\D}$. Therefore $U$ is hyperbolic.
\end{example}
\begin{definition}
 A Riemann surface is non-elementary if the fundamental group ( i.e the covering group acting on the disk) is not abelian.
\end{definition}
\begin{example}
 Consider the complex plane $\C$ and remove two disks $D_1$ and $D_2$ from it. The resulting space $\C\smallsetminus \{D_1,D_2\}$ is homeomorphic to \emph{figure eight space}. It is a hyperbolic Riemann surface since its universal covering space is conformally equivalent to $\D$ (see~\cite{milnor}).\,The fundamental group of figure eight space is a \emph{free group} on two generators, (see~\cite{massey}) hence it is non-abelian. Since homeomorphic spaces have isomorphic fundamental groups, therefore fundamental group of the space $\C\smallsetminus \{D_1,D_2\}$ is isomorphic to free group on two generators. Hence $\C\smallsetminus \{D_1,D_2\}$ is non-elementary.
\end{example}
We suppose that the Riemann surfaces $R_i$ are hyperbolic.
\begin{proposition}\label{2.3.6}
 If the maps $f_1, f_2,\ldots, $ are unbranched coverings and if one of the $R_i$ is non elementary then the direct limit $R^{\ity}$ exists. Furthermore, either
\begin{enumerate}
\item\ the fundamental group of $R^{\ity}$ is not finitely generated or
\item\ eventually the $f_i $ are isomorphisms.
\end{enumerate}
\end{proposition}
\begin{proof}
Note : For the topological facts used in the proof, see Appendix \ref{A}.\\
 Since the $R_{i}$'s are hyperbolic, therefore the universal covering surface of  $R_{i}$'s is conformally equivalent to the disk$\;{\D}$. Let $\pi : {\D}\rightarrow R_1$ be the universal cover of $R_1$.\\
We claim that  the composition 
\[
\begin{CD}
\D @>\pi>> \ R_1 @>f_{1}>> \ R_{2} @>f_{2}>>\cdots\ R_{n-1} @>f_{n-1}>> \ R_{n}
\end{CD}
\]
determines the universal covering of $R_{n}$.\\
We first prove that
\[
\begin{CD}
\D @>\pi>> \ R_1 @>f_{1}>> \ R_{2} 
\end{CD}
\]
determines the universal covering of $R_2$.

\begin{diagram}
\D &\rTo^{\pi} &R_{1} \\
&\rdTo &\dTo_{f_1} \\
& &R_{2}
\end{diagram}

 Let $x_0 \in R_2$ and $U$ be an open neighbourhood of $x_0$. Let $x$ be a point in the fibre over $x_0$. Since $f_1$ is a covering map, therefore for the neighbourhood $U$ of  $x_0.$ Then\,$\,\exists\,$ a disjoint collection $(U_{\alpha} : \alpha \in \Lambda)$ of open subsets of $R_1$ such that $f_1^{-1}(U) = \bigcup_{\alpha}U_{\alpha}$ and $f_{1}\vert U_{\alpha} : U_{\alpha}\rightarrow U$ is a homeomorphism. Let $U_{\alpha}$ be the sheet containing $x$. Since $\pi$ is a covering map, therefore for the open neighbourhood $U_{\alpha}$ of $x$, we get a disjoint collection $(V_{\beta} : \beta \in \Gamma)$ of  open subsets of $\D \;$  such that $\pi^{-1}(U_{\alpha}) = \bigcup_{\beta}V_{\beta}$ and $\pi \vert V_{\beta} : V_{\beta}\rightarrow U_{\alpha}$ is a homeomorphism. Now $(f_1 \circ \pi)^{-1}(U) = \pi^{-1}\circ f_1^{-1}(U) = \pi^{-1}(\bigcup_{\alpha}U_{\alpha}) = \bigcup_{\alpha}\pi^{-1}(U_{\alpha}) = \bigcup_{\alpha}\bigcup_{\beta}V_{\beta}$. \\
Thus each point $x_0 \in R_2$ has a neighbourhood $U$ such that $(f_1 \circ \pi)^{-1}(U) =  \bigcup_{\alpha}\bigcup_{\beta}V_{\beta}$. Since $f_1$ and $\pi$ are surjective so is $f_1 \circ \pi$. Since $f_1$ and $\pi$ are local homeomorphism, so is $f_1\circ \pi$. Hence $f_1\circ \pi$ is a covering map. Since $\D$ is simply connected, therefore $f_1\circ \pi$ is the universal covering of $R_2$.\\
Proceeding by induction we get that the composition 
\[
\begin{CD}
\D @>\pi>> \ R_1 @>f_{1}>> \ R_{2} @>f_{2}>>\cdots\ R_{n-1} @>f_{n-1}>> \ R_{n}
\end{CD}
\]
$f_{n-1}\circ\cdots\circ f_1\circ\pi : \D\rightarrow R_n$ determines the universal covering of $R_n$ which proves the claim.\\
We further claim that $\Gamma_{1}\subset \Gamma_2\subset\ldots$ , where $\Gamma_{i}$'s are fundamental group (covering group acting on the disk) of $R_{i}$'s.\\
We first show that $\Gamma_{1}\subset\Gamma_{2}.$\\
Let $h\in \Gamma_{1},\;\Rightarrow \; \pi\circ h = \pi.$ 
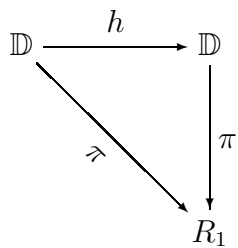
\begin{diagram}
\D &\rTo^{h} &\D \\
&\rdTo_{\pi}&\dTo_{\pi} \\
& &R_{1}
\end{diagram}
Now $(f_{1}\circ\pi)\circ h = f_{1}\circ(\pi\circ h) = f_{1}\circ\pi$ and so $ h\in\Gamma_{2}.$ Since $h\in \Gamma_{1}$ was arbitrary, therefore $\Gamma_{1}\subset \Gamma_{2}.$
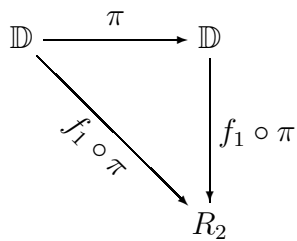
\begin{diagram}
\D &\rTo^{\pi} &\D \\
&\rdTo_{f_{1}\circ \pi}&\dTo_{f_{1}\circ \pi} \\
& &R_{2}
\end{diagram}
Similarly it can be seen on similar lines that $\Gamma_{2}\subset\Gamma_{3}\subset\Gamma_{4}\subset\ldots$ Thus we get an increasing union of covering groups $\Gamma_{1}\subset\Gamma_{2}\subset\Gamma_{3}\subset\ldots$
Let
\begin{equation}\label{eq11}
\Gamma_{\ity} = \bigcup_{i=1}^{\ity}\Gamma_{i}
\end{equation}
As the covering maps are unbranched, therefore $\Gamma_{\ity}$ contains no elliptics. Since one of the $R_i$ is non-elementary, therefore $\Gamma_{\ity}$ is non-abelian$\;(\Gamma_{i}\subset\Gamma_{\ity})$. We know that a subgroup $\Gamma$ of $PSL(2,\mathbb R)$ is either elementary, contains elliptics or is discrete (see~\cite{katok}), and therefore $\Gamma_{\ity}$ is a discrete group. As the direct limit of Riemann surfaces is a Riemann surface, therefore $R^\ity$ is a Riemann surface having $\D$ as its universal covering surface and $\Gamma_{\ity}$ as fundamental group. Hence $\D\diagup \Gamma_{\ity}\cong R^\ity$ represents the direct limit.\\
If the fundamental group of $R^\ity$\;viz.$\;\Gamma_{\ity}$ is not finitely generated, then the assertion gets proved.\\
 So suppose that $\Gamma_{\ity}$  is finitely generated. Then $\exists\;m_0\in\mathbb N$ such that $\Gamma_{\ity}=\Gamma_n\;\forall\;n>m_0.$ Now  $\D\diagup \Gamma_{n}\cong R_n,\;\D\diagup \Gamma_{n+1}\cong R_{n+1},\ldots,$  implies that $R_n\cong R_{n+1},\;R_{n+1}\cong R_{n+2},\ldots$  Hence eventually the $f_i$ are isomorphisms.
\end{proof}
\section{ The case of the wandering annulus}\label{ch2,sec4}
Let $A_0$ be a stable region (Fatou component) of the rational map ${\R}$ which is an annulus. Define $A_{n+1} = {\R}(A_n),\, n= 0,1,2,\ldots$  We assume ${\R} : A_n \rightarrow A_{n+1}$ is a covering map and  the degree of ${\R} : A_n\rightarrow A_{n+1}$ is $d_n$. 
\begin{remark}\label{2.4.1}
 The image of an annulus $A_0$ under a covering map is again an annulus. This follows from the  Riemann\,-\,Hurwitz relation expression (\ref{eq8}).\\
 As $\partial{A_0}$ has two components, therefore $\chi(A_0) = 2 - 2 = 0$ (see Appendix \ref{A}).
If $D$ is a disk then $\chi(D) = 2 - 1 = 1\;$( disk is simply connected therefore $\partial{D}$ has one component). Had the image been a disk, then by Theorem \ref{t0},   $0\geqslant 1$ a contradiction. Thus the image cannot be a disk. Also, the image cannot be any domain with negative Euler characteristic, since it violates expression (\ref{eq8}). Hence the image has to be an annulus only.
\end{remark}
With notations as before we have :
\begin{proposition}\label{2.4.2}
 If the $A_k$ are pairwise disjoint, then for n sufficiently large $d_n = 1$.
\end{proposition}
\begin{proof}
 We shall prove it by a contradiction. The proof is spread into 4 parts.
\begin{enumerate}
\item[(i)] Let $n_1<n_2<\ldots,$ be a sequence of indices so that $d_{{n_{i}}}>1$. Since $\R : A_{{n_{i}}}\rightarrow A_{{n_{i+1}}}$ is a covering map, therefore there are no critical points of $\R$ in $A_{n_{i}}$ and so  the critical points lie in the components of the complement of $A_{{n_{1}}},A_{{n_{2}}},\ldots$ Since the $A_{{{n_{i}}}}$'s are pairwise disjoint, therefore the complement of say $A_{{n_{1}}},$ contains $A_{{n_{2}}},A_{{n_{3}}},\ldots,$ which are infinitely many. Similarly, the complement of $A_{{n_{2}}}$ contains $A_{{n_{1}}},A_{{n_{3}}},\ldots,$ and  so on. From Theorem \ref{1.5.21}, the number of critical points of $\R$ is atmost $2\,deg\,\R-2,$ so there is one critical point of $\R$ around which infinitely many of the $ A_{{n_{i}}}$  nests.\\
Recall that if $X$ is a connected space and $A\subset X,$ we say that $A$ separates $X$ if $X\smallsetminus A$ is not connected.
\item[(ii)]  Let $S$ be a smooth closed curve imbedded in the annulus $A_0$ which separates the boundary components of $A_0.$ Since $S\subset A_0$ and $A_0$ is a Fatou component, therefore $(\R^n\vert_S)$ forms an equicontinuous family and hence the arc length of $\R^{n}(S)$ is bounded in the spherical metric ( for otherwise $\R^{n}(S)$ need no longer be equicontinuous). One knows from topology that the winding number of a closed curve under a covering map from one space onto the other is the degree of the map. Since a covering map is a local homeomorphism, therefore it is a closed map, and hence $\R^{n}(S)$ is closed in  $A_n$. So the winding number of $\R^{n}(S)$ in $A_n$ is the product $d_0\,d_1\ldots d_{n-1}.$
\[
\begin{CD}
\ A_0 @>\R>> \ A_1 @>\R>> \ A_{2} @>\R>>\cdots\ A_{n-1} @>\R>> \ A_{n}
\end{CD}
\]
\[
\begin{CD}
\ A_0 @>\R^{n} >> \ A_{n}
\end{CD}
\]
Thus the winding number becomes arbitrarily large with n. In order that the arc length of $\R^{n}(S)$ be bounded in  $A_n$, one component of the complement of  $A_n$ must have arbitrarily small spherical diameter for sufficiently large n.
\item[(iii)]  One knows (see~\cite{dieudonne}) from Differential Calculus that $\R'(z)$ is a continuous linear mapping of $\ti{\C},$  and so $\R'(z)$ is bounded on $\ti{\C}.$ Thus $\, \exists \, M>0$ such that $|\R'(z)|\leqslant M.$ Let ${C_n}$ be the small complementary component of $A_n$ having arbitrarily small spherical diameter for sufficiently large n. We shall see that for this large n, ${C_n}$ maps under $\R$ to the small complementary component ${C_{n+1}}$ of $A_{n+1}$.\\
 Let $z_0\in C_n$. Since $\R$ is analytic at $z_0,$ so for some neighbourhood $D$ of $z_0, \R$ has a Taylor expansion
$\R(z)=\R(z_0)+\R'(z_0)(z-z_0)\;(\text {neglecting higher order terms})$
and so $|\R(z) - \R(z_0)| =|\R'(z_0)||z - z_0|$
\begin{equation}\begin{split}\notag
&\leqslant M|z-z_0|\\
&\leqslant M\sup_{z,z_0 \in C_n}|z - z_0|\\
&\leqslant M \frac{\epsilon}{M}\;\;(\text{diameter of }C_{n}\text{ is arbitrarily small,} \text{\;so it can be made less than }\frac{\epsilon}{M})\\
&\leqslant \epsilon.
\end{split}
\end{equation}
$\Rightarrow\sup_{\xi\in C_{n+1}}|\R(z) - \R(z_0)|\leqslant \epsilon$ ( any point $\xi \in  C_{n+1}$ is of the form $\R(z)$ for some $z\in C_{n}$).
Thus, for $n$ sufficiently large, the image of ${C_n} $ under ${\R}$ does not explode onto the sphere. This property will be called non-explosion property.
\item[(iv)]  Now some power $k$ of $\R$ carries one member $A_{n_{0}}$ of the nested family of annuli described in (i) to some other member $A_m$ with the non-explosion property of (iii). Let $z_0$ be a point on the small boundary component $C_{n_{0}}$ of $A_{n_{0}}.$ By a consequence of Corollary \ref{1.5.8} $C_{n_{0}}\subset J(\R^{k}) = {\J}$ (by Theorem \ref{1.3.19}) and so $z_0 \in {\J}$.
Also $\{\R^{{k}{n}}:\; n\in \mathbb{N}\}$ is an equicontinuous family at  $z_0$ (by  non-explosion property), therefore $z_0 \in F(\R^k) = {\F}$ (by Theorem \ref{1.3.19})
and so $z_0 \in {\F}\cap {\J}$, a contradiction. The contradiction arose because we had assumed that $\exists\,$ a sequence of indices $n_1<n_2<\ldots,$ with $d_{n_{i}}>1.$   Therefore  for $n$ sufficiently large $d_n = 1$ and this completes the proof.
\end{enumerate}
\end{proof}
\section{Riemann surfaces of infinite topological type}\label{ch2,sec5}
Let $X$ be a hyperbolic Riemann surface provided with its complete hyperbolic\,   (Poinca-r$\acute{e}$) metric (see Appendix \ref{A}). \emph{Geodesics} in $X$ are the shortest curves with respect to this metric. Equivalently, a geodesic is a curve which locally minimizes distance between pair of points.
We illustrate with examples.
\begin{example}
 Let $\mathbb H$ denotes the upper half plane. The geodesics in $\mathbb H$ are \emph{semicircles} and \emph{straight lines} orthogonal to the real axis $\mathbb R.$
\end{example}
\begin{example} 
Geodesics in the Riemann sphere $\ti\C$ are \emph{great circles}.
\end{example}

\begin{definition}
A subset $S\subset X$ is called \emph{geodesically convex} if any two of its points can be joined by a unique geodesic lying in $S.$
\end{definition}
Consider a maximal collection $\mathcal C$ of disjoint simple closed geodesics in $X.$ It is known that the closure of each component of $X\smallsetminus\mathcal C$ is geodesically convex and is conformally equivalent to sphere minus
\begin{enumerate}
\item\ two disks :\qquad \emph{``\,funnel"}
\item\ three disks :\qquad \emph{``\,pair of pants"}
\item\ two disks and one point :\qquad \emph{``\,degenerated pair of pants"}
\item\ one disk and two points :\qquad \emph{``\,degenerated pair of pants"}
\item\ three points :\qquad \emph{``\,degenerated pair of pants"}
\end{enumerate}
The basic idea behind this decomposition into five basic pieces is that whenever some component of $X\smallsetminus\mathcal C$ is not conformal to one of the basic pieces, then one could add a new simple closed geodesic to the collection $\mathcal C$, which contradicts the maximality of $\mathcal C$. 
\begin{definition}
A hyperbolic Riemann surface is of \emph{infinite topological type} if and only if $\mathcal C$ is infinite.\\
When one wants to recover the conformal structure of $X$ from its decomposition into the five basic pieces, it is sufficient to give the hyperbolic lengths of the boundary curves of each piece and the only condition required to build the hyperbolic surface is the equality of the lengths of curves to be glued together and a twist parameter at each gluing curve.\\
As a consequence of this we state the following proposition :
\end{definition}
\begin{proposition}\label{2.5.5}
If $X$ is a hyperbolic Riemann surface of infinite topological type, then the space of deformations of the hyperbolic structures of $X$ is infinite dimensional.
\end{proposition}
For a more detailed discussion one can refer \textbf{Thurston} ~\cite{thurston}.
\section{The structure of a wandering stable region}\label{ch2,sec6}
Let $U_0$ be a wandering stable region (Fatou component). Let $U_1$ be the image of $U_0$ under ${\R}, U_2$ be the image of $U_1$ under ${\R}$ and so on. As $U_0$ is a wandering stable region, therefore $U_{i}$'s are pairwise disjoint. By Theorem \ref{1.5.21}, ${\R} $ has only finitely many critical points, so let $U_0,U_1,\ldots,U_k$ be the domains containing the critical points. We may discard finitely many of these $U_{i}$'s to eliminate critical points of ${\R} $. After relabling the  $U_{i}$'s, each 
\begin{diagram}
U_{n} &\rTo^{\R} &U_{n+1} 
\end{diagram}
is a finite unbranched covering ( as global degree of ${\R}$ is finite ).
\begin{proposition}\label{2.6.1}
\begin{enumerate} 
Either
\item\ from some n onwards,  $U_{n+i}$ has finite topological type and each 
\begin{diagram}
U_{n+i} &\rTo^{\R} &U_{n+i+1} 
\end{diagram}
is an isomorphism $i\geqslant 0$, or
\item\  the direct limit $U^{\ity} $ of 
\[
\begin{CD}
\ U_0 @>\R>> \ U_1 @>\R>> \ U_{2} @>\R>>\cdots
\end{CD}
\]
exists and has infinite topological type.
\end{enumerate}
\end{proposition}
\begin{proof}
 If the direct limit $U^{\ity} $ of
\[
\begin{CD}
\ U_0 @>\R>> \ U_1 @>\R>> \ U_{2} @>\R>>\cdots
\end{CD}
\]
exists and has infinite topological type, then the assertion gets proved.\\
So suppose this is not the case. We shall see that the other case holds. The proof is divided into 3 subcases.
\begin{enumerate}
\item[(i)] Suppose some $U_n$ is simply connected. 
\begin{diagram}
U_{n} &\rTo^{\R} &U_{n+1} 
\end{diagram}
From Riemann\,-\,Hurwitz relation (\ref{eq8}), we have $\chi(U_n) + \delta_{\R}(U_n) = m\,\chi(U_{n+1})$, where $m\geqslant 1$ is a finite positive integer (as $ deg\,\R$ is finite). Then $\chi(U_n)=2-1=1$. Since ${\R}$ is an unbranched covering map, therefore $\delta_{\R}(U_n) = 0$ and so $1=m\,\chi(U_{n+1})$. This is possible if and only if m=1 and $\chi(U_{n+1})=1$. So $U_{n+1}$ is simply connected. 
Since $m=1,$ so ${\R}$ is injective. Hence
\begin{diagram}
U_{n} &\rTo^{\R} &U_{n+1} 
\end{diagram}
is an isomorphism.\\
Similarly, it can be seen that all subsequent $U_{n+i},\,i\geqslant 0$ are simply connected and 
\begin{diagram}
U_{n+i} &\rTo^{\R} &U_{n+i+1} 
\end{diagram}
is an isomorphism.
\item[(ii)]  Suppose some $U_n$ is an annulus. From Riemann\,-\,Hurwitz relation (\ref{eq8}), we have $\chi(U_n) + \delta_{\R}(U_n) =k\,\chi(U_{n+1})$, where $k\geqslant 1$ is a finite positive integer ($deg\,\R$ is finite). Then $\chi(U_n)=2-2=0.$ Since ${\R}$ is an unbranched covering map, therefore $\delta_{\R}(U_n) = 0$ and so
$ 0=k\,\chi(U_{n+1})$. This is possible if and only if $\chi(U_{n+1})=0$ which implies that $U_{n+1}$ is also an annulus.\\
On a similar pattern, it can be seen that all subsequent $ U_{n+i},\, i\geqslant 0$ are annuli.\\
As seen in Section \ref{ch2,sec4}, we get that, eventually the maps 
\begin{diagram}
U_{n+i} &\rTo^{\R} &U_{n+i+1} 
\end{diagram}
are isomorphisms.\\
One knows (see ~\cite{hubbard}) that a hyperbolic Riemann surface is of finite type if and only if it is either of genus 0 with atleast 3 points removed, or of genus 1 with at least 1 point removed, or of genus more than 1 with any finite number of points removed.\\
In subcase (i), we had seen that from some $ n$ onwards, the Fatou components $U_{n+i},\, i\geqslant 0$ are simply connected and hence of genus 0. Also each $U_{n+i},\,i\geqslant 0$ misses atleast 3 points  ( as ${\J}$ contains atleast 3 points). Therefore $U_{n+i},\, i\geqslant 0$ has finite topological type.\\
In subcase (ii) we had seen that from some $ n$ onwards, the Fatou components $U_{n+i},\, i\geqslant 0$ are annuli and hence of genus 1. Also each of  $U_{n+i},\, i\geqslant 0$ misses atleast 1 point (infact atleast 3 points), therefore $U_{n+i},\, i\geqslant 0$ has finite topological type.
\item[(iii)] In this case we consider Fatou components of genus $g > 1$. We will consider the case $g =2$ and the case $g>2$ will follow by induction.\\
 We shall see that a genus 2 surface is mapped to a genus 2 surface under a covering map.\\
Let $U_n$ be of genus 2. So $c(U_n)=4$ ( there are 4 components of $\partial{U_n},$ the boundary of $U_n$) and so $\chi(U_n)=2-4= -2$. From Riemann\,-\,Hurwitz relation (\ref{eq8}), $\chi(U_n) + \delta_{\R}(U_n) =k\,\chi(U_{n+1})$, for some $k\geqslant 1,$ 
$\;\Rightarrow -2 = k\,\chi(U_{n+1})$ and so the only possibility for $k$ is 1 or 2.\\
If $k=1,$ then $\chi(U_{n+1})= -2$ and this proves our assertion.\\
Suppose $k=2,\;\Rightarrow \chi(U_{n+1})= -1,\;\Rightarrow 2 - c(U_{n+1})= -1,\;\Rightarrow c(U_{n+1})=3$ and so the number of components of $\partial{U_{n+1}}=3.$ Since $\R$ is a covering map which is a local homeomorphism, therefore ${\R}$ is  conformal. But under a conformal map boundary goes to boundary, and as $\partial{U_{n}}$ has 4 components, we arrive at a contradiction. Hence $k=1$ and  $\chi(U_{n+1})= -2$. So
\begin{diagram}
U_{n} &\rTo^{\R} &U_{n+1} 
\end{diagram}
is an isomorphism. Proceeding on similar lines we get
\begin{diagram}
U_{n+i} &\rTo^{\R} &U_{n+i+1} 
\end{diagram}
is an isomorphism and $U_{n+i}, \,i\geqslant 0$ has finite topological type.
\end{enumerate}
\end{proof}
\section{Rigidity of conjugacy on the Julia set}\label{ch2,sec7}
Let ${\R}$ be a rational mapping. Consider the set 
\begin{equation}\label{eq11}
C_{\R}= \{f : {\J}\rightarrow {\J}\; | \; f\,\text{is a homeomorphism}\;\text{and}\; f\circ {\R}={\R}\circ f\}.
\end{equation}
 $ f\circ {\R}$ is well defined, since ${\J}$ is completely invariant under $\R,$ therefore ${\R}\,({\J})={\J}$.\\ 
We assert that $C_{\R}$ is a subgroup under composition of mappings.\\
Since Identity, $I\in C_{\R}\; \Rightarrow C_{\R}\neq\emptyset$.
Let $f,g\in C_{\R},\;\Rightarrow \; f\circ{\R}={\R}\circ f.$\\
To show : $f\circ g^{-1}\in C_{\R}.$\\
Since compose of two homeomorphisms is again a homeomorphism, therefore $f\circ g^{-1}$ is a  homeomorphism. Also 
\begin{equation}
\begin{split}
\notag
(f\circ g^{-1})\circ {\R}
&=f\circ(g^{-1}\circ {\R})\\
&=f\circ({\R}\circ{g^{-1}})\; (g\;\text{ is a homeomorphism of}\; {\J}\;\text{ so is}\; g^{-1}) \\
&=({\R}\circ f)\circ{g^{-1}}\\
&={\R}\circ(f\circ g^{-1})
\end{split}
\end{equation}
$\Rightarrow f\circ g^{-1}$ commutes with ${\R}$. Hence $ f\circ g^{-1} \in C_{\R}$. Thus $ C_{\R}$ is a subgroup and this proves our assertion.\\
Let G be the group of homeomorphisms of ${\J}$ with the compact open topology. G then becomes a topological group, for the mapping $f\circ g\to fg^{-1}$ is continuous \;$\forall \;f, g\in G$.\\
We show  that with the compact open topology on G,  $ C_{\R}$ is a closed subspace of G.\\
Recall that (see ~\cite{munkres}) if X is a space and (Y,d) is a metric space, then the function space $Y^X$ has the following inclusion of topologies :
\begin{equation}
\emph{Uniform}\supset \emph{Compact convergence}\supset \emph{Pointwise convergence}.
\end{equation}
 If X is compact the first two coincide, and if X is discrete, the second two coincide. Also, on $\mathcal C(X,Y),$ the set of all continuous maps from X to Y, the compact open topology, and the topology of compact convergence coincide. 
Since ${\J}$ is a compact space therefore the uniform topology on G coincides with the topology of compact convergence which further coincides with compact  open topology.\\ Let $(f_{n})$ be a sequence in $ C_{\R}$ with $f_{n}\rightarrow f$. Since uniform limit of a sequence of continuous functions is continuous, so $f$ is continuous. Since G is a topological group, therefore  $x\rightarrow x^{-1}$ is continuous, and so $f\rightarrow f^{-1}$ is continuous. Hence $f$ is a homeomorphism. Also $f_n \circ {\R}={\R}\circ f_n \; \forall\; n\in \mathbb{N},\;\Rightarrow \lim_{n\rightarrow {\ity}}(f_n \circ {\R})=\lim_{n\rightarrow{\ity}}({\R}\circ f_n)$, so that by continuity we get 
$(\lim_{n\rightarrow{\ity}}f_n)\circ {\R}={\R}(\lim_{n\rightarrow{\ity}}f_n),$
$\;\Rightarrow f\circ{\R}={\R}\circ f$. 
Thus we have shown that $f$ is a homeomorphism of ${\J}$ and commutes with ${\R}$. Hence $f\in C_{\R}$. Thus in the compact open topology on G, $ C_{\R}$ is a closed subspace.
\begin{proposition}\label{2.7.1}
$ C_{\R}\;$ (with the induced compact open topology) admits a continuous injective homomorphism into a totally disconnected Cantor group (an inverse limit of finite groups).
\end{proposition}
\begin{proof}
 Let $A_1$ be a finite set consisting of periodic points in ${\J}$ with lowest possible period (elements in $A_1$ are of same period ). Define by induction $A_{n+1}={\R}^{-1}(A_n)\cup A_n \;\forall\; n\in \mathbb{N}.$ Therefore $A_{n} \subset A_{n+1}\; \forall\; n \in \mathbb{N}.$  Let $g\in C_{\R}$, then $g{\R}={\R}g$.\\
 We observe that $g{A_n}={A_n}$.\\
We first show that $g{A_1}={A_1}.$\\
Let $x_0\in A_1,\;\Rightarrow {\R^{n}{x_0}}=x_0\quad$(say $A_1$ consists of periodic points of period n). Then
\begin{equation}
\begin{split}
\notag
g{x_0}
&=g{\R^{n}{x_0}}\\
&=(g{\R})({\R}^{n-1}{x_0})\\
&={\R}g\,({\R}^{n-1}{x_0})\quad (g{\R}={\R}g)\\
&={\R}\,(g{\R})\,({\R}^{n-2}{x_0})\\
&={\R^2}g\,({\R}^{n-2}{x_0})\\
&=\cdots\\
&= \cdots\\
&=\cdots\\
&={\R^n}g{x_0}\\
\Rightarrow g{x_0}
&\in A_1.
\end{split}
\end{equation}
Since $x_0 \in A_1$ was arbitrary therefore $gA_1\subset A_1.$\\
Conversely,
 we show that $A_1\subset gA_1\; i.e \; g^{-1}(A_1)\subset A_1$ ( $g$ is a homeomorphism).\\
Let $y_0\in g^{-1}(A_1),\;\Rightarrow g(y_0)\in A_1,\;\Rightarrow {\R}^{n}(g(y_0))=g(y_0),\;\Rightarrow g{\R}^{n}(y_0)=g(y_0)\;(\text{since}\;g{\R}^{n}={\R}^{n}g),\;
\Rightarrow {\R}^{n}(y_0)=y_0\;( g\; \text{is a homeomorphism and hence one-one})$
so that $y_0\in A_1.$ Since $ y_0\in g^{-1}(A_1)\; \text{was arbitray},$ therefore $ g^{-1}(A_1)\subset A_1.$ Hence $g{A_1}={A_1}$. By induction we have 
 $g{A_n}={A_n} \;\forall\; n\in \mathbb{N}$, which proves the assertion.\\
 Let $A=\bigcup_{n\in\mathbb N} A_n.$ From Corollary \ref{1.5.11}, the iterated pre-images of any point $z_0\in\J$ is dense in $\J,$ therefore $A$ is dense in $\J.$ Since $A_n$ is finite $\;\forall\; n\in \mathbb{N},$ therefore $g$ is a permutation of  $A_n$. As $g\in  C_{\R}$ was arbitrary, therefore $ C_{\R}$ is a subgroup of the group of permutations of $A_n\;\forall\;n\in\mathbb N.$ Let $B_m$ be the permutation group corresponding to each $A_m$ and let $\mathfrak a=\{B_n :n\in\mathbb N\}$ be the collection of permutation groups. For $m\geqslant n$ we have seen that $A_m\supset A_n.$ Define for  $m\geqslant n$
\begin{equation}
\mu_{{m}{n}} : B_m\to B_n
\end{equation}
as \,  $\mu_{{m}{n}}(f)=f\vert_{A_n}\;\;\forall\;f\in B_m.$\\
 For $l\geqslant m\geqslant n,$ we show that $\mu_{{m}{n}}\mu_{{l}{m}}=\mu_{{l}{n}}.$\\
$\forall\;g\in B_l$
\begin{equation}
\begin{split}
\notag
\mu_{{m}{n}}\mu_{{l}{m}}(g)
&=\mu_{{m}{n}}(g\vert_{A_m})\\
&=(g\vert_{A_m})\vert_{A_n}\\
&=g\vert_{A_n}\quad((g\vert_{A_m})\vert_{A_n}=g\vert_{A_n})\\
&=\mu_{{l}{n}}(g).\\
\Rightarrow \mu_{{m}{n}}\mu_{{l}{m}}
&=\mu_{{l}{n}}.
\end{split}
\end{equation}
Clearly, $\mu_{{n}{n}}=I,\;\forall\;n\in\mathbb N.$ Let $\mathcal F$ be the collection $\{\mu_{{m}{n}}\}$ of all such maps. Thus the pair $\{\mathfrak a,\mathcal F\}$ forms an inverse limit system over the directed set $\mathbb N$ (see Section~\ref{ch2,sec1}).\\
We claim that the inverse limit\;$B_\ity$ is
\begin{equation}
B_\ity=\B\{(f_1,f_2,\ldots) : \forall\;m,n\in\mathbb N,\;m\geqslant n\;\Rightarrow f_m\vert_{A_n}=f_n\B\}.
\end{equation}
Let $\pi_m :\prod_{n\in\mathbb N}  B_n\to B_m$ be the natural projection map. Then 
\begin{equation}
\mu_{{m}{n}}\pi_m(f_1,f_2,\ldots)=f_n,\;\forall\;(f_1,f_2,\ldots)\in\prod_{n\in\mathbb N}  B_n
\end{equation}
Suppose $f_m\vert_{A_n}\neq f_n,\;\Rightarrow\;\forall\;(f_1,f_2,\ldots)\in\prod_{n\in\mathbb N}  B_n$
\begin{equation}
\begin{split}
\notag
\mu_{{m}{n}}\pi_m(f_1,f_2,\ldots)
&=f_m\vert_{A_n}\\
&\neq f_n\\
\Rightarrow \mu_{{m}{n}}\pi_m(f_1,f_2,\ldots)
&\neq f_n\\
\Rightarrow (f_1,f_2,\ldots)
\notin B_\ity.
\end{split}
\end{equation}
Hence $ B_\ity=\B\{(f_1,f_2,\ldots) : \forall\;m,n\in\mathbb N,\;m\geqslant n\;\Rightarrow f_m\vert_{A_n}=f_n\B\}$ and it is the required Cantor group.\\
We now show that there is an injective homomorphism $\phi : C_R\to B_\ity$ which is continuous.\\
 Define
 $\phi : C_\R\to B_\ity$ as 
$\phi(f)=(f_1,f_2,\ldots)\;\;\forall\;f\in C_\R,$ such that $f_n=f\vert_{A_n}\;\forall\;n\in\mathbb N.$ Since $A_n\subset\J\;\forall\;n\in\mathbb N,$ therefore $f_n=f\vert_{A_n}\;\forall\;n\in\mathbb N$ are well defined. Clearly, $(f_1,f_2,\ldots )\in B_\ity$ since $(f_m\vert_{A_n})=(f\vert_{A_m})\vert_{A_n}=f\vert_{A_n}=f_n\;$ for $m\geqslant n.$ Hence $\phi$ is well defined.\\
\underline{$\phi$ is injective}\\
Let $f,g\in C_\R$ such that 
\begin{equation}
\begin{split}
\notag
\phi(f)
&=\phi(g)\\
\Rightarrow (f_1,f_2,\ldots)
&=(g_1,g_2,\ldots)\\
\Rightarrow f\vert_{A_n}
&=g\vert_{A_n}\;\;\forall\;n\in\mathbb N\\
\Rightarrow f\vert_{\cup A_n}
&=g\vert_{\cup A_n}\;\;\forall\;n\in\mathbb N\\
\Rightarrow f\vert_A
&=g\vert_A\quad(A=\cup_{n\in\mathbb N}\, A_n)\\
\Rightarrow f
&=g\quad(A\;\text{is dense in}\;\J\;\text{and}\;f,g\;\text{are continuous}).
\end{split}
\end{equation}
Hence $\phi$ is injective.\\
\underline{$\phi$ is a group homomorphism}\\
$\forall\;f,g\in C_\R$
\begin{equation}
\begin{split}
\notag
\phi(fg)
&=\B((fg)\vert_{A_1},\,(fg)\vert_{A_2}),\ldots\B)\\
&=\B(f\vert_{A_1}g\vert_{A_1},\,f\vert_{A_2}g\vert_{A_2},\ldots\B)\quad((fg)\vert_{A_n}=f\vert_{A_n}g\vert_{A_n})\\
&=\B(f\vert_{A_1},\,f\vert_{A_2},\ldots\B)\B(g\vert_{A_1},\,g\vert_{A_2},\ldots\B)\quad(\text{definition of binary operation in} \;B_\ity)\\
&=\phi(f)\phi(g).
\end{split}
\end{equation}
Hence $\phi$ is a group homomorphism.\\
\underline{$\phi$ is continuous}\\
We show that inverse image under $\phi$ of every closed set in $B_\ity$ is closed in $C_\R.\\
$ Since $B_\ity$ is a totally disconnected space, therefore $\{g\}\;\;\forall\;g\in B_\ity\;$ where $g=(g_1,g_2,\ldots)$ and $g_n=g\vert_{A_n},$ is closed  in $B_\ity$. Since a subspace of a Hausdorff space is Hausdorff, therefore $\J\subset\ti\C$ is Hausdorff. We know that the function space $C(X,Y)$ with the compact  open topology is Hausdorff if $Y$ is so, hence $C_\R=C(\J,\J)$ is Hausdorff. Therefore 1 point sets are closed in $C_\R.$ Since $\phi$ is one-one, therefore $\phi^{-1}(g_1,g_2,\ldots)=\{f\}$ for exactly one $f\in C_\R.$ Hence $\phi$ is continuous, and this completes the proof.
\end{proof} 
\begin{corollary}
 A topological conjugacy between two rational maps is unique on the Julia set upto to an element in a totally disconnected group of homeomorphisms.
\end{corollary}
\begin{proof}
Let $R_1$ and $R_2$ be two rational maps and $\phi$ and $\psi$ be topological conjugacies between them. Therefore $\phi R_{1}\phi^{-1}=R_2$ and $\psi R_1\psi^{-1}=R_2,\;\Rightarrow \phi R_{1}\phi^{-1}=\psi R_1\psi^{-1},\;\Rightarrow \psi^{-1}\phi R_1=R_1\psi^{-1}\phi,\; \Rightarrow \psi^{-1}\phi\in C_{R_1}$ and so $\psi^{-1}\phi=g$ for some $g\in C_{R_1}$ by above theorem, so that $\phi=\psi g$, and this completes the proof.
\end{proof}
\section{Prime ends}\label{ch2,sec8}
Throughout this section $U$ will denote a simply connected domain in $\ti\C$ with boundary $\partial U.$
\begin{definition}
[\textbf{Cross cut}]
A cross cut in $U$ is a simple arc $c$ in $U$ whose end points lie in $\partial U.$
\end{definition}
\begin{example}
 Let W be any bounded subset of $\C$. Let $L$ be a line segment in W and extend it in both directions till it hits $\partial W,$ then $L$ is a cross cut in W.
\end{example}
\begin{definition}\label{d14}
 If $X$ is a connected space and $A\subset X$, $A$ \emph{separates} $X$ if $X\smallsetminus A$ is not connected. $A$ separates $X$ into $n$ components if  $X\smallsetminus A$ has $n$ components.
\end{definition}
\begin{definition}\label{d15}
Let $X$ be a topological space and $\mathfrak C$ be a collection of subsets of $X$. $\mathfrak C$ has \emph{finite intersection property}, if for every finite subcollection $\{C_1,\ldots,C_n\}$ of $\mathfrak C$, $\bigcap_{i=1}^{n}C_i$ is non-empty.
\end{definition}
\begin{example}
 Let $X$ be a topological space and $\mathfrak C$ be a collection of nested closed sets. If each of the sets in this collection is non-empty, then it can seen easily that $\mathfrak C$ has the finite intersection property.
\end{example}
 We now state the following theorems from topology (see ~\cite{munkres}).
\begin{theorem}
[\textbf{Jordan separation theorem}]
Let $C$ be a simple closed curve in $S^2$, the \emph{unit 2 - sphere} in $\mathbb R^3.$ Then $C$ $separates\;S^2.$
\end{theorem}
\begin{theorem}\label{2.8.7}
Let $X$ be a topological space. $X$ is compact if and only if for every collection $\mathfrak F$ of closed sets in $X$ having the finite intersection property,\;$\bigcap_{F\in\mathfrak F}F$ is non-empty.
\end{theorem}
\begin{theorem}\label{2.8.8}
Let $X$ be a topological space and let $C_1\supset C_2\supset \ldots$ be a nested sequence of compact connected sets. Then $\bigcap_{n\in\mathbb N} C_n$ is connected.
\end{theorem}
\begin{lemma}\label{2.8.9}
Any cross cut $c$ divides $U$ into two  components.
\end{lemma}
\begin{proof}
From \emph{quotient topology} we know that $\overline{\D}\,\diagdown S^1\cong S^2,$ where $S^1$ is the unit circle which is the boundary of $\,\overline{\D}$ and  $S^2$ is unit 2 - sphere in $\mathbb R^3.$ By Riemann mapping theorem, there is a conformal isomorphism $\phi$ from $\D$ onto $U$, so that $\overline{U}\,\diagdown\partial U\cong S^2.$ Therefore $c$ corresponds to a Jordan curve in this quotient sphere. By Jordan separation theorem,\;$c$  divides $U$ into two components.
\end{proof}
\begin{definition}
A \emph{chain} is a sequence of cross cuts $c_n$ such that
\begin{enumerate}
\item\ $diameter\;c_n\to 0\quad$(spherical metric),
\item\ the closures of $c_n$ are disjoint,
\item\ (relative to a fixed base point $b$ in $U$), $c_n$ separates $c_{n+1}$ from $b$.
\end{enumerate}
\end{definition}
We now define a relation among chains :\\
 Two chains\;$\{c_n\}$ and $\{\bar c_n\}$ are equivalent, if each cross cut $c_n$ separates all but finitely many $\bar c_n$ from $b$ and vice\,-\,versa. It can be easily seen that this relation is an equivalence relation.
\begin{example} 
Consider $\D$, the unit disk. Let $\{c_n\}$ and $\{\bar c_n\},$ where $c_n$ and $\bar c_n$ are line segments whose end points lie on $\partial\D$, be the chains at a distance $\dfrac{1}{2^n}$ and  $\dfrac{1}{3^n}$ respectively from the top of the disk. Fix any point $b$ in the disk which is at a distance more than $\dfrac{1}{2}$ from the top  of the disk. Then it can be seen easily that the two chains are equivalent.
\end{example}
\begin{definition}\label{d18}
[\textbf{Prime end}]
A prime end of $U$ is an equivalence class of chains.\\
Let $P$ be a prime end of $U$, $\{c_n\}$ be a chain which belongs to  $P$ and for each n let $B_n$ denote that subdomain of $U$ determined by $c_n$ which does not contain the fixed point $b.$ The set $I(P)=\bigcap_{n}\overline{B_n}$ is called the $impression$ of $P.$ Let $V$ be another simply connected domain with boundary $\partial V.$ By a consequence of Riemann mapping theorem, there is a conformal isomorphism $\phi$  from $U$ onto $V.$ \textbf{Caratheodory} in his principal theorem on the correspondence between boundaries under conformal mappings showed that the conformal isomorphism $\phi$ determines an isomorphism between the set of prime ends of $U$ and $V$. For the unit disk $\D$, by Riemann mapping theorem,\;$\exists$ a conformal isomorphism $\psi$ from $\D$ onto $U,$ Caratheodory showed that  $\psi$ induces a one-to-one correspondence between the points of unit circle and the prime ends of $U$ (see \textbf{Piranian}~\cite{piranian}).
\end{definition}
\begin{proposition}\label{2.8.13}
The impression $I(P)$ of a prime end\;$P$ is a non-empty compact connected subset of $\partial U.$
\end{proposition}
\begin{proof}
We have $I(P)=\bigcap_{n}\overline{B_n}$. The collection $\mathfrak C=\{\overline{B_n} : n\in\mathbb N\}$ forms a nested sequence of closed sets in $\overline{U}$ which is a closed and hence compact subspace of $\ti\C.$ Also each $\overline{B_n}$ is non-empty, therefore  the collection  $\mathfrak C$ has the finite intersection property. Hence by Theorem \ref{2.8.7}, $\bigcap_{n}\overline{B_n}$ is non-empty. Thus $I(P)$ is non-empty. Since each $B_n$ is connected, so is each $\overline{B_n}\;$( In a Hausdorff space, closure of a connected set is connected). Also each $ \overline{B_n}$ is compact. Using Theorem \ref{2.8.8} we get that $I(P)$ is connected.
\end{proof}
\begin{proposition}\label{2.8.14}
$I(P)\subset\partial U.$
\end{proposition}
\begin{proof}
Let $z\in U.$ We can find a $j$ such that $z\notin B_j.$ Let $z_0\in U\smallsetminus B_1$ and $p : [0,1]\to U$ be a path in $U$ joining $z_0$ to $z$. Since continuous image of a compact set is compact, therefore $p[0,1]$ is compact. Let $P=p[0,1]$. If 
\begin{equation}\label{eq12}
\delta=dist\{P,\partial U\}
\end{equation}
be the distance between $P$ and $\partial U,$
and $j$ is chosen to be sufficiently large so that $diameter\,c_j<\delta$, then 
\begin{equation}\label{eq13}
c_j\cap P=\emptyset
\end{equation}
so that $c_j$ does not separate $z_0$ from $z$. Also  $z\notin\partial U\;(z\in\partial U,\;\Rightarrow \delta=0$ from (\ref{eq12})). Since $z_0\notin B_1,\;\Rightarrow z_0\notin B_j\;\;\forall\;j\in\mathbb N\;\;(B_j\subset B_1\;\;
\forall\;j\in\mathbb N)$ which will imply that $z\notin B_j\;\;\forall\;j\in\mathbb N,$ i.e $z\notin \overline{B_j}\;\;\forall\;j\in\mathbb N$ and so $z\notin\bigcap_j\overline{B_j}.$  Thus we conclude that $\bigcap_j\overline{B_j}\subset\partial U.$ Hence $I(P)\subset\partial U.$
\end{proof}
\begin{proposition}
$I(P)$ consists of a single point $z_0\in\partial U$ if and only if diameter\,$\overline{B_j}$ tends to 0 as $j\to\ity.$
\end{proposition}
\begin{proof}
As $\overline{U}$ is compact space, therefore it is complete as a metric space (with spherical metric induced from $\ti\C$). So $I(P)$ is singleton set if and only if diameter\,$\overline{B_j}$ tends to 0 as $j\to\ity$ ( by \emph{Cantor's Intersection theorem} ).
\end{proof}
\begin{definition}\label{d16}
The \emph{fibre} of a point $x\in\partial U$ is the collection of prime ends represented by chains\;$\{c_n\}$ such that $c_n\to x$ as $n\to\ity.$
\end{definition}
\begin{proposition}\label{2.8.17}
Any prime end belongs to atleast one fibre (of some point in its impression). 
\end{proposition}
\begin{proof}
Let $P$ be a prime end represented by chain\;$\{c_n\}.$ Let $x\in I(P)=\bigcap_n\overline {B_n}\subset\partial U.$ We know that for each $n, B_n$ is the subdomain of $U$ determined by cross cut $c_n$ which does not contain the fixed point $b$. From this it can be seen easily that $c_n\to x$ as $n\to\ity$ and this completes the proof.
\end{proof}
\begin{definition}\label{d17}
[\textbf{Radial limit}]
A function $f(z),\; z\in\D$ has radial limit $L$ at $\zeta\in\partial\D$ if
\[
f(r\zeta)\to L\;\;\text{as}\; r\uparrow 1.
\]
Now we state two theorems (see ~\cite{milnor}), one due to Fatou and the other due to Fatou; Riesz and Riesz.
\end{definition}
\begin{theorem}\label{2.8.19}
The Riemann map \;$\phi : \D\to U$ has radial limits except on a set $E$ of measure zero.
\end{theorem}
\begin{theorem}
[\textbf{Fatou; Riesz and Riesz}]
For almost every point $e^{i\theta}\in\partial\D,\;\theta\in [0,2\pi],$ the radial line $r\to r e^{i\theta},\;0\leqslant r\leqslant 1$ maps under the Riemann map \;$\phi$ to a curve of finite spherical length in $U$. In particular, the radial limit
\begin{equation}
\lim_{r\uparrow 1}{\phi(r e^{i\theta})}\;{\in\partial U}
\end{equation}
exists for almost every $\theta\in [0,2\pi]$. However, if we fix any particular point $u_0\in\partial U,$ then the set of $\theta$ such that this radial limit is equal to $u_0$ has Lebesgue measure zero.
\end{theorem}
From the discussion following Definition \ref{d18}, we know that each point of $\partial\D$ corresponds to a prime end of $U$. Also under a conformal isomorphism, boundary goes to boundary, so we conclude that each point of $\partial U$ corresponds to a prime end of $U$ which implies that, fibre of any point in $\partial U$ is discrete, so its angular measure is $0$. We also conclude that the fibres are totally disconnected in the topology on the unit circle.
\section[Quasiconformal Mappings]{The Measurable Riemann mapping theorem with parameters}\label{ch2,sec9}
Here we discuss the notion of quasiconformal mapping following Ahlfors treatment of this topic. There are several approaches to quasiconformal mappings viz. geometric approach, analytic approach, etc. The approach which we have taken is based on the assumption that the mappings are of class $C^1.$ Throughout this section we restrict ourselves to $C^1$ mappings.  Let $w = f(z)\; (z = x +iy, w = u+iv)$ be a $ C^1$ homeomorphism from one region to another. Then
\begin{equation}
dw=f_zdz+f_{\bar z} d{\bar z}
\end{equation}
with $f_z=\dfrac{1}{2}(f_x-if_y),\; f_{\bar z}= \dfrac{1}{2}(f_x+if_y).$ The Jacobian of $f$ is given by  
\begin{equation}\label{eq14}
Jac = |f_{z}|^{2} - |f_{\bar{z}}|^{2}=u_{x}v_{y} -  u_{y}v_{x}.
\end{equation}
 We will denote $Jac$ by $J.$
The Jacobian is positive for sense (orientation) preserving mappings and negative for sense(orientation) reversing mappings. We will consider the sense preserving case i.e $\; |f_{\bar{z}}|< |f_{z}|$.\ The dilatation of $f$ at the point z is given by 
\begin{equation}\label{eq15}
D_{f}(z) = \dfrac{ |f_{z}| + |f_{\bar{z}}|} { |f_{z}| - |f_{\bar{z}}|} \geqslant 1.
\end{equation}
Define $d_{f} = \dfrac{|f_{\bar{z}}|}{|f_{z}|}< 1$, related to $D_{f} $ by  $D_{f} = \dfrac{1 + d_{f}}{1 - d_{f}}$.\ The complex dilatation $\mu_{f}$ is defined as 
\begin{equation}\label{eq16}
\mu_{f} = \dfrac{f_{\bar{z}}}{f_{z}}\cdot
\end{equation}
 We  drop the subscript $f$ from $\mu_{f} $ and simply denote it by $\mu$.\ So
\begin{equation}\label{eq17}
 D_{f}(z) = \dfrac{1 + |\mu(z)|}{1 - |\mu(z)|}\cdot
\end{equation}
\underline{\textbf{Geometrical interpretation of $\mu$ :}} Since $f$ is differentiable in the region say $D,$ $\mu$ can be thought as a field of infinitesimal ellipses on $D$ with minor axis of each ellipse tilted at $\dfrac{1}{2}\,\emph{arg}\,\mu(z)$ and major axis tilted at $\dfrac{\pi}{2}+\dfrac{1}{2}\,\emph{arg}\,\mu(z)$. The derivative map [$Df(z)$] maps the infinitesimal ellipse at $z$ to a circle at $f(z).$ The dilatation $D_{f}(z)$ gives the eccentricity, 
\begin{equation}\label{eq18}
D_{f}(z) =\text{ratio of major to minor axis}= \dfrac{ |f_{z}| + |f_{\bar{z}}|} { |f_{z}| - |f_{\bar{z}}|}= \dfrac{1 + |\mu(z)|}{1 - |\mu(z)|}\geqslant 1\;\;\text{(using (\ref{eq17}))}.
\end{equation}
Since $|\mu|<1,$ therefore eccentricities of the ellipses are bounded.
\begin{definition}\label{d19}
 The mapping $f$ is said to be quasiconformal if $D_{f}$ is bounded. It is $K -$ quasiconformal if $D_{f} \leqslant K$.
\end{definition}
\begin{remark}\label{2.9.2}
A $1-$ quasiconformal map is conformal.
\end{remark}
\begin{proof}
Suppose $f$ is $1-$ quasiconformal, so $D_{f} \leqslant 1.$ Also  from (\ref{eq15}), $D_{f} \geqslant 1$ so that $D_{f}=1,\;\Rightarrow \dfrac{ |f_{z}| + |f_{\bar{z}}|} { |f_{z}| -|f_{\bar{z}}|}=1,\;\Rightarrow  |f_{\bar{z}}|=0$ and so $f$ is analytic. Since $f$ is a homeomorphism, so it is one-one and hence $f$ is conformal.
\end{proof}
We now give an example of a $K-$ quasiconformal map.
\begin{example}
 Let $U\subset\C$ be a domain. Define $f:U\to\C$ as $f(z)=Kx+iy,$ where $K\geqslant 1$ and $z=x+iy.$ In terms of $z$ and $\bar z,\; f$ can be written as $f(z)=z+(K-1)\,\dfrac{(z+\bar z)}{2}\cdot$ Then $f_z=1+\dfrac{(K-1)}{2},\;f_{\bar z}=\dfrac{K-1}{2}\cdot$ \;So $f_z+f_{\bar z}=K$ and $f_z-f_{\bar z}=1.$ Therefore $D_{f}(z)=K.$ Hence $f$ is $K-$ quasiconformal.
\end{example}
\subsection{Composed mappings} Let $U$ be a domain in $\ti\C.$ Let $f:U\to f(U)$ and $g: f(U)\to\C$ be $C^1$ homeomorphisms. We shall now determine the complex derivatives and complex dilatations of the composed mapping $g\circ f.$ It can be easily seen that 
\begin{equation}\label{eq19}
\overline{(f_z)}=(\bar f)_{\bar z}\;\;\text{and}
\end{equation}
\begin{equation}\label{eq20}
\overline{(f_{\bar z})}=(\bar f)_z.
\end{equation}
Let $\zeta=f(z)$. Then
\begin{equation}
\begin{split}
\notag
(g\circ f)_z
&=\dfrac{d}{dz}\,(g\circ f)(z)\\
&=\dfrac{d}{dz}\,(g(f(z))\\
&=\dfrac{d}{dz}\,(g(\zeta))\\
&=\dfrac{d}{d\zeta}\,g(\zeta)\,\dfrac{d\zeta}{dz}+\dfrac{d}{d\bar\zeta}\,g(\zeta)\,\dfrac{d\bar\zeta}{dz}\\
&=g_\zeta(\zeta)\,f_z+g_{\bar\zeta}\,(\zeta)\,(\bar f)_z\\
&=g_\zeta\,(f(z))\,f_z+g_{\bar\zeta}\,(f(z))\,(\bar f)_z\\
&=(g_{\zeta}\circ f)\,f_z+(g_{\bar\zeta}\circ f)\,{(\bar f)_z}\\
&=(g_{\zeta}\circ f)\,f_z+(g_{\bar\zeta}\circ f)\,\overline{( f_{\bar z})}\quad(\text{by}\;(\ref{eq20})).\\
\end{split}
\end{equation}
\begin{equation}\label{eq21}
\Rightarrow\; (g\circ f)_z=(g_{\zeta}\circ f)f_z+(g_{\bar\zeta}\circ f)\,\overline{( f_{\bar z})}.
\end{equation}
Now we compute $(g\circ f)_{\bar z}.$
\begin{equation}
\begin{split}
\notag
(g\circ f)_{\bar z}
&=\dfrac{d}{d\bar z}\,(g\circ f)(z)\\
&=\dfrac{d}{d\bar z}\,(g(f(z))\\
&=\dfrac{d}{d\bar z}\,(g(\zeta))\\
&=\dfrac{d}{d\zeta}\,g(\zeta)\,\dfrac{d\zeta}{d\bar z}+\dfrac{d}{d\bar\zeta}\,g(\zeta)\,\dfrac{d\bar\zeta}{d\bar z}\\
&=g_\zeta(\zeta)f_{\bar z}+g_{\bar\zeta}\,(\zeta)(\bar f)_{\bar z}\\
&=g_\zeta(f(z))f_{\bar z}+g_{\bar\zeta}\,(f(z))(\bar f)_{\bar z}\\
&=(g_{\zeta}\circ f)f_{\bar z}+(g_{\bar\zeta}\circ f)(\bar f)_{\bar z}\\
&=(g_{\zeta}\circ f)f_{\bar z}+(g_{\bar\zeta}\circ f)\overline{( f_z)}\quad(\text{by}\;(\ref{eq19})).\\
\end{split}
\end{equation}
\begin{equation}\label{eq22}
\Rightarrow\; (g\circ f)_{\bar z}=(g_{\zeta}\circ f)f_{\bar z}+(g_{\bar\zeta}\circ f)\overline{( f_z)}.
\end{equation}
Multiplying (\ref{eq21}) by $\overline{(f_z)}$ and (\ref{eq22}) by $\overline{(f_{\bar z})}$ and subtracting we get 
\begin{equation}
\begin{split}
\notag
(g\circ f)_z\,\overline{(f_z)}-(g\circ f)_{\bar z}\,\overline{(f_{\bar z})}
&=(g_{\zeta}\circ f)\,(|f_{z}|^{2} - |f_{\bar{z}}|^{2})\\
&=(g_{\zeta}\circ f)\,J\quad(\text{from (\ref{eq14}})).\\
\Rightarrow (g_{\zeta}\circ f)
&=\dfrac{1}{J}\,(\,(g\circ f)_z\,\overline{(f_z)}-(g\circ f)_{\bar z}\,\overline{(f_{\bar z})}\,)\\
&=\dfrac{1}{J}\,((g\circ f)_z\,(\bar f)_{\bar z}-(g\circ f)_{\bar z}\,(\bar f)_z)\;(\text{using}\;(\ref{eq19})\;\text{and}\;(\ref{eq20})).
\end{split}
\end{equation}
\begin{equation}\label{eq23}
\Rightarrow\; (g_{\zeta}\circ f)=\dfrac{1}{J}\,((g\circ f)_z\,(\bar f)_{\bar z}-(g\circ f)_{\bar z}\,(\bar f)_z).
\end{equation}
Again multiplying (\ref{eq21}) by $f_{\bar z}$ and (\ref{eq22}) by $f_z$ and subtracting we get
\begin{equation}
\begin{split}
\notag
(g\circ f)_z\,f_{\bar z}-(g\circ f)_{\bar z}\,f_z
&=(g_{\bar\zeta}\circ f)\,(\overline{(f_{\bar z})}\,f_{\bar z}-\overline{(f_z)}\,f_z)\\
&=-(g_{\bar\zeta}\circ f)\,(|f_{z}|^{2} - |f_{\bar{z}}|^{2})\quad(\text{using}\;(\ref{eq19})\;\text{and}\;(\ref{eq20}))\\
&=-(g_{\bar\zeta}\circ f)\,J\quad(\text{using} (\ref{eq14})).\\
\end{split}
\end{equation}
\begin{equation}\label{eq24}
\Rightarrow\; (g_{\bar\zeta}\circ f)=\dfrac{1}{J}\,((g\circ f)_{\bar z}\,f_z-(g\circ f)_z\, f_{\bar z}).
\end{equation}
Now let $g=f^{-1}.$ Then from (\ref{eq23}) and (\ref{eq24}) we get
\begin{equation}
\begin{split}
\notag
(f^{-1})_{\zeta}\circ f
&=\dfrac{1}{J}\,((I)_z\,(\bar f)_{\bar z}-(I)_{\bar z}\,(\bar f)_z)\quad(f^{-1}\circ f=I)\\
&=\dfrac{1}{J}\,((\bar f)_{\bar z}-0)\quad(I\;\text{is analytic, so}\;(I)_{\bar z}=0)\\
&=\dfrac{1}{J}\,(\bar f)_{\bar z}
\end{split}
\end{equation}
\begin{equation}\label{eq25}
\Rightarrow\;(f^{-1})_{\zeta}\circ f=\dfrac{1}{J}\,(\bar f)_{\bar z}
\end{equation}
and
\begin{equation}
\begin{split}
\notag
(f^{-1})_{\bar\zeta}\circ f
&=\dfrac{1}{J}\,((I)_{\bar z}\,f_z-(I)_z\, f_{\bar z})\\
&=\dfrac{1}{J}\,(-(I)_z\, f_{\bar z})\\
&=-\dfrac{1}{J}\,f_{\bar z}
\end{split}
\end{equation}
\begin{equation}\label{eq26}
\Rightarrow\;(f^{-1})_{\bar\zeta}\circ f=-\dfrac{1}{J}\,f_{\bar z}.
\end{equation}
Now we compute $\mu_{f^{-1}}\circ f.$
\begin{equation}
\begin{split}
\notag
\mu_{f^{-1}}\circ f
&=\dfrac{({f^{-1})}_{\bar\zeta}}{({f^{-1})}_{\zeta}}\,\circ f\quad(\text{by (\ref{eq16})})\\
&=-\dfrac{f_{\bar z}}{(\bar f)_{\bar z}}\quad(\text{using (\ref{eq25}) and (\ref{eq26})})\\
&=-\dfrac{f_{\bar z}}{\overline{(f_z)}}\,\dfrac{f_z}{f_z}\quad(\text{by}\; (\ref{eq19}))\\
&=-\dfrac{f_{\bar z}}{|f_z|^2}\,f_z\\
&=-\dfrac{f_{\bar z}}{f_z}\,\dfrac{(f_z)^2}{|f_z|^2}\\
&=\B(-\B(\dfrac{f_z}{|f_z|}\B)^2\mu{_f}\B)
\end{split}
\end{equation}
\begin{equation}\label{eq27}
\Rightarrow\;\mu_{f^{-1}}\circ f=\B(-\B(\dfrac{f_z}{|f_z|}\B)^2\mu{_f}\B).
\end{equation}
Now \;$\mu_{f^{-1}}\circ f=-\dfrac{f_{\bar z}}{(\bar f)_{\bar z}}=-\dfrac{f_{\bar z}}{{\overline{(f_z)}}}$ and so $|\mu_{f^{-1}}\circ f|=\dfrac{|f_{\bar z}|}{|f_z|}=|\mu_f|.$\\
From (\ref{eq23}) and (\ref{eq24}) we have
\[
J\,(g_{\zeta}\circ f)=(g\circ f)_z(\bar f)_{\bar z}-(g\circ f)_{\bar z}(\bar f)_z\quad\text {and}
\]
\[
J\,(g_{\bar\zeta}\circ f)=(g\circ f)_{\bar z}f_z-(g\circ f)_z f_{\bar z}
\]
Multiplying first equation by $f_{\bar z},$ second equation by $(\bar f)_{\bar z}$ and adding we get\\
$J$\,($f_{\bar z}\,(g_{\zeta}\circ f)+(\bar f)_{\bar z}\,(g_{\bar\zeta}\circ f)$) = ($g\circ f$)$_{\bar z}\,(f_z\,(\bar f)_{\bar z}-(\bar f)_z\,f_{\bar z})$ and so
\begin{equation}\label{eq28}
(g\circ f)_{\bar z}=J\,\dfrac{(f_{\bar z}\,(g_{\zeta}\circ f)+(\bar f)_{\bar z}\,(g_{\bar\zeta}\circ f))}{(f_z\,(\bar f)_{\bar z}-(\bar f)_z\,f_{\bar z})}\cdot
\end{equation}
Again multiplying first equation by $f_z$ and second equation by $(\bar f)_z$ we have\\
$J\,(f_z\,(g_{\zeta}\circ f)+(\bar f)_z\,(g_{\bar\zeta}\circ f))=(g\circ f)_z(f_z\,(\bar f)_{\bar z}-(\bar f)_z\,f_{\bar z})$ and so
\begin{equation}
(g\circ f)_z=J\,\dfrac{(f_z\,(g_{\zeta}\circ f)+(\bar f)_z\,(g_{\bar\zeta}\circ f))}{(f_z\,(\bar f)_{\bar z}-(\bar f)_z\,f_{\bar z})}\cdot
\end{equation}
Then
\begin{equation}
\begin{split}
\notag 
\mu_{g\circ f}
&=\dfrac{(g\circ f)_{\bar z}}{(g\circ f)_z}\quad(\text{by (\ref{eq16}}))\\
&=\dfrac{J\,\dfrac{(f_{\bar z}\,(g_{\zeta}\circ f)+(\bar f)_{\bar z}\,(g_{\bar\zeta}\circ f))}{(f_z\,(\bar f)_{\bar z}-(\bar f)_z\,f_{\bar z})}}{J\,\dfrac{(f_z\,(g_{\zeta}\circ f)+(\bar f)_z\,(g_{\bar\zeta}\circ f))}{(f_z\,(\bar f)_{\bar z}-(\bar f)_z\,f_{\bar z})}}\\
&=\dfrac{(f_{\bar z}\,(g_{\zeta}\circ f)+(\bar f)_{\bar z}\,(g_{\bar\zeta}\circ f))}{(f_z\,(g_{\zeta}\circ f)+(\bar f)_z\,(g_{\bar\zeta}\circ f))}\\
&=\dfrac{\dfrac{f_{\bar z}}{f_z}+\dfrac{\overline{(f_z)}}{f_z}\dfrac{(g_{\bar\zeta}\circ f)}{(g_{\zeta}\circ f)}}{1+\dfrac{(\bar f)_z}{f_z}\,\dfrac{(g_{\bar\zeta}\circ f)}{(g_{\zeta}\circ f)}}\quad((\bar f)_{\bar z}=\overline{(f_z)}\;\text{and dividing throughout by}\;f_z\,(g_{\zeta}\circ f))\\
&=\dfrac{\dfrac{f_{\bar z}}{f_z}+\dfrac{\overline{(f_z)}}{f_z}\,(\mu_{ g}\,\circ f)}{1+\dfrac{\overline{(f_{\bar z})}}{f_z}\,\dfrac{\overline{(f_z)}}{\overline{(f_z)}}\,(\mu_{ g}\,\circ f)}\quad\B((\bar f)_z=\overline{(f_{\bar z})}\;\text{and}\;\mu_{g}\,\circ  f=\dfrac{g_{\bar\zeta}\circ f}{g_{\zeta}\circ f}\B)\\
&=\dfrac{\mu_ f+k_f\,(\mu_g\,\circ f)}{1+k_f\,\overline{(\mu_f)}\,(\mu_{g}\,\circ f)}\quad\text{where}\;\; k_f=\dfrac{\overline{(f_z)}}{f_z}\cdot 
\end{split}
\end{equation}
\begin{equation}\label{eq30}
\Rightarrow \mu_{g\circ f}=\dfrac{\mu_ f+k_f\,(\mu_g\,\circ f)}{1+k_f\,\overline{(\mu_f)}\,(\mu_{g}\,\circ f)}\quad\text{where}\;\; k_f=\dfrac{\overline{(f_z)}}{f_z} 
\end{equation}
which on rearrangement gives
\begin{equation}\label{eq31}
\mu_g\,\circ f=\dfrac{1}{k_f}\,\B(\dfrac{\mu_{g\circ f}-\mu_f}{1-\overline{(\mu_f)}\,\mu_{g\circ f}}\B)\cdot
\end{equation}
If $g$ is conformal, then $\mu_g=0$ and so $\mu_{g\circ f}=\dfrac{\mu_f+0}{1+0}=\mu_f,$ which implies that $D_{g\circ f}=\dfrac{1+|\mu_{g\circ f}|}{1-|\mu_{g\circ f}|}=\dfrac{1+|\mu_f|}{1-|\mu_ f|}=D_f.$\\
If $f$ is conformal then $\mu_f=0$  and so $\mu_{g\circ f}=\dfrac{k_f(\mu_g\,\circ f)}{1+0}\;$(since $\mu_f=0,\;\Rightarrow \overline{(\mu_f)}=0$). $\;\Rightarrow \mu_{g\circ f}=\dfrac{\overline{(f_z)}}{f_z} (\mu_g\,\circ f).\;$ Similarly we get $D_{g\circ f}=D_g.$ Hence we conclude that the dilatation is invariant with respect to all conformal transformations.\\
If we set $g\circ f=h,$ then from (\ref{eq31}) we get
\begin{equation}\label{eq32}
\mu_{h\circ {f^{-1}}}\,\circ f=\dfrac{f_z}{(\bar f)_{\bar z}}\,\B(\dfrac{\mu_h-\mu_f}{1-\overline{(\mu_f)}\,\mu_h}\B)\quad(\text{by (\ref{eq19})})\cdot
\end{equation}
We now show that quasiconformal maps are closed under compositions and inverses.
\begin{theorem}\label{2.9.4}
 Suppose $U,\,V$ and $W$ are open sets in $\ti\C.$ If $f: U\to V$ is $K_1-$ quasiconformal and $g:V\to W$ is $K_2-$ quasiconformal, then $g\circ f : U\to W$ is $K_1K_2-$ quasiconformal.
\end{theorem}
\begin{proof}
Let $\zeta=f(z).$ We denote $g_{\zeta}\circ f$ by $g_{\zeta}$ and $g_{\bar\zeta}\circ f$ by $g_{\bar\zeta}$. We compute the dilatation of  $g\circ f.$ 
\begin{equation}
\begin{split}
\notag
D_{g\circ f}
&=\dfrac{|(g\circ f)_z|+|(g\circ f)_{\bar z}|}{|(g\circ f)_z|-|(g\circ f)_{\bar z}|}\quad\text{(by (\ref{eq15}))}\\
&=\dfrac{|(g_{\zeta}\circ f)f_z+(g_{\bar\zeta}\circ f)\overline{(f_{\bar z})}|+|(g_{\zeta}\circ f)f_{\bar z}+(g_{\bar\zeta}\circ f)\overline{(f_z)}|}{|(g_{\zeta}\circ f)f_z+(g_{\bar\zeta}\circ f)\overline{(f_{\bar z})}|-|(g_{\zeta}\circ f)f_{\bar z}+(g_{\bar\zeta}\circ f)\overline{( f_z)}|}\quad(\text{by (\ref{eq21}) and (\ref{eq22})})\\
&=\dfrac{|g_\zeta f_z + g_{\bar\zeta}\overline{(f_{\bar z})}|+|g_\zeta f_{\bar z} + g_{\bar\zeta}\overline{(f_z)|}}{|g_\zeta f_z + g_{\bar\zeta}\overline{(f_{\bar z})}| - |g_\zeta f_{\bar z} +g_{\bar\zeta}\overline{(f_z)|}}\\
&\leqslant\dfrac{|g_\zeta||f_z |+| g_{\bar\zeta}||\overline{(f_{\bar z})}|+|g_\zeta|| f_{\bar z}|+|g_{\bar\zeta}||\overline{(f_z)|}}{|g_\zeta||f_z |-| g_{\bar\zeta}||\overline{(f_{\bar z})}|-|g_\zeta|| f_{\bar z}|+|g_{\bar\zeta}||\overline{(f_z)|}}\\
&=\dfrac{|g_\zeta||f_z |+| g_{\bar\zeta}||f_z|+|g_\zeta|| f_{\bar z}|+|g_{\bar\zeta}||f_{\bar z}|}{|g_\zeta||f_z |-| g_{\bar\zeta}||f_ z|-|g_\zeta|| f_{\bar z}|+|g_{\bar\zeta}|f_{\bar z}|}\quad(\text{by (\ref{eq19}) and (\ref{eq20})})\\
&=\dfrac{|f_z|(|g_\zeta|+|g_{\bar\zeta}|)+|f_{\bar z}|(|g_\zeta|+| g_{\bar\zeta}|)}{|f_z|(|g_\zeta|-|g_{\bar\zeta}|)-|f_{\bar z}|(|g_\zeta|-|g_{\bar\zeta}|)}\\
&=\dfrac{(|g_\zeta|+|g_{\bar\zeta}|)(|f_z|+|f_{\bar z}|)}{(|g_\zeta|-|g_{\bar\zeta}|)(|f_z|-|f_{\bar z}|)}\\
&=D_g D_f.
\end{split}
\end{equation}
As $f$ is $K_1-$ quasiconformal and $g$ is $K_2-$ quasiconformal, so $D_f\leqslant K_1$ and $D_g\leqslant K_2$ which implies that $D_{g\circ f}\leqslant  K_1 K_2.$ Hence ${g\circ f}$ is $K_1K_2-$ quasiconformal.
\end{proof}
\begin{theorem}\label{2.9.5}
The inverse of a $K-$ quasiconformal map is  $K-$ quasiconformal.
\end{theorem}
\begin{proof}
 From (\ref{eq25}) and (\ref{eq26}) we get $(f^{-1})_{\zeta}=\dfrac{1}{J}(\bar f)_{\bar z}\,f^{-1}$ and $(f^{-1})_{\bar\zeta}=-\dfrac{1}{J}f_{\bar z}\,f^{-1}.$ Then
\begin{equation}
\begin{split}
\notag
D_{f^{-1}}
&=\dfrac{|f^{-1}\,_\zeta|+|f^{-1}\,_{\bar\zeta}|}{|f^{-1}\,_\zeta|-|f^{-1}\,_{\bar\zeta}|}\\
&=\dfrac{\dfrac{1}{J}|(\bar f)_{\bar z}||f^{-1}|+\dfrac{1}{J}|f_{\bar z}||f^{-1}|}{\dfrac{1}{J}|(\bar f)_{\bar z}||f^{-1}|-\dfrac{1}{J}|f_{\bar z}||f^{-1}|}\quad(J>0)\\
&=\dfrac{|f_z|+|f_{\bar z}|}{|f_z|-|f_{\bar z}|}\quad(\text{by (\ref{eq19}) and (\ref{eq20})})\\
&=D_f\\
&\leqslant K.
\end{split}
\end{equation}
Hence $f^{-1}$ is $K-$ quasiconformal.
\end{proof}
We now state a result on compactness property of quasiconformal mappings (see~\cite{lehto}).
\begin{theorem}\label{2.9.6}
Let $U\subset\ti\C$ be a domain in $\ti\C.$ Let $\{f_n\}$ be a sequence of $K-$ quasiconformal mappings of $U$ into $\ti\C,$ which is uniformly convergent on each compact subset of $U$ to a function $f$. Then $f$ is either $constant$ or a $K-$ quasiconformal mapping.
\end{theorem}
We now deal with  measurable conformal structure on an open set $U\subset\ti\C.$\\ 
A conformal structure at some point $z\in U$ is prescribed by choosing some ellipse centered at origin in the tangent space $T_z U\cong\ti\C.$ This ellipse is viewed as a circle in the new conformal structure.
A bounded measurable conformal structure $\mu$ is a measurable field of infinitesimal ellipses on the tangent spaces to $\ti\C$, defined a.e. (almost everywhere) with minor axis of each ellipse tilted at $\dfrac{1}{2}\emph{arg}\,\mu(z)$ and major axis tilted at $\dfrac{\pi}{2}+\dfrac{1}{2}\emph{arg}\,\mu(z)$ and with eccentricities of  ellipses being bounded.
A quasiconformal mapping $\phi$ of $\ti\C$ converts the standard conformal structure into a bounded measurable conformal structure.
\begin{definition}
[\textbf{Beltrami differential equation}]
If $f\in C^1$ and $\mu$ is a complex valued Lebesgue measurable function with $\|\mu\|_{\ity}\leqslant k<1$ a.e., then the Beltrami equation is defined as 
\begin{equation}
  \dfrac{\partial f}{\partial\bar z}=\mu(z)\dfrac{\partial f}{\partial z}\cdot
\end{equation}
\end{definition}
\begin{definition}\label{d20}
\textbf{(Distributional derivatives)}
Let $U\subset\C$ be open. A continuous function $g :U\to \C$ has distributional derivatives\;$g_z$ and $g_{\bar z}$ which is defined a.e. in $U$ and belonging to $L^1(U)$ if
\begin{enumerate}
\item\ $\iint_{U}g_z\phi(z)\,dx\,dy=-\iint_{U}g(z)\dfrac{\partial\phi}{\partial z}\,dx\,dy\;\;$and
\item\ $\iint_{U}g_{\bar z}\phi(z)\,dx\,dy=-\iint_{U}g(z)\dfrac{\partial\phi}{\partial\bar z}\,dx\,dy$
\end{enumerate}
holds for every \emph{test function} $\phi.$ (A test function $\phi :U\to\C$ is of class $C^\ity$ and has compact support.) The  Beltrami equation for $f$ now requires
\begin{equation}\label{eq33}
f_{\bar z}(z)=\mu(z) f_z(z)
\end{equation}
to hold  a.e. in $U.$
\end{definition}
We want to find the solution of the  Beltrami equation which is known as \emph{The mapping\,-\,theorem}. The mapping theorem is the foundation stone for quasiconformal mappings.\\
We define two Integral Operators on $L^p(\Omega)$ where $\Omega$ is the complex plane $\C$ as :
\begin{equation}\label{eq34}
Ph(\zeta)=-\dfrac{1}{\pi}\iint_{\Omega} h(z)(\dfrac{1}{z-\zeta}-\dfrac{1}{z})\,dx\,dy.
\end{equation}
\begin{lemma}\label{2.9.9}
$Ph$ is \emph{H$\ddot{\text{o}}$lder continuous}  of exponent $1-\frac{2}{p}\cdot$
\end{lemma}
\begin{proof}
We first show that the integral is convergent.
\begin{equation}
\begin{split}
\notag
|Ph(\zeta)|
&=\dfrac{1}{\pi}|\iint h(z)(\dfrac{1}{z-\zeta}-\dfrac{1}{z})\,dx\,dy|\\
&=\dfrac{1}{\pi}|\iint\dfrac{ h(z)\,\zeta}{z(z-\zeta)}\,dx\,dy|\\
&\leqslant \dfrac{|\zeta|}{\pi}\,\|h\|_p\,\|\frac{1}{z(z-\zeta)}\|_q\quad(\text{H$\ddot{\text{o}}$lder's inequality)}). 
\end{split}
\end{equation}
(Since $h\in L^p\;\;p>2,\;\dfrac{\zeta}{z(z-\zeta)}\in L^q,\;1<q<2$ and $\dfrac{1}{p}+\dfrac{1}{q}=1$).\\
Consider $\dfrac{1}{|z(z-\zeta)|^q}=\dfrac{1}{|\zeta|^q\,|z\,(\dfrac{z}{\zeta}-1)|^q}$. Put $\dfrac{z}{\zeta}=u,\;\Rightarrow\dfrac{1}{|z\,(z-\zeta)|^q}=\dfrac{1}{|\zeta|^q\,|u\zeta|^q\,|u-1|^q}=\dfrac{1}{|\zeta|^{2q}\,|u|^q\,|u-1|^q}.\;$ Now $dx=\dfrac{dz+d\bar z}{2},\;dy=\dfrac{dz-d\bar z}{2i}.\;$ Suppose $u=x_1+iy_1$. Then $u=\dfrac{z}{\zeta},\;\Rightarrow du=\dfrac{dz}{\zeta}$ and $d\bar u=\dfrac{d\bar z}{\zeta}$ and so $dx_1=\dfrac{dx}{\zeta}$ and $dy_1=\dfrac{dy}{\zeta}$ which implies $dx_1\,dy_1=\dfrac{dx\,dy}{\zeta^2}\cdot$ Then\\
$\iint \dfrac{1}{|z(z-\zeta)|^q}\,dx\,dy=\iint \dfrac{1}{|\zeta|^{2q}}\,\dfrac{|\zeta|^2}{(|u||u-1|)^q}\,dx_1\,dy_1=|\zeta|^{2-2q}\,\iint \dfrac{1}{|z(z-1)|^q}\,dx\,dy,$ which implies $\|\dfrac{1}{z(z-\zeta)}\|_q\leqslant|\zeta|^{\frac{2}{q} - 2}\,\B(\iint \dfrac{1}{|z(z-1)|^q}\,dx\,dy\B)^{\frac{1}{q}}.\;$ So 
\begin{equation}
\begin{split}
\notag
|Ph(\zeta)|
&\leqslant\dfrac{|\zeta|}{\pi}\,|\zeta|^{\frac{2}{q}-2}\,\|h\|_p\B(\iint \dfrac{1}{|z(z-1)|^q}\,dx\,dy\B)^{\frac{1}{q}}\\
&=\dfrac{|\zeta|^{\frac{2}{q}-1}}{\pi}\,\|h\|_p\,\B(\iint \dfrac{1}{|z(z-1)|^q}\,dx\,dy\B)^{\frac{1}{q}}\\
&=|\zeta|^{1-\frac{2}{p}}\,K_p\,\|h\|_p\quad(\frac{2}{q}-1=1-\frac{2}{p})
\end{split}
\end{equation}
where the last integral is some constant that depends only on $p.$ Thus we have
\begin{equation}\label{eq35}
|Ph(\zeta)|\leqslant K_p\,\|h\|_p\,|\zeta|^{1-\frac{2}{p}}.
\end{equation}
We apply this result to $h_1(z)=h(z+\zeta_2)\;$(clearly $h_1\in L^p$ and $\|h_1\|_p=\|h\|_p$).
\begin{equation}
\begin{split}
\notag
Ph_1(\zeta_1-\zeta_2)
&=-\dfrac{1}{\pi}\,\iint h_1(z)(\dfrac{1}{z+\zeta_2-\zeta_1}-\dfrac{1}{z})\,dx\,dy\\
&=-\dfrac{1}{\pi}\,\iint h(z+\zeta_2)(\dfrac{1}{z+\zeta_2-\zeta_1}-\dfrac{1}{z})\,dx\,dy\\
&=-\dfrac{1}{\pi}\,\iint h(z)(\dfrac{1}{z-\zeta_1}-\dfrac{1}{z-\zeta_2})\,dx\,dy\quad(\text{change of variable})\\
&=Ph(\zeta_1)-Ph(\zeta_2).\\
\Rightarrow |Ph(\zeta_1)-Ph(\zeta_2)|
&=|Ph_1(\zeta_1-\zeta_2)|\\
&\leqslant K_p\,\|h\|_p\,|\zeta_1-\zeta_2|^{1-\frac{2}{p}}\quad(\text{by (\ref{eq35})}).
\end{split}
\end{equation}
Thus we have 
\begin{equation}\label{eq36}
 |Ph(\zeta_1)-Ph(\zeta_2)|\leqslant K_p\,\|h\|_p\,|\zeta_1-\zeta_2|^{1-\frac{2}{p}}
\end{equation}
which proves that $Ph$ is H$\ddot{\text{o}}$lder continuous of exponent $1-\frac{2}{p}.$
\end{proof}
The second operator $T$ is initially defined for functions $h\in {C_0}^2\;(C^2$ with compact support), as the Cauchy Principal Value
\begin{equation}\label{eq37}
Th(\zeta)=\lim_{\epsilon\to 0}\iint_{|z-\zeta|>\,\epsilon}\;\dfrac{h(z)}{(z-\zeta)^2}\,dx\,dy.
\end{equation}
We now state the following results which we will use to prove the mapping theorem (see \textbf{Ahlfors}~\cite{ahlfors}).
\begin{lemma}\label{2.9.10}
For $h\in {C_0}^2, Th$ is well defined and $Th\in C^1$ and
\begin{enumerate}
\item\ $(Ph)_{\bar z}=h$
\item\ $(Ph)_z=Th$
\item\ $\|Th\|_2= \|h\|_2.$
\end{enumerate}
\end{lemma}
(\textbf{Calderon\,-\,Zygmund inequality}) This allows us to extend $T$ to $L^p,\; p> 2.$
\begin{equation}\label{eq38}
\|Th\|_p\leqslant C_p\,\|h\|_p\;\;p>1,\;\;h\in L^p,
\end{equation}
and $C_p\to 1$ as $p\to 2.$
\begin{lemma}\label{2.9.11}
For $h\in L^p$ with $p>2,$ the relations
\begin{equation}\label{eq39}
(Ph)_{\bar z}=h
\end{equation}
\begin{equation}\label{eq40}
(Ph)_z=Th
\end{equation}
hold in the sense of distributions.
\end{lemma}
\begin{lemma}\label{2.9.12}
\textbf{(Weyl)}
Let $U\subset\C$ be open, and $f:U\to\C$ is a distribution in $U$ satisfying $f_{\bar z}=0,$ then $f$ is analytic in $U$.
\end{lemma}
We now give a genaralization of Weyl's lemma.
\begin{lemma}\label{2.9.13}
If $p$ and $q$ are continuous and have locally integrable distributional derivatives that satisfy $\; p_{\bar z}=q_z,$ then $\;\exists\;$ a function $f\in C^1$ with $f_z=p$ and $f_{\bar z}=q.$
\end{lemma}

\subsection{Solution of the Measurable Riemann mapping theorem (Ahlfors - Bers)}
We want to solve the Beltrami equation\; $f_{\bar z}(z)=\mu(z) f_z(z)$  with $\|\mu\|_{\ity}\leqslant k<1$ a.e.  Firstly we suppose that $\mu$ has compact support which will imply that $f$ is analytic in a neighbourhood of $\ity$ (a neighbourhood of $\ity$ is of the form $\{z : |z|>r \;\cup\{\ity\} \}$ for some $r>0$). We use a fixed exponent $p>2,$ fixed $k<1$ such that $kC_p<1$ where $C_p$ is the constant of 
Calderon\,-\,Zygmund inequality.
\begin{theorem}\label{2.9.14}
Under the above assumptions, there exists a unique solution of  the Beltrami equation,\;$ f_{\bar z}(z)=\mu(z) f_z(z)$ such that $f(0)=0$ and $f_z-1\in L^p.$
\end{theorem}
\begin{proof}
We first prove the uniqueness part. This will also suggest the existence part of the proof.\\
Suppose $f$ is a solution of  the Beltrami equation such that $f(0)=0$ and $f_z-1\in L^p.$ Then $f_{\bar z}(z)=\mu(z) f_z(z)\in L^p$ and so $P(f_{\bar z})$ is well defined. Consider the function
\begin{equation}\label{eq}
F=f-P(f_{\bar z}).
\end{equation}
So $F_{\bar z}=f_{\bar z}-(P(f_{\bar z}))_{\bar z}=f_{\bar z}-f_{\bar z}\;$(by (\ref{eq39}))=0 in the distributional sense. By Lemma \ref{2.9.12}, $F$ is analytic. Also we have  $f_z-1\in L^p.$ Then
\begin{equation}
\begin{split}
\notag
f
&=F+P(f_{\bar z})\\
\Rightarrow f_z-1
&=F_z+(P(f_{\bar z}))_z-1\;\in L^p\\
&=F'+T(f_{\bar z})-1\;\in L^p\quad(F\;\text{is holomorphic and}\;(Ph)_z=Th\;\text{ by (\ref{eq40})})\\
&=F'-1+T(f_{\bar z})\;\in L^p.
\end{split}
\end{equation}
$\Rightarrow F'-1\in L^p.$ This is possible only if $ F'-1=0\;\{$since $\|g'\|_p <\ity$ only if $g$ is a constant, but converse is not true as can be seen for $g(z)=e^{-z^2}\}.$ So $F'=1.$ On integrating we get $F(z)=z+a$ for some constant $a.$ Also we have $f(0)=0,$
\begin{equation}
\begin{split}
\notag
\Rightarrow (F+P(f_{\bar z}))(0)
&=0\\
\Rightarrow F(0)+P(f_{\bar z})(0)
&=0\\
\Rightarrow a +\dfrac{-1}{\pi}\,\iint f_{\bar z}(z)\,(\dfrac{1}{z-0}-\dfrac{1}{z})\,dx\,dy
&=0\\
\Rightarrow a-0
&=0\\
\Rightarrow a
&=0.
\end{split}
\end{equation}
Thus we have
\begin{equation}\label{eq}
f=P(f_{\bar z})+z.
\end{equation}
$f_z=(P(f_{\bar z}))_z+1=T(f_{\bar z})+1\;$(by (\ref{eq40}))$=T(\mu f_z)+1\;$(by (1.8.22)). Thus $f_z=T(\mu f_z)+1.$\\
Similarly if $g$ is another solution of (\ref{eq33}), then $g_z=T(\mu g_z)+1$ and so $f_z-g_z=T(\mu f_z)-T(\mu g_z)=T(\mu(f_z-g_z))\;(T$ is linear). Then
\begin{equation}
\begin{split}
\notag
\|f_z-g_z\|_p
&=\|T(\mu(f_z-g_z))\|_p\\
&\leqslant C_p\,\|\mu(f_z-g_z)\|_p\quad\text{(by (\ref{eq38}))}\\
&\leqslant C_p\,\|\mu\|_\ity\,\|f_z-g_z\|_p\\
&\leqslant kC_p\,\|f_z-g_z\|_p\quad(\|\mu\|_\ity\leqslant k).
\end{split}
\end{equation}
Since $kC_p<1,$ this is possible if and only if $\|f_z-g_z\|_p=0,\;\Rightarrow f_z-g_z=0$ a.e.\;$\Rightarrow f_z=g_z$ a.e. So we have
\begin{equation}
\begin{split}
\notag
f_{\bar z}
&=\mu f_z\quad(\text{by (\ref{eq33})})\\
&=\mu g_z\;\;\text{a.e.}\\
&=g_{\bar z}\;\;\text{a.e.}\quad(\text{by (\ref{eq33})}).
\end{split}
\end{equation}
Hence $f-g$ and $\bar f-\bar g$ are analytic and so the difference must be a constant. The normalization at origin gives $f-g=0,\;$i.e $f=g.$ Thus we have established the uniqueness.\\
To prove the existence, we consider the operator
\begin{equation}\label{eq43}
U_\mu(h)=T(\mu h), \;h\in L^p.
\end{equation}
Consider the equation 
\begin{equation}\label{eq46}
h=T(\mu h)+T(\mu).
\end{equation}
$\Rightarrow h-T(\mu h)
=T(\mu),\;
\Rightarrow h-U_\mu(h)
=T(\mu),\;
\Rightarrow (I-U_\mu)(h)
=T(\mu).$
Then
\begin{equation}
\begin{split}
\notag
\|U_\mu(h)\|
&=\|T(\mu h)\|\\
&\leqslant C_p\,\|\mu h\|_p\quad\text{(by (\ref{eq38}))}\\
&\leqslant k C_p\,\|h\|_p\\
&< \|h\|_p\quad( k C_p<1)\\
\Rightarrow \|U_\mu\|_p
&<1.\\
\Rightarrow I-U_\mu\;\text{is invertible and}\\
(I-U_\mu)^{-1}
&=\sum_{n=0}^{\ity}\,{U_\mu}^n\\
\Rightarrow\|(I-U_\mu)^{-1}\|_p
&\leqslant \sum_{n=0}^{\ity}\,\|{U_\mu}^n\|_p\;\;(\text{the sum depends analytically on}\;\mu)\\
&\leqslant \sum_{n=0}^{\ity}\,(kC_p)^n\\
&=\dfrac{1}{1-kC_p}\cdot
\end{split}
\end{equation}
Thus we can solve for $h$ to obtain
\begin{equation}\label{eq44}
h=(I-U_\mu)^{-1}\,T(\mu)\;\in L^p\quad(\text{as}\;h\in L^p).
\end{equation}
Define
\begin{equation}\label{eq45}
f(z)=P(\mu(h+1))+z.
\end{equation}
Then $f$ also depends analytically on $\mu.$ \\
We claim that $f$ is the desired solution of (\ref{eq33}).\\
Since $\mu$ has compact support, therefore $\mu(h+1)\in L^p$ and so $P(\mu(h+1))$ is well defined. Since the operator $P$ is continuous, therefore $f$ is continuous. Then
\begin{equation}
\begin{split}
\notag
f_{\bar z}
&=(P(\mu(h+1)))_{\bar z}+0\\
&=\mu(h+1)\quad\text{(by (\ref{eq39}))}
\end{split}
\end{equation}
and
\begin{equation}
\begin{split}
\notag
f_z
&=(P(\mu(h+1)))_z+1\\
&=T(\mu(h+1))+1\quad\text{( by (\ref{eq40}) )}\\
&=T(\mu h+\mu)+1\\
&=T(\mu h)+T(\mu)+1\quad(T\;\text{is linear})\\
&=h+1\quad\text{(by (\ref{eq46}))}.
\end{split}
\end{equation}
On taking pointwise product with $\mu$ on both sides we have $\mu f_z=\mu(h+1)\;(\mu\in L^\ity, f_z\in L^p$ so $\mu f_z\in L^p$) and thus we have $f_{\bar z}=\mu f_z$ which proves the claim.\\
 From (\ref{eq45}), $f(0)=P(\mu(h+1))(0)=0.$ Also $f_z=h+1,\;\Rightarrow f_z-1=h\in L^p.$\\
The function $f$ is called \emph{normal solution} of (\ref{eq33}).
\end{proof}
We now derive some estimates :\\
From (\ref{eq46}) we have 
\begin{equation}
\begin{split}
\notag
h
&=T(\mu h)+T(\mu)\\
\Rightarrow \|h\|_p
&\leqslant kC_p\,\|h\|_p+C_p\,\|\mu\|_p\\
\Rightarrow (1-kC_p)\,\|h\|_p
&\leqslant C_p\,\|\mu\|_p\\
\end{split}
\end{equation}
\begin{equation}\label{eq47}
\Rightarrow \|h\|_p\leqslant \dfrac{C_p}{1-kC_p}\,\|\mu\|_p.
\end{equation}
Also from (\ref{eq36}) 
\begin{equation}
\begin{split}
\notag
|f(\zeta_1)-f(\zeta_2)|
&=|P(\mu(h+1))\zeta_1+\zeta_1 - P(\mu(h+1))\zeta_2-\zeta_2|\\
&=|P(\mu(h+1))\zeta_1 - P(\mu(h+1))\zeta_2 +\zeta_1-\zeta_2|\\
&\leqslant K_p\,\|\mu(h+1)\|_p\,|\zeta_1-\zeta_2|^{1-\frac{2}{p}} +|\zeta_1-\zeta_2|\\
&\leqslant K_p\,(\|\mu h\|_p+\|\mu\|_p)\,|\zeta_1-\zeta_2|^{1-\frac{2}{p}} +|\zeta_1-\zeta_2|\quad(\text{triangle inequality})\\
&\leqslant\dfrac{K_p}{1-kC_p}\,\|\mu\|_p\,|\zeta_1-\zeta_2|^{1-\frac{2}{p}} +|\zeta_1-\zeta_2|\quad\text{(on simplification)}.
\end{split}
\end{equation}
Thus we have
\begin{equation}\label{eq48}
|f(\zeta_1)-f(\zeta_2)|\leqslant\dfrac{K_p}{1-kC_p}\,\|\mu\|_p\,|\zeta_1-\zeta_2|^{1-\frac{2}{p}} +|\zeta_1-\zeta_2|.
\end{equation}
\begin{lemma}\label{2.9.15}
Let $\mu_n,\,\mu$ satisfy the hypothesis of Theorem \ref{2.9.14} with compact supports and $f_n,\, f$ be the respective normal solutions. Suppose $\mu_n\to\mu$\;\;a.e. Then $\|(f_n)_z-f_z\|_p\to 0$ and $f_n\to f$ uniformly on compact sets.
\end{lemma}
\begin{proof}
From (\ref{eq47}), $\;\|h\|_p\leqslant \dfrac{C_p}{1-kC_p}\,\|\mu\|_p$. Now $f_{\bar z}=\mu f_z,\;\Rightarrow \|f_{\bar z}\|_p=\|\mu f_z\|_p\;\leqslant \|\mu\|_\ity\,\|f_z\|_p.$ Also from (\ref{eq48}), 
\[
|f(\zeta_1)-f(\zeta_2)|\leqslant\dfrac{K_p}{1-kC_p}\,\|\mu\|_p\,|\zeta_1-\zeta_2|^{1-\frac{2}{p}} +|\zeta_1-\zeta_2|.
\]
We have $(f_n)_z\in L^p$ and $f_z\in L^p.$ 
Then
\begin{equation}
\begin{split}
\notag
(f_n)_z-f_z
&=T\mu_n((f_n)_z+1)-T\mu(f_z+1)\quad(\text{using (\ref{eq45})})\\
&=T\mu_n(f_n)_z-T\mu_n(f_z)+T\mu_n(f_z)+T\mu_n-T\mu(f_z)-T\mu\\
&=T(\mu_n((f_n)_z-f_z))+ T((\mu_n-\mu)f_z)+T(\mu_n-\mu).\\
\Rightarrow \|(f_n)_z-f_z\|_p
&\leqslant \|T(\mu_n((f_n)_z-f_z))\|_p+\|T((\mu_n-\mu)f_z)\|_p+\|T(\mu_n-\mu)\|\\
&\leqslant kC_p\|(f_n)_z-f_z\|_p+C_p\|(\mu_n-\mu)f_z\|_p+C_p\|\mu_n-\mu\|_p\quad\text{(by (\ref{eq38}))}.
\end{split}
\end{equation}
Since $\mu_n\to\mu$\;\;a.e.  (on compact sets), so $\|\mu_n-\mu\|_p\to 0$. Also $\|f_z\|_p <\ity,\;\Rightarrow \|(\mu_n-\mu)f_z|_p\to 0.$ So we get that $\|(f_n)_z-f_z\|_p\leqslant  kC_p\|(f_n)_z-f_z\|_p\;<\|(f_n)_z-f_z\|_p\;$(since $kC_p<1$), which implies that $\|(f_n)_z-f_z\|_p<\|(f_n)_z-f_z\|_p,$ a contradiction.\, Therefore we must have $\|(f_n)_z-f_z\|_p\to 0$ and so $(f_n)_z-f_z\to 0$\;\;a.e., i.e $(f_n)_z\to f_z$\;\;a.e. Also 
\begin{equation}
\begin{split}
\notag
(f_n)_{\bar z}
&=\mu_n(f_n)_z\\
&\to \mu_n f_z\;\;\text{a.e.}\\
&\to \mu f_z\;\;\text{on compact sets}\;\;\text{a.e.}\\
&=f_{\bar z}\;\;\text{a.e.}\\
\Rightarrow (f_n)_{\bar z}
&\to f_{\bar z}\;\;\text{a.e.}
\end{split}
\end{equation}
Thus we have $(f_n)_z\to f_z$\;\;a.e. and $(f_n)_{\bar z}\to f_{\bar z}$\;\;a.e.which implies that the difference of $f_n$ and $f$ must be a constant and the normalization at origin gives that the constant is 0. Hence $f_n\to f$ uniformly on compact sets. 
\end{proof}
Thus we conclude from the lemma that the normal solution depends continuously on $\mu.$\\
We now show that if $\mu$  has derivatives then so does the normal solution $f$. We will use Lemma \ref{2.9.13}.
\begin{lemma}\label{2.9.16}
If $\mu$ has a distributional derivative $\mu_z\in L^p,\;p>2,$ then $f\in C^1,$ and it is a homeomorphism.
\end{lemma}
\begin{proof}
We try to find $\la$, so that the system\\
$f_z=\la;$\\
$f_{\bar z}=\mu\la$\\
has a solution. By  Lemma \ref{2.9.13}, this will be so if 
\begin{equation}\label{eq49}
\la_{\bar z}=(\mu\la)_z=\la_z\mu+\la\mu_z;
\end{equation}
or
\[
(log\la)_{\bar z}=\mu(log\la)_z+\mu_z.
\]
We can solve the equation 
\begin{equation}\label{eq50}
q=T(\mu q)+T\mu_z
\end{equation}
for $q\in L^p.$ Using (\ref{eq43}), we get $q=U_\mu(q)+T\mu_z$ and so $(I-U_\mu)q=T\mu_z$ which implies that $q=(I-U_\mu)^{-1}\,T\mu_z.$ We set $\sigma=P(\mu q+\mu_z)+constant\;$ in such a way that $\sigma\to 0$ as $z\to\ity.$ Since $\mu q+\mu_z\in L^p,$ $P(\mu q+\mu_z)$ is defined, so that $\sigma$ is well defined. As $P$ is continuous, so is $\sigma.$ Then by \ref{eq39}, $\sigma_{\bar z}=\mu q+\mu_z\;$  and 
\begin{equation}
\begin{split}
\notag
\sigma_z
&=T(\mu q+\mu_{z})\quad\text{(by (\ref{eq40}))}\\
&=q\;\;\text{(by (\ref{eq50}))}.
\end{split}
\end{equation}
Thus the distributional derivatives of $\sigma$ are in $L^p.$\\
Consider $\la=e^\sigma.$ Then
\begin{equation}
\begin{split}
\notag
\la_{\bar z}
&=(e^\sigma)_{\bar z}\\
&=e^\sigma\sigma_{\bar z}\\
&=e^\sigma(\mu q+\mu_{z})\\
&=\la(\mu q+\mu_{z})
\end{split}
\end{equation}
and
\begin{equation}
\begin{split}
\notag
\la_ z\,\mu+\la\, \mu_z
&=(e^\sigma)_z\,\mu+e^\sigma\mu_ z\\
&=e^\sigma\sigma_z\mu+\la \mu_z\\
&=\la\,\mu\,q+\la\, \mu_z\;\;(\text{as}\;\sigma_z=q,\;\Rightarrow \mu\sigma_z=\mu q)\\
&=\la(\mu\, q+\mu_ z).
\end{split}
\end{equation}
Thus $\la=e^\sigma$ satisfies (\ref{eq49}) i.e $\la_{\bar z}=(\mu\la)_z$. Hence $\la\,dz+(\la\,\mu)d\bar z$ is exact and 
\begin{equation}\label{eq51}
f(z)=\int_{0}^{z}\la\,dz+(\la\,\mu)d\bar z
\end{equation}
is well defined and of class $C^1.$ Hence the system we required to be solved can be solved for $f\in C^1.$ Moreover\\
$f_z=\la=e^\sigma\;$ and\\
$f_{\bar z}=\la\,\mu=\mu f_z$. Also $f(0)=0.$\\
We claim that $f$ is the normal solution of the Beltrami equation.\\
Since $\sigma\to 0$ as $z\to\ity\;$ (by assumption), therefore $\la=e^\sigma\to 1.$ Also $f_z=\la\to 1$ implies $f_z-1\to 0$ and so $f_z-1\in L^p.$ Hence $f$ is the normal solution of the Beltrami equation. Now the Jacobian of $f$  is
$J
=|f_z|^2-|f_{\bar z}|^2
=|\la|^2-|\mu|^2\,|\la|^2
=(1-|\mu|^2)|e^{2\sigma}|
\neq 0\;\,(\mu<1).$
$\Rightarrow f$ is locally one-one. Also $f(z)\to\ity$ as $\;z\to\ity$ i.e $f$ is one-one in a neighbourhood of $\ity.$ It can be seen easily that $f$ is   one-one in the complement of  neighbourhood of $\ity\;$(using local one-oneness). So $f$ is globally one-one. Thus $f$ is a homeomorphism of $\ti\C$ to itself.
\end{proof}
Assume the hypothesis of Lemma \ref{2.9.15}. By Theorem \ref{2.9.5}, the inverse function $f^{-1}$ is $K-$ quasiconformal with complex dilatation $\mu^{-1}=\mu_{f^{-1}},$ which satisfies $|\mu^{-1}\circ f|=|\mu|\;$ (by (\ref{eq27})). We try to estimate $\|\mu^{-1}\|_p.$ 
\begin{equation}
\begin{split}
\notag
\iint_\C |\mu^{-1}|^p\,d\xi\,d\eta
&=\iint_\C |\mu|^p J\,dx\,dy\\
&=\iint_\C |\mu|^p(|f_z|^2-|f_{\bar z}|^2)\,dx\,dy\\
&\leqslant \iint_\C |\mu|^p|f_z|^2\,dx\,dy\\
&=\iint_\C |\mu|^p\dfrac{|f_{\bar z}|^2}{|\mu|^2}\,dx\,dy\quad(f_{\bar z}=\mu f_z\;\Rightarrow f_z=\dfrac{f_{\bar z}}{\mu})\\
&=\iint_\C |\mu|^{p-2}|f_{\bar z}|^2\,dx\,dy\\
&\leqslant (\iint_\C  |\mu|^p\,dx\,dy)^{\frac{p-2}{p}}\,(\iint_\C |f_{\bar z}|^p\,dx\,dy)^{\frac{2}{p}}\quad(H\ddot{o}lder\text{'}s\;inequality)\\
&=\B\|\mu\B\|_{p}^{p-2}\,\B\|f_{\bar z}\B\|_{p}^{2}\\
\Rightarrow \|\mu^{-1}\|_p
&\leqslant \B\|\mu\B\|_{p}^{\frac{p-2}{p}}\,\B\|f_{\bar z}\B\|_{p}^{\frac{2}{p}}\\
&\leqslant \B\|\mu\B\|_{p}^{\frac{p-2}{p}}\,\dfrac{1}{(1-kC_p)^{\frac{2}{p}}}\,\B\|\mu\B\|_{p}^{\frac{2}{p}}\quad(\|f_{\bar z}\|_p\leqslant\dfrac{1}{1-kC_p}\,\|\mu\|_p)\\
&=\B\|\mu\B\|_{p}^{\frac{p-2+2}{p}}\,\dfrac{1}{(1-kC_p)^{\frac{2}{p}}}\\
&=\dfrac{1}{(1-kC_p)^{\frac{2}{p}}}\|\mu\|_p.
\end{split}
\end{equation}
Thus
\begin{equation}\label{eq52}
\|\mu^{-1}\|_p\leqslant (1-kC_p)^{\frac{-2}{p}}\|\mu\|_p.
\end{equation}
From (\ref{eq48}) we have,
$|f(\zeta_1)-f(\zeta_2)|\leqslant\dfrac{K_p}{1-kC_p}\,\|\mu\|_p\,|\zeta_1-\zeta_2|^{1-\frac{2}{p}} +|\zeta_1-\zeta_2|.$\\
Applying this to the inverse function $f^{-1}$ we get
\begin{equation}
\begin{split}
\notag
|z_1-z_2|
&\leqslant \dfrac{K_p}{1-kC_p}\,\|\mu^{-1}\|_p\,|f(z_1)-f(z_2)|^{1-\frac{2}{p}}+|f(z_1)-f(z_2)|\\
&\leqslant \dfrac{K_p}{1-kC_p}\,\|\mu\|_p(1-kC_p)^{\frac{-2}{p}}\,|f(z_1)-f(z_2)|^{1-\frac{2}{p}}+|f(z_1)-f(z_2)|\quad\text{(by (\ref{eq52}))}\\
&=K_p\,(1-kC_p)^{-1-\frac{2}{p}}\|\mu\|_p\,|f(z_1)-f(z_2)|^{1-\frac{2}{p}}+|f(z_1)-f(z_2)|.
\end{split}
\end{equation}
Thus
\begin{equation}\label{eq53}
|z_1-z_2|\leqslant  K_p\,(1-kC_p)^{-1-\frac{2}{p}}\|\mu\|_p\,|f(z_1)-f(z_2)|^{1-\frac{2}{p}}+|f(z_1)-f(z_2)|.
\end{equation}
\begin{theorem}\label{2.9.17}
For any $\mu$ with compact support and $\|\mu\|_\ity\leqslant k<1,$ the normal solution to the Beltrami equation (\ref{eq33}), is a quasiconformal homeomorphism which satisfies $\mu_f=\mu.$
\end{theorem}
\begin{proof}
We can approximate $\mu$\; a.e. in $p-norm$ with a sequence of functions $\mu_n\in C_{0}^{1}\;$ ($C^1$ with compact support) with $|\mu_n|\leqslant k.$ Let $f_n$ be the normal solution corresponding to each $\mu_n.$ Then $f_n$\text{'}s satisfy (\ref{eq53}). By Lemma \ref{2.9.15}, $f_n\to f$ uniformly on compact sets and $\|\mu_n\|_p\to \|\mu\|_p,$ we have 
\begin{equation}
\begin{split}
\notag
|z_1-z_2|
&\leqslant  K_p\,(1-kC_p)^{-1-\frac{2}{p}}\|\mu_n\|_p\,|f_n(z_1)-f_n(z_2)|^{1-\frac{2}{p}}+|f_n(z_1)-f_n(z_2)|\\
&\to K_p\,(1-kC_p)^{-1-\frac{2}{p}}\|\mu\|_p\,|f(z_1)-f(z_2)|^{1-\frac{2}{p}}+|f(z_1)-f(z_2)|\\
\Rightarrow |z_1-z_2|
&\leqslant K_p\,(1-kC_p)^{-1-\frac{2}{p}}\|\mu\|_p\,|f(z_1)-f(z_2)|^{1-\frac{2}{p}}+|f(z_1)-f(z_2)|.
\end{split}
\end{equation}
So $f$ also satisfies (\ref{eq53}) and hence can be seen easily to be injective. Then $f$ being a uniform limit of a sequence of $K-$ quasiconformal mappings, by Theorem  \ref{2.9.6}, $f$ is either a \emph{constant} or a $K-$  quasiconformal mapping. But $f$ is not \emph{constant} since it is injective, so it must be a $K-$ quasiconformal mapping and so $f$ possesses locally integrable partial derivatives which are also distributional derivatives. Since $f$ is one-one, therefore $f_z\neq 0$\; a.e. and so $\mu_f=\dfrac{f_{\bar z}}{f_z}$ is defined  a.e. and $\mu_f=\mu.$ 
\end{proof}
We shall now get rid of the assumption that $\mu$ has compact support.
\begin{theorem}\label{2.9.18}
For any complex valued Lebesgue measurable function $\mu\in L^{\ity}(\C)$ with $\|\mu\|_\ity<1,$ there exists a unique normalized quasiconformal mapping $f^\mu$ with complex dilatation $\mu$ that leaves $0,1$ and $\ity$ fixed.
\end{theorem}
\begin{proof}
We divide the proof into three cases.
\begin{enumerate}
\item[(i)] If $\mu$ has compact support, then by Theorem \ref{2.9.17} we need to only normalize $f.$
\item[(ii)] Now suppose $\mu=0$ in a neighbourhood of $0.$ We set 
\begin{equation}\label{eq54}
\tilde{\mu}(z)=\mu(\dfrac{1}{z})\,\dfrac{z^2}{{\bar z}^2}\cdot
\end{equation}
Then \emph{support}\;$\tilde{\mu}$ equals $\ti\C\smallsetminus\D,$ which is closed and hence compact, so that $\tilde{\mu}$ has compact support. So by case (i), $f^{\tilde{\mu}}$ is the normal solution of the Beltrami equation\;$f_{\bar z}=\tilde{\mu}\,f_{z},\;\Rightarrow f_{\bar z}^{\tilde{\mu}}(\dfrac{1}{z})=\tilde{\mu}(\dfrac{1}{z})\,f_{z}^{\tilde{\mu}}(\dfrac{1}{z})\;$ (replacing $z\to\dfrac{1}{z}$).\\
We claim that 
\[
f^{\mu}(z)=\dfrac{1}{f^{\tilde{\mu}}(\frac{1}{z})}
\]
is the required normal solution of (\ref{eq33}).
Using (\ref{eq54}) we get
\begin{equation}\label{eq55}
f_{\bar z}^{\tilde{\mu}}(\dfrac{1}{z})=\mu(z)\,\dfrac{{\bar z}^2}{z^2}\,f_{z}^{\tilde{\mu}}(\dfrac{1}{z}).
\end{equation}
Then\\
$f_{z}^{\mu}(z)=\dfrac{-1}{f^{\tilde{\mu}}(\dfrac{1}{z})^2}\,\dfrac{-1}{z^2}\,f_{z}^{\tilde{\mu}}(\dfrac{1}{z}),$ which implies
\begin{equation}\label{eq56}
f_{z}^{\tilde{\mu}}(\dfrac{1}{z})=z^2\,f^{\tilde{\mu}}(\dfrac{1}{z})^2\,f_{z}^{\mu}(z)
\end{equation}
and $\;f_{\bar z}^{\mu}(z)=\dfrac{1}{f^{\tilde{\mu}}(\dfrac{1}{z})^2}\,\dfrac{1}{{\bar  z}^2}\,f_{\bar z}^{\tilde{\mu}}(\dfrac{1}{z}),$ implies
\begin{equation}\label{eq57}
f_{\bar z}^{\tilde{\mu}}(\dfrac{1}{z})={\bar  z}^2\,f^{\tilde{\mu}}(\dfrac{1}{z})^2\,f_{\bar z}^{\mu}(z).
\end{equation}
We note that the computations are well defined, since a quasiconformal map is differentiable  a.e.\,(from the analytic approach to quasiconformal mappings we get that $f^{\tilde{\mu}}$ is differentiable a.e.) Thus we have 
\begin{equation}
\begin{split}
\notag
{\bar  z}^2\,f^{\tilde{\mu}}(\dfrac{1}{z})^2\,f_{\bar z}^{\mu}(z)
&=\mu(z)\,\dfrac{{\bar z}^2}{z^2}\,f_{z}^{\tilde{\mu}}(\dfrac{1}{z})\quad\text{(by (\ref{eq55}))}\\
&=\mu(z)\,\dfrac{{\bar z}^2}{z^2}\,z^2\,f^{\tilde{\mu}}(\dfrac{1}{z})^2\,f_{z}^{\mu}(z)\quad\text{ (by (\ref{eq56})) }\\
\Rightarrow f_{\bar z}^{\mu}(z)
&=\mu(z)\,f_{z}^{\mu}(z)
\end{split}
\end{equation}
which proves the claim.
\item[(iii)] In the general case we set $\mu=\mu_1+\mu_2$ where $\mu_1=0$ in a neighbourhood of $\ity$ and $\mu_2=0$ in a neighbourhood of $0.$ We try to find $\la$ so that 
\[
f^{\la}\circ f^{\mu_2}=f^{\mu} \;\;\text{and}\;\;
f^\la=f^{\mu}\circ(f^{\mu_2})^{-1}.
\]
From (\ref{eq32}), we find that this is so for
\[
\la={\B(\dfrac{\mu-\mu_2}{1-\mu\,\overline{\mu_2}}\B)\,\B(\dfrac{(f^{\mu_2})_{z}}{(\overline{f^{\mu_2}})_{\bar z}}\B)}\circ (f^{\mu_2})^{-1}.
\]
Proceeding like case (ii), we get that $ f^{\mu_2}$ is the normalized solution of $f_{\bar z} =\mu_2\,f_z$. Also it can be seen that $\la$ has compact support and therefore by case (i), $f^\la$ is the normalized solution of  $f_{\bar z} =\la f_z$. Hence $f^\mu$ is the required unique normalized quasiconformal mapping with complex dilatation $\mu$ which leaves $0,1$ and $\ity$ fixed.
\end{enumerate}
\end{proof}
In short, the \textbf{Ahlfors\,-\,Bers Measurable Riemann mapping theorem} says that the mapping $\phi\to \mu(\phi)$ is onto the space of bounded measurable conformal structures. Moreover, if we normalize $\phi$ to fix three points $\{0,1,\ity\},$ then the inverse correspondence 
$\{$bounded measurable conformal structures$\}\to\{$normalized quasiconformal homeomorp-hisms of $\;\ti\C\}$\\
is well defined and bijective.
\section{Quasiconformal deformations of rational mappings}\label{ch2,sec10}
Let $\mu$ be a measurable conformal structure which is preserved a.e. by the rational mapping $\R,$ i.e $\R$ carries infinitesimal ellipses to infinitesimal ellipses. Let $\phi_\mu$ be the normalized quasiconformal homeomorphism of the sphere (by Ahlfors-Bers theorem) which carries $\mu$ to the standard conformal structure a.e., i.e $\phi_\mu$ maps infinitesimal ellipses to infinitesimal circles.\\
We form the quasiconformal conjugates of the rational map $\R.$
\begin{proposition}\label{2.10.1}
Let $\R_\mu=\phi_\mu\R{\phi_\mu}^{-1}.$ Then $\R_\mu$ is a rational mapping of the sphere ( the $\mu -$\,quasiconformal deformation of $\R$ with $deg\,\R_\mu=deg\,\R$).
\end{proposition}
\begin{proof}
$\R$ is locally a conformal mapping away from its branch points and so it is $1-$  quasiconformal there. By Theorem \ref{2.9.5}, ${\phi_\mu}^{-1}$ is quasiconformal. Also  by Theorem \ref{2.9.4}, the composition of two quasiconformal mappings is  quasiconformal, so $\R_\mu=\phi_\mu\R{\phi_\mu}^{-1}$ is locally a quasiconformal homeomorphism of $\ti\C.$ We have $\R$ preserves $\mu$\; a.e. ${\phi_\mu}^{-1}$ maps infinitesimal circles to infinitesimal ellipses, $\R$ maps infinitesimal  ellipses to infinitesimal  ellipses a.e. and $\phi_\mu$ then maps  infinitesimal ellipses to  infinitesimal circles a.e. So $\R_\mu$ maps infinitesimal circles to infinitesimal circles a.e. Hence $\R_\mu$ preserves the standard conformal structure a.e. and so locally $\R_\mu$ must be conformal (a conformal map maps infinitesimal circles to infinitesimal circles).\\
Now let us see what happens on the branch points of $\R.$\\
We have $\R_\mu\phi_\mu=\phi_\mu\R.$ Suppose $z_0\in\ti\C$ is such that $\R(z_0)$ is a branch point of $\R.$ Then $\phi_\mu$ sends  $\R(z_0)$ to a removable singularity of $\R.$ Also $(\R_\mu\phi_\mu)(z_0)=\phi_\mu(\R(z_0)),$ implies $z_0$ is a removable singularity of  $\R_\mu\phi_\mu$ and so $\phi_\mu(z_0)$ is a branch point of $\R_\mu.$ Thus the branch points of $\R$ are transformed by $\phi_\mu$ to branch points of $\R_\mu.$ Hence $\R_\mu$  is a continuous map of $\ti\C$, which is locally analytic and thus a rational mapping of $\ti\C.$
\end{proof}
\section{The arc argument for simply connected wandering domains}\label{ch2,sec11}
Let $U$ be a simply connected component of $\F$ and suppose that $U$ is a wandering domain. Then the iterates of $U$ under $\R$ viz. 
\[
U,\R(U),\R^2(U),\ldots
\]
 are pairwise disjoint. Let 
\begin{diagram}
\R^n(U)&\rTo^{\R} &\R^{n+1}(U)\, ,\quad n>0
\end{diagram}
be injective.\\
We claim that such a simply connected wandering domain is not possible.\\
Since $U$ is simply connected, by Riemann mapping theorem, we have a conformal isomorphism $\phi : \D\to U.$ Let $\bar{a_1},\bar{b_1},\bar{c_1}\,\in\partial U$ be three distinct radial limits of $\phi$ at the points $a_1,b_1,c_1\,\in\partial\D.$ We shall make use of a real analytic family of diffeomorphisms of $\partial\D$ which fixes $a_1,b_1,c_1$ and having dimension greater than the dimension of the space of rational maps of degree $d$. (Recall that the space of rational maps of degree $d$ viz. $CP_0^{2d+1}$ has dimension $2d+1$). Let $W_0$ be an open neighbourhood of the origin in the Euclidean space of dimension greater than $4d+2.$ Let $W\subset W_0$ be a compact neighbourhood of origin. Consider a real analytic map
\begin{equation}
\partial\varphi : W_0\times\partial\D\to\partial\D.
\end{equation}
For $w\in W, \partial\varphi(w,\theta)$ is a diffeomorphism of $\partial\D$, and $\partial\phi(0,\theta)=I,$ the identity map. We suppose that for $w_1\neq w_2,\;\exists\;\,\theta$ such that $\partial\varphi(w_1,\theta)\neq\partial\varphi(w_2,\theta)$ and call this as the injective property of $W_0.$ Now with the help of $\partial\varphi$ we define
\begin{equation}\label{eq}
\varphi :  W_0\times\D\to\D\quad\text{as}
\end{equation}
$\varphi(w;(r,\theta))=(r,\partial\varphi(w,\theta)).$ As $\partial\phi(w,\theta)$ is a real analytic family of diffeomorphisms, so is $\varphi(w;.)$ for $w\in W.$ Thus  $\varphi(w;.)$ for $w\in W$ and its inverse have their derivative bounded. Hence  $\varphi(w;.),$ $w\in W$  and its inverse are uniformly \emph{Lipschitz} (using Taylor expansion of each $\varphi(w;.)$ and its inverse).Thus $\{\varphi(w;.)\,,w\in W\}$ are uniform bi\,-\,Lipschitz homeomorphism. So $\varphi(w;.)$ for $w\in W$ are $K-$ quasiconformal for some $K$ (a bi\,-\,Lipschitz homeomorphism is $K-$ quasiconformal).
On applying the Riemann map $\phi$, the conformal distortion $\tilde{\mu}$ of $\varphi(w;.)$ is transported to $U.$\\
Define a relation `$\sim$' on $U$ as
\[
x\sim y\;\;\text{if and only if}\;\;\R^n(x)=\R^m(y)\;\text{for some}\;n,m\geqslant 0.
\]
It can be easily seen that $\sim$ is an equivalence relation. Let $\tilde{U}$ denote the union of equivalence classes of points of $U$.\\
We claim that an equivalence class $[x]$ of $x\in U,$ intersects $U$ in atmost one point.\\
Suppose $y\in [x],\;\Rightarrow \R^n(x)=\R^m(y)\;\text{for some}\;n,m\geqslant 0.$ Then $\R^n(x)\in\R^n(U)$ and $\R^m(y)\in\R^m(U)$. As $U$ is a wandering domain, so $\R^n(U)\cap\R^m(U)=\emptyset,$ thus we arrive at a contradiction and hence the claim.\\
With the help of $\R$ we can transport each conformal structure $\tilde{\mu}(w)$ over $\tilde{U}$ to get a conformal structure preserved by $\R,$ so that with the help of standard structure on $\tilde{U}$, we get a measurable conformal structure $\mu(w)$ which is defined on $\ti\C,$ is invariant under $\R$ and depends real analytically on $w.$\\
By the measurable Riemann mapping theorem of Ahlfors\,-\,Bers, we get 
\begin{diagram}
W\times\ti\C &\rTo^{\eta} &\ti\C
\end{diagram}
a real analytic family of normalized (to fix $\{a_1,b_1,c_1\}$) quasiconformal homeomorphisms of $\ti\C$ so that $\varphi_w=\varphi(w,.)$ converts bounded measurable conformal structure $\mu(w)$ to the standard conformal structure.\\
As $\{\varphi_w\}$ is a real analytic family of normalized quasiconformal homeomorphism of $\ti\C$, therefore by Proposition \ref{2.10.1}, $\R_w=\varphi_w\R{\varphi_w}^{-1}$ is a real analytic family of quasiconformal deformations of $\R$, having the same degree $d$ as $\R$ .\\
Define $p : W\to C{P_0}^{2d+1}$ as $p(w)=\R_w$ . Since $\R_w$ is real analytic so is $p$ . Also $p(0)=\R_0=\varphi_0\R{\varphi_0}^{-1}=\R$ (since $\varphi_0=I$). As dim$\,W>$ dim$\,C{P_0}^{2d+1},$ therefore some fibre of $p$ has positive dimension. Let $\R_t,\, t\in [0,1],$ be a non trivial arc in the fibre of $p$ over some $g\in C{P_0}^{2d+1}.$ Then $\R_t=g\,$ i.e $\,\varphi_t\R{\varphi_t}^{-1}=g\;$ or $\;\bar{\varphi_t}\R\bar{\varphi_t}^{-1}=\R,$ where $\bar{\varphi_t}={\varphi_0}^{-1}\varphi_t$, since 
\begin{equation}
\begin{split}
\notag
{\varphi_0}^{-1}\varphi_t\R{\varphi_t}^{-1}\varphi_0
&={\varphi_0}^{-1}g\varphi_0\quad(\varphi_t\R{\varphi_t}^{-1}=g)\\
&=\R\quad(t=0,\;\Rightarrow \varphi_0\R{\varphi_0}^{-1}=g,\;\Rightarrow {\varphi_0}^{-1}g\varphi_0=\R).
\end{split}
\end{equation}
Now $\bar{\varphi_t}\R\bar{\varphi_t}^{-1}=\R$ on the Julia set $\J$ implies $\bar{\varphi_t}\R=\R\,\bar{\varphi_t}$ on $\J$ and so $\bar{\varphi_t}\in$ \emph{centralizer}$\,\R.$\\
We claim that this implies $\bar{\varphi_t}=I$ on $\J.$\\
Now $\bar{\varphi_0}={\varphi_0}^{-1}\varphi_0=I$ (since $\varphi_0=I$). From Proposition \ref{2.7.1}, there is a continuous injection of $C_\R$ into a totally disconnected Cantor group. So the centralizer of $\R$ on $\J$ cannot contain any non trivial arc. If $\bar{\varphi_t}\neq I$ for some $t,$ then there exists a connected neighbourhood of t in which $\bar{\varphi_t}\neq I$ and we will get the image of that connected neighbourhood to be an arc, which will be a contradiction to the fact that the centralizer of $\R$ on $\J$ does not contain any non trivial arc. Hence we conclude that $\bar{\varphi_t}\R=\R\,\bar{\varphi_t}$ on $\J\;\Rightarrow\bar{\varphi_t}=I$ on $\J.$\\
We further claim that $\bar{\varphi_t}$ has trivial action on the prime ends of $U.$\\
From Section \ref{ch2,sec8}, we know that the fibres of points of $\partial U$ are totally disconnected in the topology on the unit circle. Since $\bar{\varphi_t}=I$ on $\J$ and $\partial U\subset \J,$ we get $\bar{\varphi_t}$ is the identity on $\partial U$, so that $\bar{\varphi_t}$ can only permute the points in various fibres. But by Proposition \ref{2.8.17}, points in various fibres cover all the prime ends of $U,$ hence we conclude that $\bar{\varphi_t}$ acts trivially on the prime ends of $U$ which proves the claim.\\
The restricted map $\varphi_t : U\to U$ is transported by the Riemann map $\phi$ to the map $\tilde{\varphi_t}={\phi}^{-1}\phi_t\phi : \D\to\D.$ Thus $\varphi_t$ is conjugate to $\tilde{\varphi_t}.$ As $\partial\D$ corresponds to $\partial U$ under $\phi$ and $\bar{\varphi_t}= {\varphi_0}^{-1}\varphi_t$ has trivial action on the prime ends of $U,$ therefore $\,\tilde{\varphi_t}\vert_{\partial\D}=\varphi_t\vert_{\partial U}.$ But this is a contradiction to the injective property of $W,$ for $\tilde{\varphi_t}$ must agree with $\varphi_t$ having the same normalization and the same conformal distortion. So our assumption that $U$ is a simply connected wandering domain and $\R : \R^n(U)\to\R^{n+1}(U),\;\,n>0$ is injective is not possible.
\section{Sullivan's proof of Fatou's no wandering domain conjecture}\label{ch2,sec12}
We now collect together the various tools to prove Fatou's no wandering domain conjecture.
\begin{theorem}\label{2.12.1}
If $\R : \ti\C\to\ti\C$ is a rational map of degree $d$, then every component of the Fatou set $\F$ is eventually periodic.
\end{theorem}
\begin{proof}
Suppose $U_0$ is a wandering Fatou component. Then the forward images of $U_0$ under $\R$ viz. $U_1=\R(U_0), U_2=\R(U_1), \ldots $ are pairwise disjoint. From Proposition \ref{2.6.1} of \textbf{Section \ref{ch2,sec6}},  \\
Either
\begin{enumerate}
\item[(i)]  from some n onwards,  $U_{n+i}$ has finite topological type and each 
\begin{diagram}
U_{n+i} &\rTo^{\R} &U_{n+i+1} 
\end{diagram}
is an isomorphism $i\geqslant 0$, or
\item[(ii)]  the direct limit $U^{\ity} $ of 
\[
\begin{CD}
\ U_0 @>\R>> \ U_1 @>\R>> \ U_{2} @>\R>>\cdots
\end{CD}
\]
exists and has infinite topological type.
\end{enumerate}
For case (i), we have described in detail the arc argument in \textbf{Section \ref{ch2,sec11}} which leads to a contradiction in the simply connected case.\\
Suppose now that $U_0$ is finitely connected, say $\partial U_0$ has $n$ components. We have $\R :U_{n+i}\to U_{n+i+1}$ is eventually bijective and $U_{n+i}$ has finite topological type, $i\geqslant 0$. Consider $\ti\C\smallsetminus\{n\;\text{number of round disks}\}.$ We now have a Riemann map 
\[
\tilde\phi : U_0\to \ti\C\smallsetminus\{n\;\text{number of round disks}\}.
\]
This boundary of circles is described by prime ends in the region $U_0.$ The discussion and results of \textbf{Section \ref{ch2,sec8}} and the arc argument of \textbf{Section \ref{ch2,sec11}} virtually remains the same. So we again get a contradiction in the finitely connected case. So case (i) is not possible.\\
For case (ii), we make use of Proposition \ref{2.5.5} of \textbf{Section \ref{ch2,sec5}}. \\
We have that the direct limit $U^\ity$ exists and has infinite topological type. Since the direct limit of Riemann surfaces is again a Riemann surface, so $U^\ity$ is a Riemann surface of infinite topological type. Therefore the space $\mathcal C(U^\ity)$ of conformal structures of  $U^\ity$ is infinite dimensional. With the help of these structures $\mu$, by \textbf{Section \ref{ch2,sec10}}, we construct quasiconformal deformations $\R_\mu$ of $\R$ having the same degree $d$ as that of $\R$. Thus we get a continuous injection $\sigma :  C(U^\ity)\to CP_0^{2d+1}$ defined by $\sigma(\mu)=\R_\mu$. But dim$\,CP_0^{2d+1}=2d+1$ and dim$\,C(U^\ity)=\ity$, so there cannot be any injective mapping from $C(U^\ity)$ into $CP_0^{2d+1}.$ So we arrive at a contradiction and hence case (ii) is also not possible. This completes the proof of the theorem.
\end{proof}

\appendix
\chapter{Topological Viewpoint on Complex Analysis}\label{A}
This chapter deals with the standard definitions and basic results from classical complex analysis, and topological facts used throughout the dissertation. The details and proofs of the results are omitted.
\section{Basic definitions and results in complex analysis}
In this section, we introduce standard definitions of complex analysis. In addition, we briefly review some classical theorems  that has been used in this dissertation. \\
A \emph{domain} is an open connected subset of $\mathbb{C}$. A domain is said to be \emph{simply connected} if its complement is connected. An \emph{arc} in the complex plane is a continuous mapping $z=\phi(t)$ of an interval $a \leq t \leq b$ into $\mathbb{C}$. A \emph{Jordan arc} is an arc without self-intersection. A \emph{closed curve} is the continuous image of an arc whose endpoints coincide. A \emph{simple closed curve}, or a \emph{Jordan curve}, is a closed curve without self-intersections.\\
A complex valued function $f$ of a complex variable is said to be \textbf{differentiable} at a point $z_{0} \in \C$ if it has a derivative
\[f'(z_0)=\lim_{z\rightarrow z_0}\frac{f(z)-f(z_0)}{z-z_0}\]
at $z_{0}$. Such a function $f$ is \textbf{analytic} at $z_{0}$ if it is differentiable at every point in some neighborhood of $z_{0}$. It is one of the ``miracles'' of complex analysis that $f$ must then have derivatives of all orders at $z_{0}$, and that $f$ has a \emph{Taylor series} expansion
\[f(z)=\sum_{n=0}^{\infty}a_{n}(z-z_{0})^{n},\quad a_{n}=\frac{f^{n}(z_{0})}{n!},\]
convergent in some open disk centered at $z_{0}$.
The \emph{uniqueness principle} for analytic functions is a direct consequence of the Taylor series representation. It asserts that whenever two analytic functions agree on a sequence of points which cluster inside their common domain of analyticity, they must agree everywhere. Equivalently, if $f$ is analytic at a point $z_{0}$ and $f(z_{n})=0$ on a sequence of distinct points $z_{n}$ converging to $z_{0}$, then $f\equiv 0$.
\begin{theorem}
[\textbf{Schwarz Lemma}]
Let $f$ be analytic in the unit disk $\D$, with $f(0)=0$ and $|f(z)|<1$ in $\D$.\,Then $|f'(0)| \leq 1$ and $|f(z)| \leqslant |z|$ in $\D$. Strict inequality holds in both estimates unless $f$ is a rotation of the disk: $f(z)=e^{i\theta}z$.
\end{theorem}

A function $f$ analytic in a domain $D \subset\ti\C$ is said to be \emph{conformal} there if $f'(z) \neq 0$ for every $z \in D$. Such a function $f$ is conformal at a point $\zeta$ if it is conformal in some neighborhood of $\zeta$. Every M$\ddot{\text{o}}$bius transformation
\[w=f(z)=\frac{az+b}{cz+d}, \quad ad-bc \neq 0,\]
where $a$, $b$, $c$, $d$ are complex constants, provides a conformal mapping of the extended complex plane $\tilde{\mathbb{C}}$ onto itself. 

The most general conformal automorphism of the unit disk $\D$  has the form
\[f(z)=e^{i\phi}\frac{z-a}{1-\bar{a}z},\quad |a| <1.\]
With the aid of the Schwarz lemma, it can be shown that every conformal automorphism of the unit disk  has this form.

As early as 1851, Riemann enunciated that every simply connected domain can be mapped conformally onto the unit disk. This is contained in the \textbf{Riemann Mapping Theorem} stated as follows:

\emph{Let $U$ be a simply connected domain which is a proper subset of $\C$. Let $\zeta$ be a given point in $U$. Then there is a unique conformal isomorphism $\phi$ which maps $U$ onto $\D$ and has the properties $\phi(\zeta)=0$ and $\phi'(\zeta) >0$.}\\
For more details and proofs of the theorem one may refer \textbf{Ahlfors}~\cite{complex analysis}.
\subsection{Hyperbolic Metric}
From above discussion, a conformal automorphism of $\D$ has the form 
\[
w=e^{i\phi}\dfrac{(z-a)}{(1-\bar{a}z)},\; 0\leqslant\phi\leqslant 2\pi,\; |a|<1.
\]
By a simple computation one derives $|\dfrac{dw}{dz}|=\dfrac{1-|w|^2}{1-|z|^2}\cdot$
Then the metric 
\[
d\rho=\dfrac{2\,|dz|}{1-|z|^2}=\dfrac{2\,|dw|}{1-|w|^2}
\]
is invariant under the conformal automorphisms of $\D.$ This metric $d\rho$ is called the hyperbolic (or Poincar$\acute{\text{e}}$) metric on $\D.$ The hyperbolic distance $\rho(z,w)$ between two points $z, w\in\D$ is defined as 
\[
\rho(z,w)=\inf_{\gamma}\int_{\gamma}d\rho
\]
where the infimum is taken over all curves $\gamma$ joining $z, w\in\D.$ The geodesics in this metric are arcs of Euclidean circles orthogonal to $\partial\D$ and its diameters (as a special case). For a detailed discussion on hyperbolic metric one can refer \textbf{Beardon}~\cite{the geometry of discrete groups}.

The \textbf{uniformization theorem} is the most fundamental result of Riemann surface theory. This theorem is the cornerstone of the classification of Riemann surfaces. It states that :

\emph{The universal covering space of every Riemann surface is conformally equivalent to either the complex plane $\C,$ the Riemann sphere $\ti\C$ or the unit disk $\D.$}\\
As the universal covering space of any space is simply connected, so equivalently the uniformization theorem may be stated as :

\emph{Every simply connected Riemann surface is conformally equivalent to either the complex plane $\C,$ the Riemann sphere $\ti\C$ or the unit disk $\D.$}\\
The three cases are referred to as \emph{parabolic}, \emph{elliptic} and \emph{hyperbolic}.
For a proof of the uniformization theorem one may refer \textbf{Ahlfors}~\cite{conformal invariants: topics in geometric function theory}.
\section{Some topological results}
\begin{definition}
[\textbf{Topological groups}]
A topological group $G$ is both a group and a Hausdorff topological space, the two structures being \emph{compatible} in the sense that, the \emph{multiplication map} $m : G\times G\to G$ and the \emph{inversion map} $i : G\to G$ are both continuous.
\end{definition}
\begin{example}
Any abstract group with the \emph{discrete} topology is a topological group.
\end{example}
\begin{example}
The real line with \emph{additive structure} is a topological group.
\end{example}
\begin{example}
$S^1=\{z:\,|z|=1\},$ the set of complex numbers of unit modulus is a topological group since\\
$m : S^1\times S^2\to S^1$ defined as $m(e^{i\theta}, e^{i\phi})=e^{i(\theta+\phi)},$ and\\
$i : S^1\to S^1$ defined as $i(e^{i\theta})=e^{-i\theta}$ are both  continuous. This group is known as \emph{circle group}.\\
For more details one can refer \textbf{Armstrong}~\cite{armstrong}.
\end{example}
The \textbf{Euler characteristic} of a surface is a numerical invariant. Let $S$ be a compact surface and $T$ be a triangulation of $S.$ Then $T$ partitions $S$ into mutually disjoint subsets which is either a vertex, an edge or a face and are called simplices  of $T.$ If $T$ contains $F$ faces, $E$ edges and $V$ vertices, then the Euler characteristic $\chi(S)$ of $S$ equals $F-E+V.\;\chi(S)$ is a topological invariant, which is independent of the choice of triangulation. If $S_1$ and $S_2$ are subsets of $S$ each consisting of a union of simplices in $T,$ then
\[
\chi(S_1\cup S_2)=\chi(S_1)+\chi(S_2)-\chi(S_1\cap S_2).\tag{$\ast$}
\]
By constructing explicit triangulations one can see that $\chi(\ti\C)=2$ and $\chi(D)=1$ for any disk $D$ in $\ti\C.$  Now let $D=\ti\C\smallsetminus\{D_1, D_2,\ldots, D_k\}$ where 
$D_i,\, 1\leqslant i\leqslant k$ are mutually disjoint topological closed disks each being bounded by a Jordan curve, so that connectivity of $D$ is $k.$ By triangulating the sphere, each of $D, D_1,\ldots, D_k$ is a union of simplices, so that by ({$\ast$}),$    \;2=\chi(\ti\C)=\chi(D)+\sum_{i=1}^{k}\chi(D_i)=\chi(D)+k$ (as $\chi(D_i)=1\;\forall\;i$), which implies that $\chi(D)=2-k.$ Hence we conclude that if $D\subset\ti\C$ has connectivity $k,$ then $\chi(D)=2-k.$ For more details one may refer \textbf{Massey}~\cite{massey}.


\subsection{The Fundamental Group}
Here we discuss a basic method for attaching a group to each topological space, which is a topological invariant. This algebraic invariant was introduced by Henri Poincar$\acute{\text{e}},$ and is called the \textbf{Fundamental Group} or \textbf{Poincar$\acute{\text{e}}$ Group} of the space.\\
Let $X$ be a topological space. Fix a point $x_0\in X.$ Consider the collection $\mathcal C$ of all curves $\gamma : [0,1]\to X$ such that $\gamma(0)=\gamma(1)=x_0.$ Such a curve $\gamma$ is called a \emph{loop} based at $x_0.$ Two curves $\gamma_1$ and $\gamma_2$ are called \emph{homotopic} relative to $x_0,$ if there is a continuous function $H : [0,1]\times [0,1]\to X$ such that
\begin{enumerate}
\item[(i)] $H(0,t)=\gamma_1(t)\;\; \forall\;t\in [0,1];$
\item[(ii)] $H(1,t)=\gamma_2(t)\;\; \forall\;t\in [0,1];$
\item[(iii)] $H(s,0)=H(s,1)=x_0\;\; \forall\;s\in [0,1].$
\end{enumerate}
The property of being \emph{homotopic} is an equivalence relation on $\mathcal C.$ Let $\mathcal H$ be the collection of all equivalence classes.\\
We now define a binary operation `$\ast$' on $\mathcal H$ which turns it into a group.\\
Let $\gamma_1, \gamma_2$ be curves in $\mathcal C.$ Define $\gamma_1\ast\gamma_2$ 
 to be the curve $\gamma_1$ \emph{followed} by the curve $\gamma_2,$ by the parametrization
 \begin{equation}
\notag
(\gamma_1\ast\gamma_2)(t)=
\begin{cases}
\gamma(2t)    &\text{if $0\leqslant t\leqslant \frac{1}{2}$ \quad and}\\
\gamma(2t-1)    &\text{if $\frac{1}{2}\leqslant t\leqslant 1.$}
\end{cases}
\end{equation}
We note that, the particular parametrization chosen has no significance. The operation is well defined and associative on equivalence classes. The identity element is the equivalence class $[e]$ consisting of the constant loop based at $x_0.$ The inverse of the curve $\gamma$ is the curve $\gamma^{-1}$ defined by $\gamma^{-1}(t)=\gamma(1-t)\;\;\forall\;t\in [0,1].$ Both curves start and end at the same point $x_0$ and have the same image; but $\gamma^{-1}$ traverses in the reverse direction of $\gamma.$ The operation of taking inverses respects the equivalence relation and so is well defined on equivalence classes. Also $[\gamma]\ast [\gamma]^{-1}=[\gamma]^{-1}\ast [\gamma]=[e].$\\
Thus $\mathcal H,$ equipped with the binary operation $\ast,$ forms a group. This group is usually denoted by $\pi_1(X, x_0)$  and is called the fundamental group of $X$ based at $x_0.$ However if $X$ is \emph{path connected} (i.e any two points in $X$ can be joined by a path in $X$), then the fundamental group is independent of the choice of base point. Thus we speak of fundamental group of $X.$\\ 
We give examples of fundamental groups of a few topological spaces.
\begin{example}
A {convex} subset of the Euclidean space $\mathbb R^n$ has \emph{trivial} fundamental group.
\end{example}
\begin{example}
The fundamental group of $S^1$ is isomorphic to the additive group of integers.
\end{example}
\begin{example}
For $n>1,$ the fundamental group of $S^n$ is trivial.
\end{example}
\begin{example}
For $n>1,$ the fundamental group of $\mathbb RP^n$ is $\mathbb Z_2,$ where $\mathbb RP^n$ is the projective $n-$space and is obtained from the unit $n-$sphere $S^n$ by identifying the \emph{antipodal points}.
\end{example}
\begin{example}
The fundamental group of Torus, $T=S^1\times S^1$ is $\mathbb Z\times\mathbb Z.$\\
For a detailed discussion on fundamental group one may refer ~\cite{massey},~\cite{munkres}.
\end{example}
\subsection{Covering Spaces}

\begin{definition}
[\textbf{Covering map}]
Let $X$ and $\tilde X$ be topological spaces. A covering map is a continuous surjective map $p : \tilde X\to X,$ such that each point $x\in X$ has an open neighbourhood $U$ in $X$ such that $p^{-1}(U)$ is the disjoint union of open sets $\tilde{U_{\alpha}}\subset\tilde X,\;\alpha\in\Lambda$ with $p\vert_{\tilde{U_{\alpha}}} : \tilde{U_{\alpha}}\to U$ a homeomorphism. We say that $U$ is \emph{evenly covered} by $p.$ $\tilde X$ is called \emph{covering space} of $X.$ Every covering map is a local homeomorphism and therefore an open map.
\end{definition}
\begin{example}
Any homeomorphism between topological spaces is a covering map.
\end{example}
\begin{example}
The \emph{exponential map} $p :\mathbb R\to S^1$ defined as $p(x)=e^{2\pi ix}$ is a covering map.
\end{example}
\begin{definition}
[\textbf{Deck transformation}]
A deck transformation or a covering transformation of a covering map $p :\tilde X\to X$ is a homeomorphism $h : \tilde X\to\tilde X$ such that the following diagram 
\begin{diagram}
\tilde X &\rTo^{h} &\tilde X \\
&\rdTo^{p} &\dTo_{p} \\
& &X
\end{diagram}
is commutative.\\
The collection $\Gamma=\{h :\tilde X\to\tilde X\,|\;h\,\,\text{is a homeomorphism and}\,\,ph=p\}$ forms a group under composition of mappings, called the group of deck transformations of the covering map $p.$
\end{definition}
\begin{definition}
[\textbf{Regular covering map}] Let $p :\tilde X\to X$ be a covering map and let $x\in X.$ Let $F=p^{-1}(x)$ be the fibre of $x.$ $p$ is called a regular covering map if \;$\forall\; \tilde{x_1}, \tilde{x_2}\in F,\; \exists$\; a unique deck transformation $\gamma\in\Gamma$ such that $\gamma(\tilde{x_1})=\tilde{x_2},$ and in this case $\tilde{X}$ is called a regular covering space of $X.$\\
If $\tilde{X}$ is simply connected covering space of $X,$ then it is unique upto homeomorphism and it is a regular covering space of any other covering space of $X.$ Hence $\tilde{X}$ is called the \emph{universal covering space} of $X.$
\end{definition}
\begin{example}
The universal covering space of $S^1$ is the \emph{real line} $\mathbb R.$
\end{example}
\begin{example}
The universal covering space of \emph{Torus}, $T=S^1\times S^1$ is $\mathbb R^2.$
\end{example}
\begin{example}
The universal covering space of $\mathbb RP^n$ is $S^n.$ 
\end{example}
A universal covering map $p :\tilde X\to X$ is a regular covering map, and in this case the deck transformation group $\Gamma$ coincides with the fundamental group $\pi_{1}(X)$ of $X.$ Moreover $\tilde X\diagup\Gamma\cong X.$ For a detailed discussion on covering spaces one may refer ~\cite{bredon}, ~\cite{massey}.\\

\chapter{Pictures of Julia Set}\label{B : Pictures of Julia set}
More than sixty years after the fundamental work of Julia and Fatou, the field attracted new interest in the early eighties when Sullivan announced the solution of the most important problem which is the topic of this thesis and the empirical discoveries due to Mandelbrot and the beautiful dynamical color- plates. Probably these computer assisted graphics displays stimulated new interest in rational iteration. 

Here we are displaying some graphics of iterations of functions (Julia sets) mentioned in the thesis.
\backmatter

\end{document}